\pdfoutput=1
\PassOptionsToPackage{dvipsnames}{xcolor}

\documentclass[aop]{imsart}

\usepackage[T1]{fontenc} 
\usepackage[utf8]{inputenc}

\usepackage[colorlinks,citecolor=blue,urlcolor=blue]{hyperref}%

\RequirePackage{amsthm,amsmath,amsfonts,amssymb}
\RequirePackage[numbers]{natbib}

\usepackage[backgroundcolor=yellow]{todonotes}
\usepackage{mathdots}
\usepackage{glossaries}
\usepackage{subfiles}
\usepackage{bbm}

\usepackage{commath}
\usepackage{nicefrac}
\usepackage{float}

\usepackage{xspace}
\usepackage{enumitem}
\usepackage[config, labelfont={normalsize}]{caption,subfig}

\usepackage{xcolor}

\usepackage{tikz}
	\usetikzlibrary{arrows.meta}
	\usetikzlibrary{colorbrewer}
	
	\usetikzlibrary{patterns}
	\usetikzlibrary{matrix,arrows,decorations.pathmorphing}
	\usetikzlibrary{shapes.geometric}
	\usetikzlibrary{decorations.markings}
	\usetikzlibrary{calc}
	
	\tikzstyle{very densely dashed}=[dash pattern=on 5pt off 1.5pt]
	\tikzstyle{very very densely dashed}=[dash pattern=on 8pt off 1.5pt]
	\tikzstyle{middle densely dashed}=[dash pattern=on 4pt off 1.5pt]
	\tikzstyle{space densely dashed}=[dash pattern=on 0pt off 0.4pt on 10pt off 0pt]

\usepackage[capitalize,noabbrev]{cleveref}
\usepackage{aliascnt}   %

\usepackage{graphicx} 
\usepackage{datatool}
\usepackage{catchfile}
\usepackage{multirow}

\startlocaldefs
\numberwithin{equation}{section}

\theoremstyle{plain}
\newtheorem{thm}{Theorem}[section]

\newaliascnt{theorem}{thm}
\newtheorem{theorem}[theorem]{Theorem}
\aliascntresetthe{theorem}

\newaliascnt{prop}{thm}

\aliascntresetthe{prop}

\newaliascnt{lemma}{thm}
\newtheorem{lemma}[lemma]{Lemma}
\aliascntresetthe{lemma}

\newaliascnt{problem}{thm}
\newtheorem{problem}[problem]{Problem}
\aliascntresetthe{problem}

\newaliascnt{corollary}{thm}
\newtheorem{corollary}[corollary]{Corollary}
\aliascntresetthe{corollary}

\newaliascnt{proposition}{thm}
\newtheorem{proposition}[proposition]{Proposition}
\aliascntresetthe{proposition}

\newaliascnt{conjecture}{thm}
\newtheorem{conjecture}[conjecture]{Conjecture}
\aliascntresetthe{conjecture}

\crefname{thm}{Theorem}{Theorems}
\Crefname{thm}{Theorem}{Theorems}
\crefname{theorem}{Theorem}{Theorems}
\Crefname{theorem}{Theorem}{Theorems}
\crefname{prop}{Proposition}{Propositions}
\Crefname{prop}{Proposition}{Propositions}
\crefname{lemma}{Lemma}{Lemmas}
\Crefname{lemma}{Lemma}{Lemmas}
\crefname{problem}{Problem}{Problems}
\Crefname{problem}{Problem}{Problems}
\crefname{corollary}{Corollary}{Corollaries}
\Crefname{corollary}{Corollary}{Corollaries}
\crefname{proposition}{Proposition}{Propositions}
\Crefname{proposition}{Proposition}{Propositions}
\crefname{conjecture}{Conjecture}{Conjectures}
\Crefname{conjecture}{Conjecture}{Conjectures}

\theoremstyle{remark}
\newtheorem{remark}[thm]{Remark}

\newcommand\loaddata[1]{\CatchFileDef\loadeddata{#1}{\endlinechar=-1}}

\makeatletter
\newcommand{\@giventhatstar}[2]{\left[#1\;\middle|\;#2\right]}
\newcommand{\@giventhatnostar}[3][]{#1(#2\;#1|\;#3#1)}
\newcommand{\giventhat}{\@ifstar\@giventhatstar\@giventhatnostar}
\makeatother

\newcommand{\cauchy}{\mathbf{G}}

\newcommand{\E}{\mathbb{E}}
\newcommand{\R}{\mathbb{R}}
\newcommand{\N}{\mathbb{N}}
\newcommand{\Z}{\mathbb{Z}}
\newcommand{\CC}{\mathbb{C}}

\newcommand{\Pro}{\mathbb{P}}
\newcommand{\doublecumulative}{\mathcal{F}}
\newcommand{\indicator}{\mathbbm{1}}

\newcommand{\nklatki}{n}

\newcommand{\Force}{\mathcal{E}}
\newcommand{\ForceRegularized}[2]{\Force_{#1;\ #2}}

\newcommand{\Sy}[1]{\mathfrak{S}_{#1}}
\DeclareMathOperator{\Ins}{Ins}
\DeclareMathOperator{\uIns}{u-Ins}
\DeclareMathOperator{\Var}{Var}
\DeclareMathOperator{\SC}{SC}
\DeclareMathOperator{\AS}{AS}
\DeclareMathOperator{\RSK}{RSK}

\newcommand{\kerx}{\mathbbm{x}}
\newcommand{\kery}{\mathbbm{y}}
\newcommand{\kerell}{\mathbb{L}}

\newcommand{\contour}{\mathcal{C}}

\newcommand{\sniady}{the second named author\xspace}

\newcommand{\assum}[1]{Assumption~\ref{#1}}
\newcommand{\Assum}[1]{Assumption~\ref{#1}}

\newcommand{\assumptions}{Assumptions\xspace}
\newcommand{\Assumptions}{Assumptions\xspace}

\newcommand{\iid}{i.i.d.\xspace}
\newcommand{\CDF}{CDF\xspace}
\newcommand{\CDFs}{CDFs\xspace}

\newcommand{\conti}{continual\xspace}
\newcommand{\Conti}{Continual\xspace}

\newcommand{\zmienna}{\mathbf{Q}}
\newcommand{\stala}{q}
\newcommand{\sigmafield}{\mathfrak{F}_n}
\newcommand{\uwithnoise}{\mathbf{u}_n}

\endlocaldefs

\begin{document}

\begin{frontmatter}
\title{Fluctuations  of Schensted row insertion}
\runtitle{Fluctuations  of Schensted row insertion}

\begin{aug}
\author[A]{\fnms{Mikołaj}~\snm{Marciniak}\ead[label=e1]{marciniak@int.pl}\orcid{0000-0003-0358-5887}},
\author[B]{\fnms{Piotr}~\snm{Śniady}\ead[label=e2]{psniady@impan.pl}\orcid{0000-0002-1356-4820}}
\address[A]{Interdisciplinary Doctoral School ``Academia Copernicana'',
    Faculty of Mathematics and Computer Science,
    \mbox{Nicolaus Copernicus University in Toru{\'n}}, 
    ul.~Fryderyka Chopina 12/18, 87-100 Toru{\'n}, Poland\printead[presep={,\ }]{e1}}

\address[B]{Institute of Mathematics, Polish Academy of Sciences,
    \mbox{ul.~\'Sniadec\-kich 8,} 00-656 Warszawa, Poland\printead[presep={,\ }]{e2}}
\end{aug}

\begin{abstract}	
We investigate asymptotic probabilistic phenomena arising from the application
of the Schensted row insertion algorithm, a key component of the
Robinson--Schensted--Knuth (RSK) correspondence, to random inputs. Our analysis
centers on a random tableau $T$ with a given shape $\lambda$, which may itself
be random or deterministic. We examine the stochastic properties of the position
of the new box created when inserting a deterministic entry into $T$.
Specifically, we focus on the fluctuations of this position around its expected
value as the size of the Young diagram $\lambda$ approaches infinity. Our
findings reveal that these fluctuations are asymptotically Gaussian, with the
mean and variance expressed in terms of Kerov's transition measure of the
diagram $\lambda$.
An important application of this analysis is the RSK algorithm applied to a
finite, long sequence of independent, identically distributed random variables.
While there remains a gap in the reasoning for this case, we present an explicit
conjecture regarding its behavior.
\end{abstract}



\end{frontmatter}
\tableofcontents

\section{Teaser: new conjectures related to RSK algorithm  applied to random input}
\label{sec:intro}

This paper is quite extensive; to engage the reader, we begin with a teaser: two
new conjectures (\cref{conj:CLT-Plancherel} and \cref{conj:jdt}) concerning the
Robinson--Schensted--Knuth (RSK) algorithm applied to random input.

\subsection{Basic definitions}

We begin by reviewing some fundamental combinatorial concepts. For a more
comprehensive treatment of this topic, we refer the reader to the book of Fulton
\cite{Fulton1997}.

\subsubsection{Young diagrams and tableaux}
\label{sec:young-diagrams}

A Young diagram is a finite collection of boxes arranged in the positive
quadrant, aligned to the left and bottom edges. This arrangement is known as the
French convention (see \cref{subfig:french}). To each Young diagram with $\ell$
rows, we associate an integer partition $\lambda = (\lambda_1, \ldots,
\lambda_\ell)$, where $\lambda_j$ denotes the number of boxes in the $j$-th row,
counting from bottom to top. We identify a Young diagram with its corresponding
partition $\lambda$ and denote the total number of boxes by $|\lambda| =
\lambda_1 + \cdots + \lambda_\ell$.

For asymptotic problems, it is convenient to draw Young diagrams using the
\emph{Russian convention} (see \cref{subfig:russian}). This corresponds to the
coordinate system $(u,v)$, which relates to the usual French Cartesian
coordinates as follows:
\begin{equation}
	\label{eq:Russian}
	u = x - y, \qquad \qquad v = x + y.
\end{equation}

\medskip

\begin{figure}[t]
		\subfloat[]{
		\begin{tikzpicture}
			
			\begin{scope}[scale=0.5/sqrt(2),rotate=-45,draw=gray]
				
				\begin{scope}[draw=gray,rotate=45,scale=sqrt(2)]
					\fill[fill=red!10] (4,0) -- (4,1) -- (3,1) -- (3,2) -- (1,2)
					-- (1,3) -- (0,3) -- (0,0) -- cycle ;
				\end{scope}
				
				\begin{scope}[rotate=45,draw=blue,scale=sqrt(2)]
					\draw[ultra thin, loosely dashed] (0,0) grid (5,4);
				\end{scope}
				
				\draw[->,thin] (-4.5,0) -- (4.5,0)
				node[anchor=west,rotate=-45]{\textcolor{gray}{$u$}};
				
				\draw[->,thin] (0,-0.4) -- (0,9.5)
				node[anchor=south,rotate=-45]{\textcolor{gray}{$v$}};
				
				\begin{scope}[draw=blue,rotate=45,scale=sqrt(2)]
					
					\draw[->,thick] (0,0) -- (6,0) node[anchor=west]{\textcolor{blue}{$x$}};
					\foreach \x in {1, 2, 3, 4, 5}
					{ \draw (\x, -2pt) node[anchor=north] {\textcolor{blue}{\tiny{$\x$}}} -- (\x,
						2pt); }
					
					\draw[->,thick] (0,0) -- (0,5) node[anchor=south] {\textcolor{blue}{$y$}};
					\foreach \y in {1, 2, 3, 4}
					{ \draw (-2pt,\y) node[anchor=east] {\textcolor{blue}{\tiny{$\y$}}} -- (2pt,\y); }
					
					\draw[ultra thick,draw=red] (5.5,0) -- (4,0) -- (4,1) -- (3,1) --
					(3,2) -- (1,2) -- (1,3) -- (0,3) -- (0,4.5) ;
					
				\end{scope}
				
			\end{scope}
		\end{tikzpicture}
		\label{subfig:french}
	}
	\hfill
	\subfloat[]
	{
		\begin{tikzpicture}
			\begin{scope}[xshift=7cm, yshift=-0.5cm, scale=0.5]
				
				\begin{scope}[draw=gray,rotate=45,scale=sqrt(2)]
					\fill[fill=red!10] (4,0) -- (4,1) -- (3,1) -- (3,2) -- (1,2)
					-- (1,3) -- (0,3) -- (0,0) -- cycle ;
				\end{scope}
				
				\begin{scope}
					\clip (-4.5,0) rectangle (5.5,5.5);
					\draw[ultra thin, loosely dashed] (-6,0.01) grid (6,6);
				\end{scope}
				
				\draw[->,thick] (-6,0) -- (6,0) node[anchor=west]{$u$};
				\foreach \z in {-5, -4, -3, -2, -1, 1, 2, 3, 4, 5}
				{ \draw (\z, -2pt) node[anchor=north] {\tiny{$\z$}} -- (\z, 2pt); }
				
				\draw[->,thick] (0,-0.4) -- (0,6) node[anchor=south]{$v$};
				\foreach \t in {1, 2, 3, 4, 5}
				{ \draw (-2pt,\t) node[anchor=east] {\tiny{$\t$}} -- (2pt,\t); }

				\begin{scope}[draw=blue!50,rotate=45,scale=sqrt(2)]
					
					\draw[->,thin] (0,0) -- (6,0) node[anchor=west,rotate=45]
					{\textcolor{blue!50}{{$x$}}};
					
					\draw[->,thin] (0,0) -- (0,5) node[anchor=south,rotate=45]
					{\textcolor{blue!50}{{$y$}}};
					
					\draw[ultra thick,draw=red] (5.5,0) -- (4,0) -- (4,1) -- (3,1) --
					(3,2) -- (1,2) -- (1,3) -- (0,3) -- (0,4.5) ;
					
				\end{scope}
			\end{scope}
			
		\end{tikzpicture}
		\label{subfig:russian}
	}
	\caption
	{   The Young diagram $(4,3,1)$ is depicted in two conventions:
		\protect\subref{subfig:french} the French convention,
		\protect\subref{subfig:russian} the Russian convention.
		In both representations, the solid red line illustrates the diagram's profile. 
		The coordinate systems are as follows:
		$(x,y)$ for the French convention, and
		$(u,v)$ for the Russian convention.
	}
	\label{fig:french}
\end{figure}
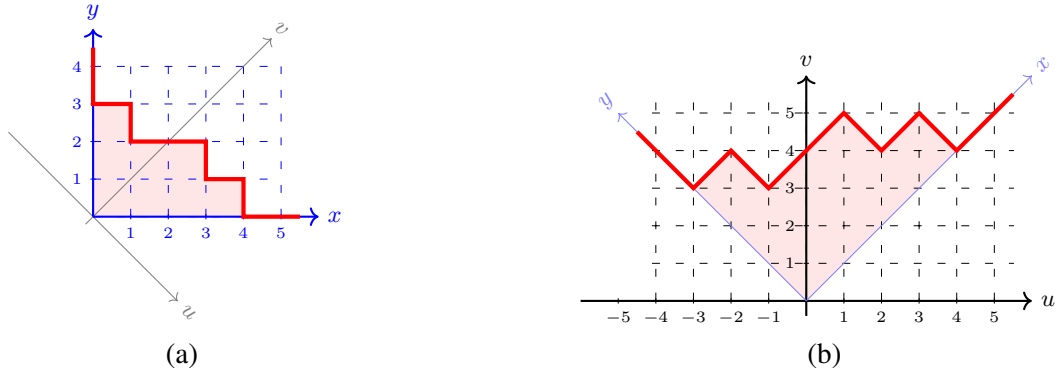

A \emph{tableau} (also known as \emph{semi-standard tableau}) is a filling of
the boxes of a Young diagram with numbers; we require that the entries should be
weakly increasing in each row (from left to right) and strictly increasing in
each column (from bottom to top). An example is given in \cref{fig:RSKa}. We say
that a tableau~$T$ of shape $\lambda$ is a \emph{standard Young tableau} if it
contains only entries from the set  $\{ 1,2,\dots,|\lambda|\}$ and each element
is used exactly once.

\subsubsection{The Schensted row insertion}

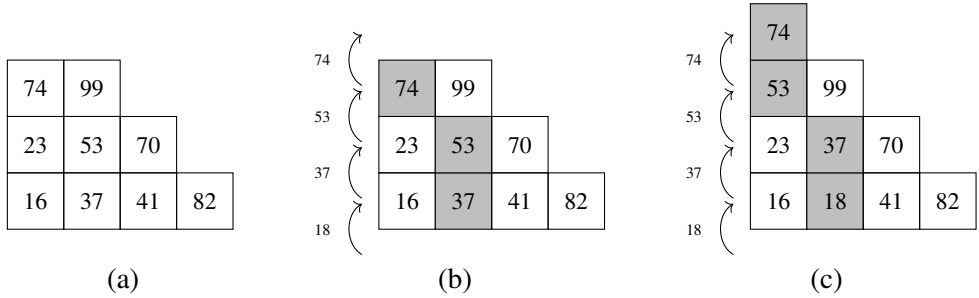
\begin{figure}[t]
		\centering
	\hfill
	\subfloat[]
	{
		\begin{tikzpicture}[scale=0.75]
			\clip (-0.1,-0.5) rectangle (4.1,4.5);
			\draw (0,0) rectangle +(1,1); 
			\node at (0.5,0.5) {16}; 
			\draw (1,0) rectangle +(1,1); 
			\node at (1.5,0.5) {37}; 
			\draw (2,0) rectangle +(1,1); 
			\node at (2.5,0.5) {41}; 
			\draw (3,0) rectangle +(1,1); 
			\node at (3.5,0.5) {82}; 
			\draw (0,1) rectangle +(1,1); 
			\node at (0.5,1.5) {23}; 
			\draw (1,1) rectangle +(1,1); 
			\node at (1.5,1.5) {53}; 
			\draw (2,1) rectangle +(1,1); 
			\node at (2.5,1.5) {70}; 
			\draw (0,2) rectangle +(1,1); 
			\node at (0.5,2.5) {74}; 
			\draw (1,2) rectangle +(1,1); 
			\node at (1.5,2.5) {99}; 	
		\end{tikzpicture}
		\label{fig:RSKa}
	}
	\hfill
	\subfloat[]
	{
		\begin{tikzpicture}[scale=0.75]
			\clip (-1.5,-0.5) rectangle (4.1,4.5);
			\fill[blue!10] (1,0) rectangle +(1,1);
			\fill[blue!10] (1,1) rectangle +(1,1);
			\fill[blue!10] (0,2) rectangle +(1,1);
			\draw (0,0) rectangle +(1,1); 
			\node at (0.5,0.5) {16}; 
			\draw[fill=lightgray] (1,0) rectangle +(1,1); 
			\node at (1.5,0.5) {37}; 
			\draw (2,0) rectangle +(1,1); 
			\node at (2.5,0.5) {41}; 
			\draw (3,0) rectangle +(1,1); 
			\node at (3.5,0.5) {82}; 
			\draw (0,1) rectangle +(1,1); 
			\node at (0.5,1.5) {23}; 
			\draw[fill=lightgray] (1,1) rectangle +(1,1); 
			\node at (1.5,1.5) {53}; 
			\draw (2,1) rectangle +(1,1); 
			\node at (2.5,1.5) {70}; 
			\draw[fill=lightgray] (0,2) rectangle +(1,1); 
			\node at (0.5,2.5) {74}; 
			\draw (1,2) rectangle +(1,1); 
			\node at (1.5,2.5) {99}; 
			\draw[->] (-0.3,-0.45) to[bend left=60] (-0.3,0.45);
			\draw[->] (0,1) +(-0.3,-0.45) to[bend left=60] +(-0.3,0.45);
			\draw[->] (0,2) +(-0.3,-0.45) to[bend left=60] +(-0.3,0.45);
			\draw[->] (0,3) +(-0.3,-0.45) to[bend left=60] +(-0.3,0.45);
			\tiny
			\node[] at (-1,0) {18};
			\node[] at (-1,1) {37};
			\node[] at (-1,2) {53};
			\node[] at (-1,3) {74};
		\end{tikzpicture}
		\label{fig:RSKb}
	}
	\hfill
	\subfloat[]
	{
		\begin{tikzpicture}[scale=0.75]
			\clip (-1.5,-0.5) rectangle (4.1,4.5);
			\fill[blue!10] (1,0) rectangle +(1,1);
			\fill[blue!10] (1,1) rectangle +(1,1);
			\fill[blue!10] (0,2) rectangle +(1,1);
			\fill[blue!10] (0,3) rectangle +(1,1);
			\draw (0,0) rectangle +(1,1); 
			\node at (0.5,0.5) {16}; 
			\draw[fill=lightgray] (1,0) rectangle +(1,1); 
			\node at (1.5,0.5) {18}; 
			\draw (2,0) rectangle +(1,1); 
			\node at (2.5,0.5) {41}; 
			\draw (3,0) rectangle +(1,1); 
			\node at (3.5,0.5) {82}; 
			\draw (0,1) rectangle +(1,1); 
			\node at (0.5,1.5) {23}; 
			\draw[fill=lightgray] (1,1) rectangle +(1,1); 
			\node at (1.5,1.5) {37}; 
			\draw (2,1) rectangle +(1,1); 
			\node at (2.5,1.5) {70}; 
			\draw[fill=lightgray] (0,2) rectangle +(1,1); 
			\node at (0.5,2.5) {53}; 
			\draw (1,2) rectangle +(1,1); 
			\node at (1.5,2.5) {99}; 
			\draw[fill=lightgray] (0,3) rectangle +(1,1); 
			\node at (0.5,3.5) {74}; 
			\draw[->] (-0.3,-0.45) to[bend left=60] (-0.3,0.45);
			\draw[->] (0,1) +(-0.3,-0.45) to[bend left=60] +(-0.3,0.45);
			\draw[->] (0,2) +(-0.3,-0.45) to[bend left=60] +(-0.3,0.45);
			\draw[->] (0,3) +(-0.3,-0.45) to[bend left=60] +(-0.3,0.45);
			\tiny
			\node[] at (-1,0) {18};
			\node[] at (-1,1) {37};
			\node[] at (-1,2) {53};
			\node[] at (-1,3) {74};
		\end{tikzpicture}
		\label{fig:RSKc}
	}
	\hfill
	
	\caption{\protect\subref{fig:RSKa}~The original tableau $T$.
		\protect\subref{fig:RSKb}~The highlighted boxes indicate the bumping
		route for the Schensted insertion $T \leftarrow 18$. The numbers next to
		the arrows represent the bumped entries.
		\protect\subref{fig:RSKc}~The resulting tableau after the Schensted
		insertion $T \leftarrow 18$.}

	\label{fig:RSK}
\end{figure}

The \emph{Schensted row insertion} is an algorithm that takes a tableau $T$ and
a number $z$ as input. The process begins by inserting $z$ into the first (bottom)
row of $T$, following these rules:
\begin{itemize}
	\item
	$z$ is placed in the leftmost box containing an entry strictly larger than $z$.
	
	\item
	If no such box exists, $z$ is appended to the end of the row in a new box, 
	and the algorithm terminates.
	
	\item
	If $z$ displaces an existing entry $z'$, this $z'$ is \emph{``bumped''} to the second row.
	
	\item
	The process repeats with $z'$ being inserted into the second row, following the same rules.
	
	\item
	This continues until a number is inserted into an empty box.
\end{itemize}
The resulting tableau is denoted as $T\leftarrow z$, see
\cref{fig:RSKb,fig:RSKc} for an example. The sequence of boxes whose contents
change during this process is called \emph{the bumping route}.

Schensted insertion is a key component of the Robinson--Schensted--Knuth
algorithm (RSK), see below.

\subsubsection{The Robinson--Schensted--Knuth algorithm}
\label{sec:RSK-def}

This article considers a simplified version of the Robinson--Schensted--Knuth
algorithm (RSK), which is more accurately described as the Robinson--Schensted
algorithm. However, we retain the RSK acronym due to its widespread recognition.
The RSK algorithm maps a finite sequence $w=(w_1,\dots,w_{\nklatki})$ to a pair
of tableaux: the insertion tableau $P(w)$ and the recording tableau $Q(w)$.

The insertion tableau is defined as:
\begin{equation}
	\label{eq:insertion}
	P(w) = \Big( \big( (\emptyset \leftarrow w_1) \leftarrow
	w_2 \big) \leftarrow \cdots \Big) \leftarrow w_{\nklatki}.
\end{equation}
This represents the result of iteratively applying Schensted insertion to the
entries of $w$, beginning with an empty tableau $\emptyset$.

The recording tableau $Q(w)$ is a standard Young tableau with the same shape as
$P(w)$. Each entry in $Q(w)$ corresponds to the iteration number in
\eqref{eq:insertion} when that box was first filled. In other words, $Q(w)$
records the order in which the entries of the insertion tableau were populated.
Both $P(w)$ and $Q(w)$ share a common shape, denoted as $\RSK(w)$, which we
refer to as \emph{the RSK shape associated with $w$}.

\smallskip

The RSK algorithm is a
fundamental tool in algebraic combinatorics and representation theory,
particularly in relation to Littlewood--Richardson coefficients (see
\cite{Fulton1997,Stanley1999}).

\subsection{Context and motivations: RSK applied to random input}

A fruitful area of research is investigating the Robinson--Schensted--Knuth
algorithm applied to random input. We shall review some selected highlights of
this field.

\subsubsection{Plancherel measure}
\label{sec:plancherel}

The simplest example involves applying RSK to a finite sequence 
$w=(w_1,\dots,w_{\nklatki})$ of
\iid random variables uniformly distributed on the unit interval~$[0,1]$.
The resulting probability distribution of $\RSK(w)$ is the celebrated
\emph{Plancherel measure $\operatorname{Pl}_{\nklatki}$} on Young diagrams with
$\nklatki$ boxes. This measure arises naturally in the context of decomposing
the left regular representation of the symmetric group $\Sy{\nklatki}$ into
irreducible components. For a Young diagram $\lambda$ with
$\nklatki$ boxes, the Plancherel measure is defined as:
\[ \operatorname{Pl}_{\nklatki}(\lambda) = \frac{(f^\lambda)^2}{\nklatki !} \]
where $f^\lambda$ denotes the number of standard Young tableaux of shape
$\lambda$, see \cite[Section~1.8]{Romik2015}

A remarkable result related to the probability distribution of $\RSK(w)$ is the
solution to \emph{the Ulam--Hammersley problem}. This solution
\cite{BaikDeift1999,Okounkov2000} reveals a surprising connection with \emph{the
	Tracy--Widom distribution} \cite{TracyWidom}, which originates in random matrix
theory. The Tracy--Widom distribution describes the asymptotic behavior of the
largest eigenvalue in certain random matrix ensembles and has found applications
in various fields, including growth models.

For a comprehensive introduction to these topics, we recommend the book by Romik
\cite{Romik2015}.

\subsubsection{Extremal characters of $\Sy{\infty}$}
\label{sec:extremal}

Consider an infinite sequence of \iid random variables $w_1, w_2, \dots$ (possibly with a
more complex probability distribution including atoms). Let
\begin{equation}
	\label{eq:plancherel-growth-via-RSK}
	\lambda^{(\nklatki)} = \RSK(w_1, \dots, w_{\nklatki})
\end{equation}
denote the RSK shape corresponding to the first $\nklatki$ variables. The
resulting sequence of Young diagrams
\begin{equation}
	\label{eq:walk-young}
	\emptyset = \lambda^{(0)} \nearrow \lambda^{(1)} \nearrow \cdots 
\end{equation}
forms a Markov random walk on the Young graph, where vertices are Young diagrams
and edges connect diagrams differing by one box. This random walk
\eqref{eq:walk-young} has some additional convenient properties, which are out
of scope of the current paper.

Vershik and Kerov \cite{VershikKerov1981a,KerovVershik1986} demonstrated that
classifying such random walks on the Young graph with these additional
properties (a problem intersecting probability theory, harmonic analysis on the
Young graph, and ergodic theory) is equivalent to identifying \emph{the extremal
characters of the infinite symmetric group $\Sy{\infty}$}. This approach led to a
novel, conceptual proof of Thoma's classification of such characters
\cite{Thoma1964}. As a consequence, there exists a bijection between the
extremal characters of $\Sy{\infty}$ and the probability laws of the random
variables $(w_{\nklatki})$. The RSK algorithm provides an explicit method for
generating the corresponding random walk \eqref{eq:walk-young}.

\subsubsection{The key motivation: RSK as an isomorphism of dynamical systems}

Consider the recording tableau for an infinite sequence $w_1, w_2, \dots$: \[
Q_\infty(w_1,w_2,\dots) := \lim_{n\to\infty} Q(w_1,\dots,w_n). \] This infinite
standard tableau fills a subset of the upper-right quarterplane, where each
natural number appears exactly once, with rows and columns increasing. Each
entry in $Q_\infty(w_1,w_2,\dots)$ represents the iteration at which the
corresponding box became non-empty during the infinite sequence of row
insertions: \[  \big( (\emptyset \leftarrow w_1) \leftarrow w_2 \big) \leftarrow
\cdots. \] 
Thus, $Q_\infty$ encodes the infinite path \eqref{eq:walk-young} in
the Young graph.

The connection between RSK and the Plancherel measure can be formulated as
follows: $Q_\infty$ is a homomorphism between two probability spaces:
\begin{itemize}
\item  the product space $[0,1]^\infty$ with the product Lebesgue measure
(representing a sequence of \iid uniform random variables on $[0,1]$), and

\item  the set of infinite standard Young tableaux with the Plancherel measure
(fundamental in the harmonic analysis of $\Sy{\infty}$).
\end{itemize}
Each space admits a natural measure-preserving transformation: the one-sided
shift and \emph{the jeu de taquin transformation}, respectively. This leads to our
central question:
\begin{problem}[Key motivation] \label{problem:iso} Is $Q_\infty$ an isomorphism
	of measure-preserving dynamical systems? If so, how can we construct its
	inverse? 
\end{problem}

The answer is non-trivial because finite RSK requires both the recording tableau
$Q(w)$ and the insertion tableau $P(w)$ to recover $w$, but $P(w)$ is
unavailable in the infinite case. An affirmative answer would illuminate aspects
of harmonic analysis on $\Sy{\infty}$, such as the ergodicity of jeu de taquin.

Romik and \sniady \cite{RomikSniady2015} investigated this problem, and we
present their findings below. Spoiler alert: \cref{problem:iso} has a positive
resolution.

\subsection{The main problem: position of the new box}
\label{sec:definition-of-uIns}

\subsubsection{Notations} 

For a finite sequence $w=(w_1,\dots,w_{\nklatki+1})$ we denote by
\[ \Ins(w_1,\dots,w_{\nklatki};\ w_{\nklatki+1})=(x_{\nklatki},y_{\nklatki})\] 
the coordinates of the last box which was inserted to the Young diagram by the
RSK algorithm applied to the sequence $w$. In other words, it is the box
containing the biggest number in the recording tableau $Q(w)$. Above,
$(x_{\nklatki},y_{\nklatki})$ refer to the Cartesian coordinates of this box in
the French convention, i.e., $x_{\nklatki}$ is the number of the column and
$y_{\nklatki}$ is the number of the row. By
\[ \uIns (w_1,\dots,w_{\nklatki};\ w_{\nklatki+1}) =
x_{\nklatki}- y_{\nklatki} \]
we denote the $u$-coordinate of the aforementioned box.

For later use, given a tableau $T$ and a real number $z$ we denote by $\Ins(T;
z)$ the coordinates of the new box which was created by the Schensted row
insertion $T\leftarrow z$; in other words it is the unique box of the skew
diagram
\[ \operatorname{shape} \big( T \leftarrow z \big) / \operatorname{shape} T. \]
The quantity $\uIns(T; z)$ is defined in an analogous way as the $u$-coordinate
of $\Ins(T; z)$.

\subsubsection{The main problem: relationship between the value of the new entry
	and the position of the new box}

In this paper, we focus on a fundamental case where $w_1, w_2, \ldots$ is a
sequence of \iid random variables, each uniformly distributed on the unit
interval $[0,1]$. Romik and \sniady \cite{RomikSniady2015} observed that
constructing the inverse map to $Q_\infty$ (and consequently, providing a
positive answer to \cref{problem:iso}) requires addressing the following
question about the finite version of RSK applied to such a random input.

\begin{problem}[The main problem]
	\label{pro:what-is-relationship} 
	
	Let $w = (w_1, w_2, \ldots)$ be a sequence of \iid random variables, each
uniformly distributed on $[0,1]$. We seek to understand the relationship
between:
	\begin{itemize}
		\item The value of the new entry $w_{\nklatki+1}$, and 
		\item The position of the corresponding newly created box:
		\begin{equation}
			\label{eq:my-beloved-rainbow-box}
			(x_\nklatki, y_\nklatki) = \Ins(w_1, \ldots, w_\nklatki; w_{\nklatki+1}).
		\end{equation}
	\end{itemize}
	Our interest lies in the asymptotic behavior of this probabilistic relationship
	as $\nklatki \to \infty$.
\end{problem}

\begin{figure}
    \centering 
    \subfloat[]{
    \begin{tikzpicture}[scale=0.8,rotate=45]
        \foreach \x\y\c in {0 / 0 / 1,
            1 / 0 / 2,
            2 / 0 / 3,
            3 / 0 / 4,
            0 / 1 / 5,
            1 / 1 / 7,
            2 / 1 / 9,
            0 / 2 / 6,
            1 / 2 / 10,
            0 / 3 / 8,
            1 / 3 / 11
        }{
            \draw (\x,\y)  rectangle +(1,1)  +(0.5,0.5) node {$\c$}    ;    }
        \draw[blue,thick,->] (0,0) -- (0,4.7) node[anchor=south]{\textcolor{blue}{$y$}};
        \draw[blue,thick,->] (0,0) -- (4.7,0) node[anchor=south]{\textcolor{blue}{$x$}};
    \end{tikzpicture}
    \label{subfig:historyA}
}
    \hfill
      \subfloat[]{
    \begin{tikzpicture}[scale=0.8,rotate=45]
        \foreach \x\y\c\kolor in 
        {0 / 0 / 14 /1,
            1 / 0 / 59 /3,
            2 / 0 / 75 /3,
            3 / 0 / 91 /4,
            0 / 1 / 58 /3,
            1 / 1 / 48 /2,
            2 / 1 / 86 /4,
            0 / 2 / 41 /2,
            1 / 2 / 79 /4,
            0 / 3 / 7 /1,
            1 / 3 / 63 /3}{
            \fill[fill=PuBu-5-\kolor](\x,\y)  rectangle +(1,1)  +(0.5,0.5) node {$\c$}    ;    }
                \draw[blue,thick,->] (0,0) -- (0,4.7) node[anchor=south]{\textcolor{blue}{$y$}};
        \draw[blue,thick,->] (0,0) -- (4.7,0) node[anchor=south]{\textcolor{blue}{$x$}};
    \end{tikzpicture}
    \label{subfig:historyB}
} 

	\caption{ \protect\subref{subfig:historyA} Recording tableau $Q_{xy}=Q(w)$
	(depicted using Russian convention) corresponding to the sequence $w=(14, 59,
	75, 91, 58, 41, 48, 7, 86, 79, 63)$ selected from the interval $J=[0,100]$.
		\protect\subref{subfig:historyB} Responsibility matrix
	$\left(w_{Q_{xy}}\right)$ derived by substituting each entry of the recording
	tableau with its corresponding value from sequence $w$. The background employs
	a four-color layer tinting scheme, representing different quarters of the
	interval $J$: $[0,25]$ (near-white), $(25,50]$ (beige), $(50,75]$ (blue), and
	$(75,100]$ (dark blue). } 
	
	\label{fig:history}
\end{figure}
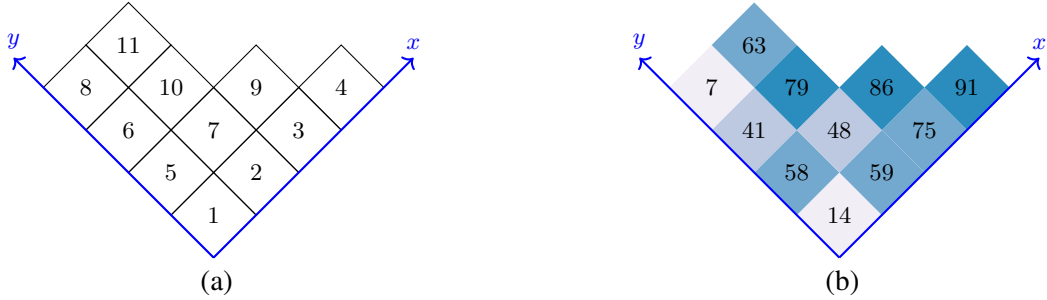

\begin{figure}
\subfile{figures/FIGURE-history-big.tex} \caption{ Analogue of
	\cref{subfig:historyB} for a sequence $w=(w_1,\dots,w_{1000})$ comprising
	$\nklatki=1000$ \iid random variables following
	the uniform distribution $U(0,1)$ on the unit interval $I=[0,1]$. The layer
	tinting represents the values of the responsibility matrix $\left( w_{Q_{xy}}
	\right)$, with four color-coded intervals: $\big[0,\nicefrac{1}{4}\big]$
	(near-white), $\big(\nicefrac{1}{4},\nicefrac{1}{2}\big]$ (beige),
	$\big(\nicefrac{1}{2},\nicefrac{3}{4}\big]$ (blue), and
	$\big(\nicefrac{3}{4},1\big]$ (dark blue). The red solid line depicts the
	Logan--Shepp--Vershik--Kerov limit curve $\Omega_*$. Five black dots indicate
	its natural parametrization, dividing the area between the curve and the $Oxy$
	axes into four curvilinear triangles of equal areas. } 
	\label{fig:history-big}
\end{figure}

The RSK algorithm's dynamics can be visualized by replacing each entry $(x,y)$
in the recording tableau $Q(w) = (Q_{xy})$ with the corresponding element from
the input sequence $w = (w_1, \ldots, w_n)$ that led to its creation. The
resulting matrix $(w_{Q_{xy}})$ is termed the \emph{responsibility matrix} (see
\cref{subfig:historyB}). For improved readability in \cref{fig:history}, we
sample integers from $[0,100]$ instead of using real numbers.

Following Pittel and Romik's approach \cite[Section 1.1]{PittelRomik2007}, we
can represent the responsibility matrix geometrically as a three-dimensional
array of cuboids above the plane $\R^2 \times \{0\}$. Each cuboid's
height corresponds to $w_{Q_{xy}}$, with a unit square $[x-1,x] \times [y-1,y]
\times \{0\}$ as its base. Alternatively, we can interpret the function $(x,y)
\mapsto w_{Q_{xy}}$ as the graph of a (discontinuous) surface forming the upper
envelope of this array.

By rescaling the base unit squares to side length $\frac{1}{\sqrt{n}}$ (where
$n$ is the length of sequence~$w$), we normalize the total base area to 1 (see
\cref{fig:history-big}). In both \cref{subfig:historyB} and
\cref{fig:history-big}, we use \emph{layer tinting} to indicate elevation, i.e.,
the values of the responsibility matrix.

Monte Carlo simulations, such as the one depicted in \cref{fig:history-big},
suggest that the value of the new entry $w_{n+1}$ in
\cref{pro:what-is-relationship} approximately determines the ray (a half-line
originating from the coordinate system's origin) on which the new box
$\Ins(w_1,\ldots,w_n;\ w_{n+1})$ will appear. This behavior is observed with
high probability as $n \to \infty$. Smaller values of the new entry ($w_{n+1}
\approx 0$) tend to correspond to rays closer to the $Oy$ axis (the nearly white
area in \cref{fig:history-big}), while larger values ($w_{n+1} \approx 1$)
typically align with rays nearer to the $Ox$ axis (the dark blue area in
\cref{fig:history-big}).

\subsection{The limit shape and its parametrization} 

To address \cref{pro:what-is-relationship}, we begin with a crucial observation:
the newly created box $\Ins(w_1,\dots,w_{\nklatki};\ w_{\nklatki+1})$ must be
situated in one of the concave corners of the Young diagram
$\lambda^{(\nklatki)}=\RSK(w_1,\dots,w_{\nklatki})$. Consequently, understanding
the asymptotic behavior of the random Young diagram $\lambda^{(\nklatki)}$
becomes paramount for our analysis. Fortunately, the probability distribution of
such an RSK shape is well-understood: it follows the Plancherel measure on Young
diagrams with $\nklatki$ boxes. The limit shape for this measure has been
extensively studied and characterized. We will review this limit shape in the
following discussion. For a pedagogical introduction to this topic we refer to
\cite[Chapter~1]{Romik2015}.

\subsubsection{Scaling of Young diagrams}
\label{sec:scaling}

The boundary of a Young diagram $\lambda$ is called its \emph{profile}, see
\cref{subfig:french}. In the Russian coordinate system the profile can be seen
as the plot of the function $\omega_\lambda\colon \R\to\R_+$, see
\cref{subfig:russian}.

If $c>0$ is a positive number, the output $c \lambda\subset \R_+^2$ of a
homogeneous dilation with scale $c$ applied to the Young diagram $\lambda$,
viewed as a subset of the quarterplane,  might no longer be a Young diagram.
Nevertheless, its profile is still well defined as:
\begin{equation}
    \label{eq:rescaled-profile}
 \omega_{c\lambda} (u) = c\ \omega_{\lambda}\left( \frac{u}{c} \right) \qquad \text{ for $u\in\R$}.
\end{equation}

\subsection{The Logan--Shepp--Vershik--Kerov limit shape}
\label{sec:lsvk}

Logan and Shepp \cite{LoganShepp1977} as well as Vershik and Kerov
\cite{VershikKerov1977} independently proved a law of large numbers for the
shapes of Young diagrams: as the number of boxes $n \to \infty$ tends to infinity, the scaled-down
profile of a Plancherel-distributed random Young diagram converges in
probability to an explicit limit curve (see \cref{fig:history-big} for an
illustration). The limit shape is given by the function $\Omega_*: \R \to [0,\infty)$ defined as:
\begin{equation} 
	\label{eq:omegastar}
	\Omega_*(u) = 
	\begin{cases} 
		\frac{2}{\pi}\left[ u \arcsin\left(\frac{u}{2}\right) + \sqrt{4-u^2} \right] 
		        & \text{if } -2 \leq u \leq 2, \\
		|u| & \text{otherwise.}
	\end{cases}
\end{equation}

\begin{theorem}[Limit shape of Plancherel-distributed Young diagrams \cite{LoganShepp1977,VershikKerov1977}] 
	\label{thm:plancherel-limitshape}
	
	Let $\lambda^{(n)}$ be a random Young diagram distributed according to the
Plancherel measure on Young diagrams with $n$ boxes. Then we have convergence
in the supremum norm in probability:
	\[
	\sup_{u \in \R}
	      \left|\omega_{\frac{1}{\sqrt{n}}\lambda^{(n)}}(u) - \Omega_*(u) \right| 
	                                          \xrightarrow[n \to \infty]{\Pro} 0.
	\]
\end{theorem}
The rate of convergence is quite fast, and precise results about the magnitude
of local fluctuations have been established \cite{BogachevSu2007}.

\subsection{Position of the new box for an unspecified new entry}
\label{sec:semicircle}

In the context of \cref{pro:what-is-relationship}, let us momentarily
assume that we do not know the value of the new entry $w_{\nklatki+1}$. The
complete answer to \cref{pro:what-is-relationship} would then be the
probability distribution of the new box~\eqref{eq:my-beloved-rainbow-box}. This
distribution coincides with that of the box containing the entry $\nklatki+1$ in
a large Plancherel-distributed random standard Young tableau. Pittel and
Romik \cite[Section~1.2]{PittelRomik2007} studied an analogous problem for a
different class of random tableaux.

As expected, the probability distribution of this vector (after rescaling, and
in Russian coordinates)
\[
(u_{\nklatki}, v_{\nklatki}) = 
   \left(\frac{x_{\nklatki} - y_{\nklatki}}{\sqrt{\nklatki}}, 
          \frac{x_{\nklatki} + y_{\nklatki}}{\sqrt{\nklatki}}\right)
\]
converges, as $\nklatki \to \infty$, to a certain probability measure $\mu$
supported on the limit curve $\Omega_*$. To uniquely specify this measure, it
suffices to find the limit distribution for the $u$-coordinates, i.e., for the
random variables $(u_{\nklatki})$.

It turns out that the sequence $(u_{\nklatki})$ converges in distribution to the
\emph{semicircle measure} on the interval $[-2,2]$ (see \cite{Kerov1993-transition} and
\cite[Theorem 3.2]{RomikSniady2015}). The density of this measure is given by
\begin{equation}
	\label{eq:SCdensity}
	f_{\SC}(u) = \frac{1}{2\pi} \sqrt{4 - u^2} \quad \text{for } u \in [-2,2].
\end{equation}
We denote by $F_{\SC}: [-2,2] \to [0,1]$ the cumulative distribution function (\CDF) of
this semicircle law, given by
\[
F_{\SC}(u) = \frac{1}{2\pi} \int_{-2}^u \sqrt{4 - w^2} \, \dif w =
\frac{1}{2} + \frac{u \sqrt{4 - u^2}}{4 \pi} + \frac{\arcsin\left(\frac{u}{2}\right)}{\pi}
\quad \text{for } u \in [-2,2].
\]

\subsection{The natural parametrization of the limit curve}

The following two \emph{RSK-tri\-go\-no\-met\-ric functions} are defined for $z \in [0,1]$:
\begin{align*}
	\operatorname{RSK cos} z &= F^{-1}_{\SC}(z), \\
	\operatorname{RSK sin} z &= \Omega_*\left(F^{-1}_{\SC}(z)\right).
\end{align*}

The map
\begin{equation}
	\label{eq:parametrization}
	[0,1] \ni z \mapsto \left(\operatorname{RSK cos} z, \operatorname{RSK sin} z \right) \in \R^2
\end{equation}
is a two-dimensional analogue of the \emph{quantile function} (the inverse of
the \CDF) in the context of the limit measure $\mu$.
This map provides a convenient parametrization of the limit curve $\Omega_*$, as
illustrated by the five black dots in \cref{fig:history-big}. It is a
natural analogue of the parametrization of the unit circle by the angle
\[ z \mapsto (\cos z, \sin z) \]
as well as the parametrization of the hyperbola by the hyperbolic angle
\[ z \mapsto (\cosh z, \sinh z), \]
since in all three cases the area of the curvilinear triangle between the ray,
the curve, and the $y$-axis (or $x$-axis) is proportional to the value of the
parameter $z$.

\subsection{The first main conjecture: the fluctuations of Schensted row insertion}

To address \cref{pro:what-is-relationship}, it is most convenient to condition
on the random variable $w_{n+1}$ taking a fixed value $z \in [0,1]$. The
following result by Romik and \sniady provides a first-order asymptotic
description.

\begin{theorem}[Asymptotic determinism of RSK insertion {\cite[Theorem 5.1]{RomikSniady2015}}]
	\label{thm:Romik-Sniady-determinism}
	Let $z \in [0,1]$ be fixed. We denote by
	\begin{equation}
		\label{eq:point-on-limit-curve}
		(u_0, v_0) = (\operatorname{RSK cos} z, \operatorname{RSK sin} z)
	\end{equation}
	the corresponding point on the limit curve $\Omega_*$ via \eqref{eq:parametrization}.
	
	Let $w_1, w_2, \ldots$ be a sequence of \iid random variables with the uniform
distribution $U(0,1)$ on the unit interval. We denote by
	\[ (x_n, y_n) = \Ins(w_1, \ldots, w_n; z) \]
	the position of the box created by the insertion of $z$.
	
	Then we have the following convergence in probability:
	\begin{equation}
		\label{eq:ins-rus}
		\left(\frac{x_n - y_n}{\sqrt{n}}, \frac{x_n + y_n}{\sqrt{n}}\right) 
		     \xrightarrow[n \to \infty]{P} (u_0, v_0).
	\end{equation}
\end{theorem}

This relationship between the value of the new entry being inserted and the
location of the corresponding new box was visualized in \cref{fig:history-big}
by the layer tinting. \cref{thm:Romik-Sniady-determinism} was illustrated by a
Monte Carlo simulation in \cref{fig:new-box-100}. One of the results of
the current paper is a generalization of this result in
\cref{thm:determinism-new}.

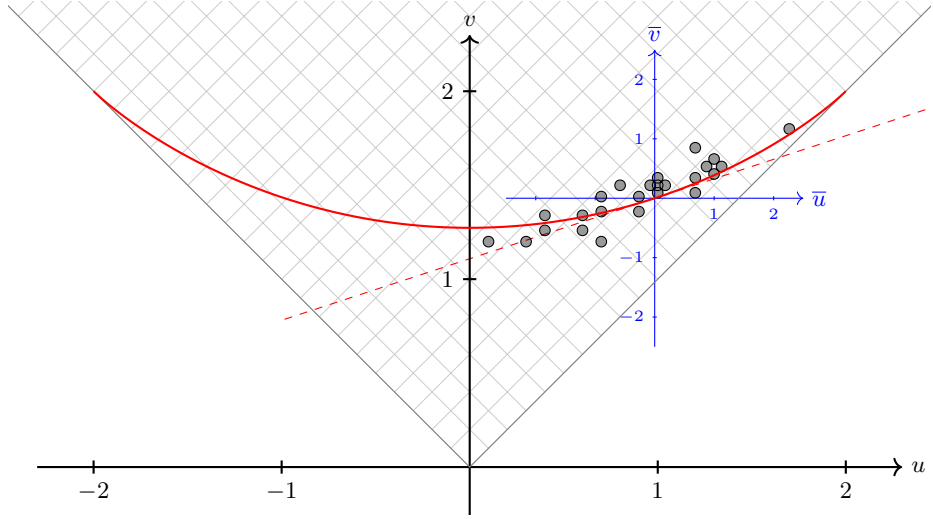
\begin{figure}
	
	\centering
	\begin{tikzpicture}[scale=2.5]
		\clip(-2.45,-0.25) rectangle (2.45,2.48);

		\begin{scope}[rotate=45,scale=sqrt(2/100)]
			\draw[ultra thin, black!20] (0,0) grid (71,71);
		\end{scope}
		
		\draw[red,dashed] (0.9837236655274301, 1.43061427786138)  +(2, 0.654733957003398) -- +(-2, -0.654733957003398);

		\begin{scope}[shift={(0.9837236655274301, 1.43061427786138)},scale=1/sqrt(sqrt(100))]
			\draw[blue,->] (-2.5,0) -- (2.5,0) node[anchor=west] {$\overline{u}$};
			\draw[blue,->] (0,-2.5) -- (0,2.5) node[anchor=south] {$\overline{v}$};
			\foreach \x in {1,2}  \draw[blue] (\x,1pt) -- (\x,-1pt) node[anchor=north] {\tiny$\x$};
			\foreach \x in {-1,-2}  \draw[blue] (\x,1pt) -- (\x,-1pt);
			\foreach \y in {-2,-1,1,2} %
			\draw[blue] (1pt,\y) -- (-1pt,\y) node[anchor=east] {\tiny$\y$}; 
			
		\end{scope}
		
		\begin{scope}[rotate=45,scale=sqrt(2/100)]

			\foreach \position in { (9, 2), (17, 0), (11, 3), (7, 4), (14, 2), (6, 5) } 
			\draw[fill=black,fill opacity=0.4] \position +(0.5,0.5) circle (0.2);
			\foreach \position in {  (13, 1), (9, 3), (10, 3), (11, 2), (8, 4) } 
			{ \draw[fill=black,fill opacity=0.4] \position +(0.3,0.3) circle (0.2) +(0.7,0.7) circle (0.2); }
			\foreach \position in {  (14, 1) } 
			{ \draw[fill=black,fill opacity=0.4] \position +(0.3,0.3) circle (0.2) 
				+(0.3,0.7) circle (0.2) 
				+(0.7,0.3) circle (0.2) 
				+(0.7,0.7) circle (0.2); }
			\foreach \position in {  (12, 2) } 
			{ \draw[fill=black,fill opacity=0.4] \position +(0.3,0.3) circle (0.2) 
				+(0.3,0.7) circle (0.2) 
				+(0.7,0.3) circle (0.2) 
				+(0.7,0.7) circle (0.2)
				+(0.5,0.5) circle (0.2);
			}
		\end{scope}
		
		\draw[thick,->] (-2.3,0)-- (2.3,0) node[anchor=west] {$u$};
		\draw[thick,->] (0,-0.5) -- (0, 2.3) node[anchor=south] {$v$};
		
		\draw[black!50] (-2.5,2.5) -- (0,0) -- (2.5,2.5);
		\foreach \x in {-2,-1,1,2} %
		\draw[thick] (\x,1pt) -- (\x,-1pt) node[anchor=north] {$\x$};
		\foreach \y in {1,2} %
		\draw[thick] (1pt,\y) -- (-1pt,\y) node[anchor=east] {$\y$}; 
		
		\draw[red,thick] plot[smooth] file {figures/data/vershik-kerov-RU.txt};
		
	\end{tikzpicture} 

\caption{ Grey circles represent simulated (rescaled) positions of the new box
	$\Ins(w_1,\dots,w_\nklatki; z)$ for an initial sequence $w_1,\dots,w_\nklatki$
	of \iid $U(0,1)$ random variables with length $\nklatki=100$ and new entry
	$z=0.8$. The $(u,v)$ coordinates of the circles correspond to the quantities in
	the law of large numbers (\cref{thm:Romik-Sniady-determinism}). The grid
	indicates the actual size of the boxes. The solid red curve depicts the
	Logan--Shepp--Vershik--Kerov limit shape, with the red dashed line tangent to
	this curve at the point corresponding to $z$ in the natural parametrization. The
	coordinates of the circles in the blue $(\overline{u}, \overline{v})$ coordinate
	system correspond to the quantities in \cref{conj:CLT-Plancherel}.}

\label{fig:new-box-100}
\end{figure} 

\medskip

We conjecture that the ultimate answer to \cref{pro:what-is-relationship} is
provided by the following description of the fluctuations of the new box.
\begin{conjecture}[The first main conjecture]
\label{conj:CLT-Plancherel}
We keep the notation from \cref{thm:Romik-Sniady-determinism}.

Then the scaled sequence of random points
\begin{multline} 
\label{eq:my-pony-rainbow-vector}
\sqrt[4]{\nklatki}  \left[  \left( \frac{x_{\nklatki} -y_{\nklatki}}{\sqrt{\nklatki}}, 
        \frac{x_{\nklatki}+y_{\nklatki} }{\sqrt{\nklatki}} \right) - (u_0,v_0)  \right]   =  \\
\left( 
\frac{x_{\nklatki}-y_{\nklatki}}{\sqrt[4]{\nklatki}}- \sqrt[4]{\nklatki}\ u_0, \;\;
\frac{x_{\nklatki}+y_{\nklatki}}{\sqrt[4]{\nklatki}}- \sqrt[4]{\nklatki}\ v_0
\right) \in\R^2
\end{multline}
converges in distribution to a centered, degenerate Gaussian distribution on the
plane which is supported on the line
\begin{equation}\label{eq:degenerate-line}
\big\{ (u,v):  v = \Omega_*'(u_0)\ u \big\} 
\end{equation}
that is parallel to the tangent line to the limit curve $\Omega_*$ in $u_0$.

This Gaussian distribution is uniquely determined by the limit measure for the
$u$-coordinates
\[ 
\sqrt[4]{\nklatki}  \left[   \frac{x_{\nklatki} -y_{\nklatki}}{\sqrt{\nklatki}} - u_0  \right]  
\xrightarrow[\nklatki\to\infty]{\mathcal{D}} N\left( 0, \sigma^2_{u_0} \right).
\]
which is the centered Gaussian distribution with the variance
\begin{equation}
\label{eq:varianceCLT}
\sigma^2_{u_0}= \frac{\pi}{3} \sqrt{4- u_0^2 }.
\end{equation}
\end{conjecture}

\begin{figure}
	
	\centering
	\begin{tikzpicture}[scale=15]
		\clip(0.65,1.3) rectangle (1.5,1.72);

		\draw[red,dashed] (0.9837236655274301, 1.43061427786138)  +(2, 0.654733957003398) -- +(-2, -0.654733957003398);

		\begin{scope}[shift={(0.9837236655274301, 1.43061427786138)},scale=1/sqrt(sqrt(10000))]
			\draw[blue,->] (-3.5,0) -- (4.5,0)  node[anchor=west] {$\overline{u}$};
			\draw[blue,->] (0,-2.5) -- (0,2.5)  node[anchor=south] {$\overline{v}$};
			\foreach \x in {1,2,3,4}  \draw[blue] (\x,1pt) -- (\x,-1pt) node[anchor=north] {$\x$};
			\foreach \x in {-1,-2,-3}  \draw[blue] (\x,1pt) -- (\x,-1pt) node[anchor=south] {$\x$};
			\foreach \y in {-2,-1,1,2} %
			\draw[blue] (1pt,\y) -- (-1pt,\y) node[anchor=east] {$\y$};            
		\end{scope}

		\begin{scope}[rotate=45,scale=sqrt(2/10000)]

			\foreach \position in { (116, 23), (110, 26), (137, 14), (130, 18), (128, 18), (110, 27), (115, 25), (109, 29), (101, 33), (121, 22), (130, 17), (122, 21), (147, 9), (137, 15), (126, 20), (117, 23), (116, 24), (114, 26), (124, 20), (112, 26), (123, 21), (108, 29), (134, 16) } 
			\draw[fill=black,fill opacity=0.4] \position +(0.5,0.5) circle (0.4);
			\foreach \position in {  (129, 19) } 
			{ \draw[fill=black,fill opacity=0.4] \position +(0.3,0.3) circle (0.4) +(0.7,0.7) circle (0.4); }
		\end{scope}

		\draw[red,thick] plot[smooth] file {figures/data/vershik-kerov-RU.txt};
		
	\end{tikzpicture}  

\caption{Analogue of \cref{fig:new-box-100} for sequence length $\nklatki=10^4$
	and $z=0.8$. The image is zoomed to focus on the $(\overline{u}, \overline{v})$
	coordinate system. For improved visibility, the grid showing the actual size of
	the boxes has been omitted.}

\label{fig:clt-big}
\end{figure}
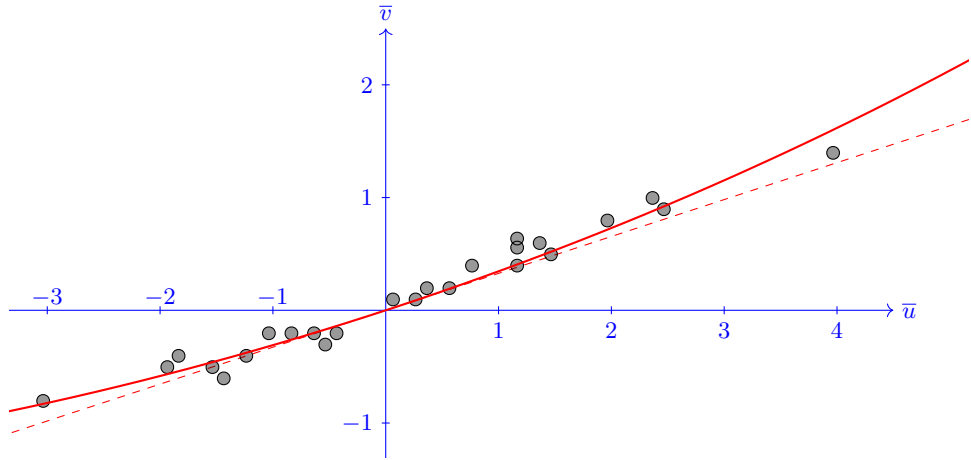

This conjecture is visualized by Monte Carlo simulations in
\cref{fig:new-box-100,fig:clt-big}. The blue coordinate
system~$(\overline{u},\overline{v})$
\[
(\overline{u},\overline{v})=
\sqrt[4]{\nklatki}  \left[  \left( \frac{u}{\sqrt{\nklatki}}, 
          \frac{v}{\sqrt{\nklatki}} \right) - (u_0,v_0)  \right]   \]
corresponds to the quantities which appear in the $\nklatki$-th random point
\eqref{eq:my-pony-rainbow-vector}.

\medskip

The formula for the variance \eqref{eq:varianceCLT} may raise questions about
its conceptual interpretation. This motivates further reading, as
\cref{conj:CLT-Plancherel} appears to be a special case of the more general
\cref{thm:CLT-shape}, though establishing a rigorous reduction would require
additional arguments.

\begin{remark}
	It would be valuable to compare \cref{conj:CLT-Plancherel} with the Monte Carlo
experiments conducted by Vassiliev, Duzhin, and Kuzhin
\cite{VassilievDuzhin2019}. However, it is important to note that some of their
results appear to be based on corrupted computer data. For instance, the plots
in \cite[Figure~8]{VassilievDuzhin2019} lack the expected axial symmetry with
respect to the $W=0$ axis and the central symmetry around $W=0$ and
$\mu=\frac{1}{2}$. These symmetries should be present due to the inherent
properties of the RSK correspondence. Specifically, replacing the numbers $z,
w_1, w_2, \ldots$ with $1-z, 1-w_1, 1-w_2, \ldots$ (using the notation from
\cref{conj:CLT-Plancherel}) should result in the transposition of the
corresponding RSK diagram's shape. This property is a well-established feature
of the RSK algorithm and should be reflected in any accurate simulation
results.
\end{remark}

\subsection{Donsker's functional central limit theorem for jeu de taquin trajectories}

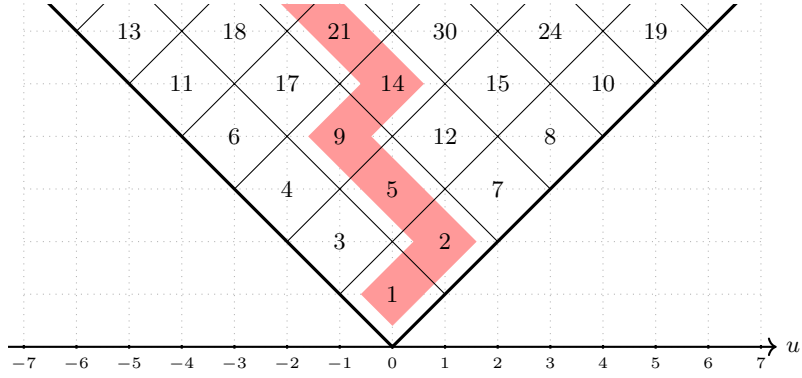
\begin{figure}[t]
    \centering
    \centering    
\begin{tikzpicture}[scale=0.7]
    \draw[black!30,dotted] (-7,0) grid (7,6.6); 
    \draw[thick,->] (-7.3,0)-- (7.3,0) node[anchor=west] {$u$};
    \foreach \x in {-7,-6,-5,-4,-3,-2,-1,0,1,2,3,4,5,6,7} %
    \draw[thick] (\x,1pt) -- (\x,-1pt) node[anchor=north] {\tiny $\x$};

    \clip (7,6.5) -- (-7,6.5) -- (-7,-0.1) -- (7,-0.1) ; 
    
    \draw[draw=none,fill=red!40]
(-9.4,19.0) -- (-8.4,18.0) -- (-7.4,17.0) -- (-8.4,16.0) -- (-7.4,15.0) --
(-8.4,14.0) -- (-7.4,13.0) -- (-6.4,12.0) -- (-5.4,11.0) -- (-4.4,10.0) --
(-3.4,9.0) -- (-2.4,8.0) -- (-1.4,7.0) -- (-0.4,6.0) -- (0.6,5.0) --
(-0.4,4.0) -- (0.6,3.0) -- (1.6,2.0) -- (0.6,1.0) -- (0.0,0.4) --
(-0.6,1.0) -- (0.4,2.0) -- (-0.6,3.0) -- (-1.6,4.0) -- (-0.6,5.0) --
(-1.6,6.0) -- (-2.6,7.0) -- (-3.6,8.0) -- (-4.6,9.0) -- (-5.6,10.0) --
(-6.6,11.0) -- (-7.6,12.0) -- (-8.6,13.0) -- (-9.6,14.0) -- (-8.6,15.0) --
(-9.6,16.0) -- (-8.6,17.0) -- (-9.6,18.0) -- (-10.6,19.0) -- (-10.0,19.6)
-- cycle;
    
 \begin{scope}[rotate=45,scale=sqrt(2)]
        \draw (0,0) grid (10,10);
 \end{scope}   

    \draw[very thick](12,12) -- (0,0) -- (-12,12);

\draw (0.0,1.0) node {$1$};
\draw (1.0,2.0) node {$2$};
\draw (2.0,3.0) node {$7$};
\draw (3.0,4.0) node {$8$};
\draw (4.0,5.0) node {$10$};
\draw (5.0,6.0) node {$19$};
\draw (6.0,7.0) node {$22$};
\draw (-1.0,2.0) node {$3$};
\draw (0.0,3.0) node {$5$};
\draw (2.0,5.0) node {$15$};
\draw (1.0,4.0) node {$12$};
\draw (3.0,6.0) node {$24$};
\draw (-2.0,3.0) node {$4$};
\draw (-1.0,4.0) node {$9$};
\draw (0.0,5.0) node {$14$};
\draw (1.0,6.0) node {$30$};
\draw (2.0,7.0) node {$34$};
\draw (-3.0,4.0) node {$6$};
\draw (-2.0,5.0) node {$17$};
\draw (-1.0,6.0) node {$21$};
\draw (0.0,7.0) node {$42$};
\draw (-4.0,5.0) node {$11$};
\draw (-3.0,6.0) node {$18$};
\draw (-2.0,7.0) node {$29$};
\draw (-5.0,6.0) node {$13$};
\draw (-4.0,7.0) node {$20$};
\draw (-6.0,7.0) node {$16$};

\end{tikzpicture}
    
    \caption{A part of an infinite standard Young tableau drawn in the Russian
	coordinate system. The highlighted boxes form the beginning of the jeu de
	taquin path.}
    
    \label{fig:jdt-path}
\end{figure}

Given an infinite standard Young tableau $T$, we view it in the Russian
coordinate system (see \cref{fig:jdt-path}). We define an infinite lattice
path, referred to as the \emph{jeu de taquin path} of tableau $T$, as follows.
Starting from the corner box of the tableau, we travel in unit steps northwest
or northeast. At each step, we choose the direction among the two in which the
entry in the tableau is smaller. This process continues indefinitely, creating
the jeu de taquin path. 
For an integer $n \geq 1$, we denote by $j_n = j_n(T) =
(u_n, v_n) \in \Z^2$ the last box in the jeu de taquin path of $T$ which
contains a number $\leq n$. We refer to the sequence $(j_n)$ as the \emph{jeu de
	taquin path in the lazy parametrization} (in the Russian coordinate system).
\cref{fig:jdt-path} illustrates these concepts. 

We conjecture that the following analogue of Donsker's functional central limit
theorem holds true for jeu de taquin paths. This claim is a refinement of a
result of Romik and \sniady \cite[Theorem~5.2]{RomikSniady2015}.

\begin{conjecture}[The second main conjecture]
	\label{conj:jdt} 
	
	Let $z \in [0,1]$ be fixed and let $(u_0,v_0)$ given by
	\eqref{eq:point-on-limit-curve} be the corresponding point on the limit curve.
	Let $w_1,w_2,\dots$ be a sequence of \iid random variables with the uniform
	distribution $U(0,1)$ on the unit interval. Let
	\[T = Q_\infty(1-z,w_1,w_2,\dots)\] 
	be the random infinite recording
	tableau and let $j_1,j_2,\dots$ be the jeu de taquin path in the
	tableau $T$ in the lazy parametrization.
	
	As $c \to \infty$, the random function
	\begin{equation}
		\label{eq:donsker}
	\R_+ \ni t \mapsto 
	\sqrt[4]{c}
	\left[ \frac{1}{\sqrt{c}} j_{\lceil c t^2 \rceil} - (t u_0, t v_0)  \right] \in \R^2 
	\end{equation}
	converges in distribution to the two-dimensional Brownian motion $(U_t,V_t)$
	that is supported on the line \eqref{eq:degenerate-line} and for which the
	covariance of the $u$-coordinates is given by
	\[ \E U_{t} U_{s} = \min(t,s) \sigma_{u_0}^2, \]
	where the constant $\sigma_{u_0}$ is given by \eqref{eq:varianceCLT}.
\end{conjecture}

A special case of \cref{conj:CLT-Plancherel} for $z=\frac{1}{2}$ was conjectured
by Wojtyniak \cite{Wojtyniak2019} (with a slightly different form of the
covariance) based on extensive Monte Carlo simulations.

\cref{conj:jdt} would imply \cref{conj:CLT-Plancherel}; indeed, one of the
consequences of \cref{conj:jdt} is the information about the limit distribution
of the random function \eqref{eq:donsker} evaluated at the fixed time $t=1$.
Fairly standard methods allow to relate this probability distribution to the one
from \cref{conj:CLT-Plancherel}, see \cite[Section~5.2]{RomikSniady2015}.

\begin{figure}[tb]
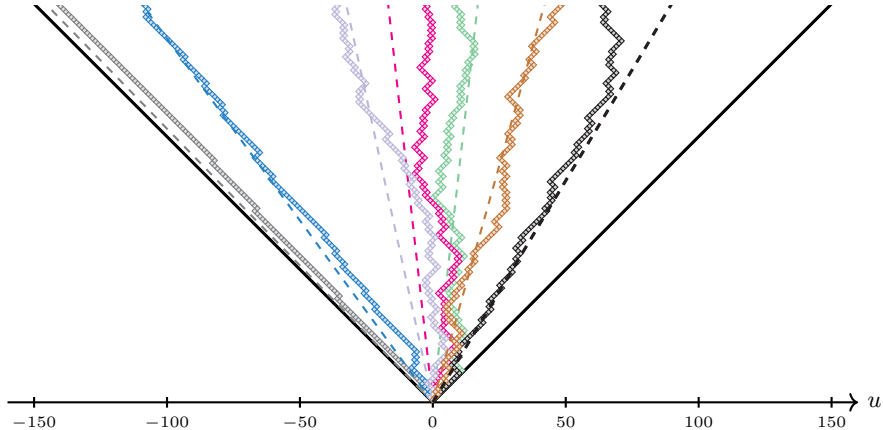

    \subfile{figures/FIGURE-simulated-trajectories.tex} 
    
    \caption{Several simulated paths of jeu de taquin and (dashed lines) their
    asymptotes. Figure excerpted from \cite{RomikSniady2015}.}
\label{fig:theta}
\end{figure}

\subsection{Second class particles}

With the notations of \cref{conj:jdt}, if $z$ is assumed to be random with the
uniform distribution $U(0,1)$, then the recording tableau $T$ is a
Plancherel-distributed infinite standard Young tableau. Thanks to Rost's map
\cite{Rost1981}, this can be seen as the \emph{Plancherel-TASEP interacting
	particle system} \cite[Section~7]{RomikSniady2015}. In this context, the
$u$-coordinate $u_n$ of the jeu de taquin box $j_n$ can be interpreted as the
position of \emph{the second class particle} in this interacting particle system
at time $n$. Thus, part of \cref{conj:jdt} can be reformulated as follows.

\begin{conjecture}
	Consider the Plancherel-TASEP interacting particle system and let $u_n$ be the
	position of the second class particle at time $n$. There exists a random
	variable $V$ with the semicircle distribution $\mu_{\SC}$ on the interval
	$[-2,2]$ such that the random function
	\[
	\R_+ \ni t \mapsto \frac{\sqrt[4]{c}}{\sigma_V} 
	\left[\frac{u_{\lfloor c t^2 \rfloor}}{\sqrt{c}} - V t \right]
	\]
	converges in distribution to the standard Brownian motion $B(t)$ as $c \to \infty$.
\end{conjecture}

We suspect that in the above conjecture $V = F_{\SC}^{-1}(z)$ can be defined as
the appropriate quantile of the semicircle distribution.

Heuristically, this conjecture suggests that asymptotically the second class
particles follow a drift with the random velocity $V$ plus a (rescaled)
Brownian motion, so that the approximate equality
\[
u_{\lfloor t^2 \rfloor} \approx  V t + \sigma_V B(t)
\]
holds true for $t \to \infty$. In particular, the fluctuations are of order
$t^{\frac{1}{2}}$. It would be interesting to explore the links of this
conjecture with analogous results available for the usual TASEP model and the
competition interface \cite[Section~2.3]{FerrariMartin2009}, \cite{RahmanVirag2021}. 
It \emph{seems} that the corresponding fluctuations are superdiffusive, of order
$t^{\frac{2}{3}}$; however, we are not aware of a definitive, rigorous treatment
of this topic in the literature.

\section{Main results and structure of the paper}

\subsection{Poissonized tableaux: a generalization}

\subsubsection{Plancherel tableau, the key example}
\label{sec:plancherel-tableau}

In \cref{sec:intro}, we introduced \cref{thm:Romik-Sniady-determinism} and
\cref{conj:CLT-Plancherel}, which focused on the Schensted row insertion $T
\leftarrow z$ for a specific random tableau $T = P(w_1, \dots, w_n)$. This
tableau, generated by applying the RSK algorithm to a sequence of \iid random
variables with the uniform $U(0,1)$ distribution, is referred to as a
\emph{Plancherel-distributed random Poissonized tableau with $n$ boxes}. While
significant, this is just one example of a broader class known as
\emph{Poissonized tableaux of a given shape}, which we will explore in detail.

\subsubsection{Poissonized tableaux of a given shape}
\label{sec:poissonized}

Poissonized tableaux \cite{GorinRahman2019} are a generalization of standard
Young tableaux where entries are real numbers from the unit interval $[0,1]$. We
denote the set of Poissonized tableaux with shape $\lambda$ as
$\mathcal{T}^{\lambda}$.

Each Poissonized tableau in $\mathcal{T}^{\lambda}$ can be represented as a
point in the unit cube $[0,1]^n$, where $n$ is the number of boxes in $\lambda$.
The increasing row and column constraints define a set of inequalities, making
$\mathcal{T}^{\lambda}$ identifiable with a convex polytope in $[0,1]^n$. This
polytope has positive volume, allowing us to equip it with a uniform probability
measure. Thus, it makes sense to speak about a \emph{uniformly random
	Poissonized tableau with shape $\lambda$}.

\subsubsection{Additional randomness of the shape}

This class becomes even broader if we allow $\lambda$ itself to be a random
variable. In this case, the \emph{uniformly random Poissonized tableau with
	shape $\lambda$} results from a two-step sampling process: first, we sample
$\lambda$ with some prescribed probability distribution, and then we sample a
Poissonized tableau with this shape.

As we already mentioned, a notable example within this broader class is the
aforementioned random Plancherel tableau which arises as the insertion tableau
corresponding to a sequence of \iid $U(0,1)$ random variables, see \cite[proof
of Lemma~5.1]{MarciniakSniadyAlternating}. This connection to the RSK algorithm
provides a bridge between Poissonized tableaux and classical combinatorial
constructions.

This framework offers significant flexibility in selecting the desired level of
randomness, ranging from $\lambda$ being a deterministic Young diagram to more
randomized examples. This flexibility will prove advantageous in subsequent
discussions. For instance, we will utilize some results from our recent work
\cite{MarciniakSniadyAlternating}, which concern the setup in which $\lambda$ is
deterministic. The main results of the current paper address scenarios where
$\lambda$ is either deterministic or random.

\subsubsection{Challenges in applying results to random Plancherel tableaux}

The setup of random Plancherel tableaux in \cref{conj:CLT-Plancherel} fits
within the general framework of \cref{thm:CLT-shape,thm:CLT-trans}. However, an
insufficient understanding of the level of randomness in the shape of $\lambda$
creates a gap that prevents directly applying these theorems to prove
\cref{conj:CLT-Plancherel}.

The primary challenge lies in the specific rate of convergence required by
\cref{thm:CLT-shape,thm:CLT-trans} for the tableau shape to converge towards the
limit shape. We are uncertain whether Plancherel-distributed random Young
diagrams achieve this rate, which hinders the direct application of these
theorems to prove \cref{conj:CLT-Plancherel}. See
\cref{sec:towards-big-conjecture} for more details.

\subsection{Assumptions of the main results. 
	Challenges in describing large Young diagrams} 
	\label{sec:describing-large-diagrams}

Describing large Young diagrams in asymptotic problems poses some
challenges. In the following, we will examine two standout methods developed to
address these issues.

\subsubsection{Method 1: continual diagrams and the Russian convention}

The first method involves using \emph{the Russian convention} (see
\cref{sec:young-diagrams}) and the \emph{profiles} / \emph{\conti diagrams} (see
\cref{sec:scaling,sec:continuous}). This approach provides an intuitive
understanding of the Young diagram shape as well as its convergence rates in
asymptotic problems. 
We used it to state the results and conjectures in
\cref{sec:intro}.

\subsubsection{Method 2: Kerov's transition measure}

The second method utilizes Kerov's \emph{transition measure}, a foundational
concept introduced by Sergey Kerov \cite{Kerov1993-transition,KerovBook}. This measure
represents the shape of a Young diagram as a centered probability measure on the
real line, see \cref{sec:transition-measure}. This very natural concept is
directly linked to \emph{the Plancherel growth process}, a fundamental tool for
harmonic analysis on the Young graph. Additionally, the moments of the
transition measure are related to the irreducible characters of the
symmetric group~\cite{Biane1998}. However, the bijection between continual
diagrams and centered probability measures is complex and non-local. Moreover,
the transition measure does not provide a direct intuitive understanding of the
rate of convergence for Young diagrams.

\subsubsection{Assumptions of the main results. Navigating between the two viewpoints}

The assumptions of our main results can be expressed using either method, as
summarized in \cref{tab:main_results}. The transition measure approach
(\cref{thm:CLT-trans,thm:determinism-old}) offers the most natural expression
for our findings. However, recognizing that some readers may find the transition
measure less intuitive, we also present results using continual diagrams
(\cref{thm:determinism-new,thm:CLT-shape}). This dual presentation ensures
accessibility to a broader audience, accommodating diverse conceptual
preferences.

The key distinctions between these methods are noteworthy:
\begin{itemize} 
	\item 
	Theorems based on
Kerov's transition measure have weaker assumptions, potentially allowing for
broader applications. These theorems are also proven directly, without any
detours. 
\item Conversely, theorems formulated using continual diagrams require stronger,
more intuitive assumptions, which can be advantageous for certain analyses,
albeit with a potentially narrower range of applications. Additionally, their
proofs are not direct; initially, we must reduce them to their counterparts
formulated in the language of Kerov's transition measure.
\end{itemize}

By presenting our findings through both lenses, we aim to provide a
comprehensive understanding of the problem.

\subsection{Conclusions of the main results}

\begin{table}[tbp]
	\centering
	\begin{tabular}{|l|c|c|}
		\hline
		\multirow{2}{*}{\textbf{Assumption: description of Young diagram $\lambda$}} & \multicolumn{2}{c|}{\textbf{Conclusion}} \\
		\cline{2-3}
		& \textbf{First-order asymptotics} & \textbf{Gaussian fluctuations} \\
		\hline
		via continual diagrams & \cref{thm:determinism-new} & \cref{thm:CLT-shape} \\
		\hline
		via Kerov's transition measure & \cref{thm:determinism-old} & \cref{thm:CLT-trans} \\
		\hline
	\end{tabular}
	\caption{Main results classified by conclusion and assumption type.}
	\label{tab:main_results}
\end{table}

\cref{tab:main_results} summarizes the four main results of this paper,
categorized by assumption type and conclusion. Among these findings, we particularly
highlight \cref{thm:CLT-shape} as our key contribution.
The results fall into two categories:
\begin{itemize}
	\item \textbf{First-order asymptotics
	(\cref{thm:determinism-new,thm:determinism-old}):} These results, analogous to
	the law of large numbers, demonstrate that the rescaled position of the new box
	converges in probability to a specific point on the plane and generalize the
	findings of Romik and \sniady (\cref{thm:Romik-Sniady-determinism}). These
	results require relatively weak assumptions about the convergence of the
	tableau's shape towards a limit shape.
	
	\item \textbf{Gaussian fluctuations (\cref{thm:CLT-shape,thm:CLT-trans}):}
	Similar to the central limit theorem, these findings  show that the rescaled
	fluctuation of the new box around the first-order approximation converges to a
	Gaussian distribution. These results align with \cref{conj:CLT-Plancherel}.
	These results have more stringent requirements, necessitating either the
	tableau shape or the \CDF of the corresponding
	transition measure to converge towards the limit at a prescribed rate.

\end{itemize}

\subsection{Proof strategy for the asymptotic Gaussianity of the new box position}

The crux of our paper lies in proving \cref{thm:CLT-trans}, which establishes
the asymptotic Gaussianity of the $u$-coordinate for the new box created during
Schensted row insertion.

Our approach builds upon our previous work \cite{MarciniakSniadyAlternating},
centering on the \emph{cumulative function} $u \mapsto F_T(u)$ as our primary
analytical tool. This function establishes a relationship between the position
of a newly created box and the value of the inserted number for a given
tableau~$T$, see \cref{sec:cumulative}. In our earlier analysis, we treated
$F_T(u)$ as a threshold, differentiating between small inserted numbers
(yielding new boxes to the left of $u$) and large inserted numbers (resulting in
new boxes to the right of $u$).

A crucial insight for the current paper is the inverse relationship between the
insertion function $z \mapsto \uIns(T; z)$ and the cumulative function $u
\mapsto F_T(u)$. This relationship parallels that of a probability
distribution's \CDF and its corresponding quantile function.
Consequently, to determine the $u$-coordinate $\uIns(T; z)$ of the new box for a
specific inserted value $z$, we must comprehend the threshold value $F_T(u)$
across the entire spectrum of $u$-coordinates.

An additional layer of complexity arises from the randomness of tableau $T$
(specifically, a uniformly random Poissonized tableau of fixed shape $\lambda$),
rendering the corresponding cumulative function $F_T$ random as well. To address
this, we introduce the \emph{double cumulative function}
$\doublecumulative_\lambda(u, z)$, a two-dimensional analogue of the usual
\CDF from the probability theory, see
\cref{sec:double-cumulative}. For a given diagram $\lambda$, the value of
$\doublecumulative_\lambda(u, z)$ represents the probability that the plot of
the cumulative function $u\mapsto F_T(u)$ lies southeast of the point $(u,z)$.
This construct integrates information about the new box's position, the inserted
number, and the associated probability, providing a comprehensive analytical
framework for our study.

Fixing one argument of the double cumulative function yields interesting
observations: 
\begin{itemize} 
	\item For a fixed $u$-coordinate $u_0$, the map $z \mapsto
	\doublecumulative_{\lambda}(u_0, z)$ coincides with the cumulative distribution
	function of the threshold value $F_T(u_0)$. Our previous work \cite{MarciniakSniadyAlternating} 
	established
	explicit upper bounds on its cumulants, allowing us to prove the asymptotic
	Gaussianity of $F_T(u_0)$. Consequently, we can approximate
	$\doublecumulative_{\lambda}(u, z)$ using the \CDF
	of the standard normal distribution.

	\item For a fixed inserted number $z$, the map $u \mapsto 1 -
\doublecumulative_{\lambda}(u, z)$ is the \CDF of
the $u$-coordinate $\uIns(z)$ of the new box created during insertion. This is
precisely the random variable we need to analyze. The asymptotics of the double
cumulative function found in the previous step provides the necessary
information to conclude the Gaussianity of $\uIns(z)$ which concludes the proof.
\end{itemize}

\subsection{Content of the paper}

This paper is structured as follows:

In \Cref{sec:preliminaries}, we introduce the necessary vocabulary for our
results, particularly focusing on \emph{Kerov's transition measure} of a Young
diagram, which will play an important role.

\Cref{sec:general-form} presents the two main results
(\cref{thm:determinism-new,thm:CLT-shape}) with the assumptions expressed in the
language of continual diagrams.

In \Cref{sec:main-thm-via-transmeasure}, we introduce \cref{thm:CLT-trans},
which reformulates \cref{thm:CLT-shape} under weaker assumptions and serves as
an intermediate step in its proof. The two theorems differ in their approach to
the assumptions: 
\begin{itemize} 
	\item 
	\Cref{thm:CLT-shape} is expressed in terms of the shape of the Young diagrams,	
	providing a more concrete geometric interpretation.
	
	\item 
	\Cref{thm:CLT-trans} is formulated in terms of the
	transition measures of the diagrams, offering a more general probabilistic
	perspective. 
\end{itemize}
This distinction allows \cref{thm:CLT-trans} to serve as a bridge between the
geometric and probabilistic aspects of the problem. It facilitates the proof of
\cref{thm:CLT-shape} while potentially extending its applicability to a broader
class of random structures.

\Cref{sec:old-version-implies-new-version} presents the proof that the main
result in the second formulation (\cref{thm:CLT-trans}) implies the main result
in the first formulation (\cref{thm:CLT-shape}).

In \Cref{sec:towards-proof}, we revisit the \emph{cumulative function} $u
\mapsto F_T(u)$, a concept we first introduced in our recent work
\cite{MarciniakSniadyAlternating}. In \cref{lem:alternating-trees}, we summarize
the key findings from that paper. While the original work presented a wealth of
results, \cref{lem:alternating-trees} focuses specifically on the aspects that
are essential for our current context.

In \Cref{sec:double-cumulative}, we introduce our primary tool: the \emph{double
	cumulative function} $\doublecumulative_{\lambda}$. This function provides
comprehensive insights into the relationship between the number inserted into a
random Poissonized tableau of shape $\lambda$, the location of the new box, and
the associated probabilities.

In \Cref{sec:first-order-double-function}, we present and prove
\cref{lem:convergence-easy}. This result offers a first-order approximation for
the asymptotic behavior of the double cumulative function, serving as a pivotal
step toward establishing the determinism of the Schensted insertion, as detailed
in \cref{thm:determinism-new}.

In \Cref{sec:fine-asymptotics}, we establish \cref{thm:gaussian-profile}, which
elucidates the fine asymptotic behavior of the double cumulative function
$\doublecumulative_{\lambda}$. This result is intricately connected to the
\CDF of the standard normal distribution and plays a
crucial role in the proof of \cref{thm:CLT-trans}.

Finally, \Cref{sec:proofs-of-main-results} culminates in the demonstration of
our principal findings: \cref{thm:determinism-new} and \cref{thm:CLT-trans}.
Here, we leverage the tools developed in the preceding sections to achieve our
conclusions.

\section{Preliminaries: transition measure and \conti diagrams}
\label{sec:preliminaries}

\subsection{Plancherel growth process}

Let $w_1, w_2, \ldots$ be a sequence of \iid random variables with the uniform distribution $U(0,1)$. We define:
\[
\lambda^{(n)} := \RSK(w_1, \ldots, w_n).
\]
The resulting random sequence of Young diagrams:
\begin{equation}
\label{eq:Plancherel-growth-process}	
\emptyset = \lambda^{(0)} \nearrow \lambda^{(1)} \nearrow \lambda^{(2)} \nearrow \cdots
\end{equation}
is called the \emph{Plancherel growth process}.
This process forms a Markov chain with specific transition probabilities which
will be discussed below. The construction described above is a special case of
the more general setup discussed earlier in Section~\ref{sec:extremal}.

\subsection{Transition measure of a Young diagram}
\label{sec:transition-measure}

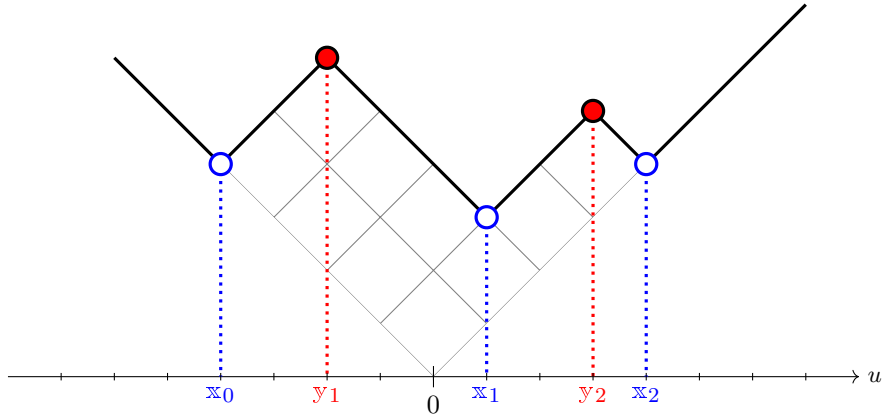
\begin{figure}
    	
	\begin{tikzpicture}[scale=1,rotate=45]
		\coordinate (p0) at (4,0);
		\coordinate (p1) at (4,1);
		\coordinate (p2) at (2,1);
		\coordinate (p3) at (2,4);
		\coordinate (p4) at (0,4);
		\coordinate (p5) at (0,6);
		\begin{scope}
			\clip  (p0) -- (p1) -- (p2) -- (p3) -- (p4) -- (0,0);
			\draw[black!50] (0,0) grid (10,10);
		\end{scope}
		\draw[very thick] (7,0) -- (p0) -- (p1) -- (p2) -- (p3) -- (p4) -- (p5);
		\draw[->] (-4,4) -- (4,-4) node[anchor=west] {$u$};
		\draw[dotted,blue,very thick] (p0) -- (4/2,-4/2) node[anchor=north] {\small $\kerx_2$};
		\draw[dotted,blue,very thick] (p2) -- (1/2,-1/2) node[anchor=north] {\small $\kerx_1$};
		\draw[dotted,blue,very thick] (p4) -- (-4/2,4/2) node[anchor=north] {\small $\kerx_0$};
		\draw[dotted,red,very thick] (p1) -- (3/2,-3/2) node[anchor=north] {\small $\kery_2$};
		\draw[dotted,red,very thick] (p3) -- (-2/2,2/2) node[anchor=north] {\small $\kery_1$};
        \draw (0,0) +(0.1,0.1) -- +(-0.1,-0.1) node[anchor=north] {\small $0$};
		\draw[blue,very thick,fill=white] 
		(p0) circle (4pt)
		(p2) circle (4pt)
		(p4) circle (4pt);
		\draw[very thick,fill=red] 
		(p1) circle (4pt)
		(p3) circle (4pt);
		
		\foreach \u in {-7,...,7}
		{\draw (\u/2,-\u/2)  +(-1pt,-1pt) -- +(1pt,1pt); }
	\end{tikzpicture}
    \caption{Concave corners (empty, blue) and convex corners (filled, red) 
    	    of a Young diagram $(4,2,2,2)$ and their $u$-coordinates.}
    \label{fig:cauchy}
\end{figure}

For a given Young diagram~$\lambda$ with $n$ boxes we denote by
$\kerx_0<\cdots<\kerx_{\kerell}$ the $u$-coordinates of its concave corners and
by $\kery_1<\cdots<\kery_\kerell$ the $u$-coordinates of its \emph{convex
	corners}, see \cref{fig:cauchy}. The \emph{Cauchy transform} of $\lambda$ is
defined as the rational function (see \cite{Kerov1993-transition} and \cite[Chapter~4,
Section~1]{KerovBook}):
\begin{equation}
    \label{eq:cauchy-def}     
    \cauchy_\lambda(z)= 
        \frac{(z-\kery_1)\cdots(z-\kery_\kerell)}{(z-\kerx_0)\cdots(z-\kerx_\kerell)}.
\end{equation}
Note that in the work of Kerov this function is called \emph{the generating function} of $\lambda$.
The Cauchy transform can be written uniquely as a sum of simple fractions:
\[ \cauchy_\lambda(z)= \sum_{0\leq i \leq \kerell}  \frac{p_i}{z-\kerx_i} \]
with the coefficients $p_0,\dots,p_\kerell>0$ such that $p_0+\cdots+p_\kerell=1$.
We define \emph{the transition measure} of $\lambda$ 
as the discrete measure
\[ \mu_\lambda= p_0\ \delta_{\kerx_0} + \cdots + p_\kerell\ \delta_{\kerx_\kerell};\]
in this way
\[ \cauchy_\lambda(z) = \int_{\R} \frac{1}{z-x} \dif \mu_\lambda(x)\]
is indeed the Cauchy transform of $\mu_\lambda$. 

Kerov proved that the transition probabilities of the Markov chain
\eqref{eq:Plancherel-growth-process} are encoded by the transition measure, see
\cite{Kerov1993-transition} and \cite[Chapter~4, Secton~1]{KerovBook}. More specifically,
the conditional probability that the new box will have the $u$-coordinate equal
to~$\kerx_i$ is given by the corresponding atom of the transition measure
\begin{equation}
    \label{eq:residuum} 
    \Pro\giventhat*{ u\left(\lambda^{(n+1)} / \lambda^{(n)}\right)
    	= \kerx_i }{ \lambda^{(n)}=\lambda } 
        =  p_i.   
\end{equation}

We denote by 
\begin{equation}
    \label{label:CDFkerov}
    K_\lambda(u)= \mu_\lambda\big( (-\infty,u] \big) \qquad \text{for } u\in\R
\end{equation}
the \CDF of $\mu_{\lambda}$. To keep the
notation short, we will call it \emph{the cumulative function of $\lambda$}.

\subsection{\Conti diagrams}
\label{sec:continuous}

We say that $\omega\colon \R \to \R_+$ is a \emph{\conti diagram} (see
\cite{Kerov1993-transition} and \cite[Chapter~4, Secton~1]{KerovBook}) if the following two
conditions are satisfied:
\begin{itemize}
	\item $|\omega(u_1) - \omega(u_2) | \leq | u_1 - u_2|$ holds true for any
	$u_1,u_2\in\R$,
	\item $\omega(u) = |u|$ for sufficiently large $|u|$.
\end{itemize}
In the literature, such \conti diagrams are also referred to as
\emph{generalized diagrams}, \emph{continuous diagrams}, or simply
\emph{diagrams}. 

Notably, an important class of examples arises from the
profiles $\omega_{\lambda}$ associated with the usual Young diagrams. Another
important class of examples arises from the \emph{rescaled} profiles
\begin{equation}
	\label{eq:rescaled-normalized}
	   \omega_{\frac{1}{\sqrt{n}} \lambda^{(n)}}, 
\end{equation}
where $\lambda^{(n)}$ is a Young diagram with $n$ boxes, see \cref{sec:scaling}.
The motivation for the scaling factor in \eqref{eq:rescaled-normalized} lies in
the observation that if each box of $\lambda^{(n)}$ is drawn as a square with
the side~$\frac{1}{\sqrt{n}}$ then the total area of the boxes is equal to $1$,
irrespective of the Young diagram size~$n$. In the current paper the \conti
diagrams will serve as limit objects toward which a sequence of (rescaled) Young
diagrams of the form \eqref{eq:rescaled-normalized} can converge.

\subsection{Transition measure for \conti diagrams} 

Let $\omega=\omega_{\lambda}$ be a profile of a Young diagram $\lambda$; we will
use the notations from \cref{fig:cauchy}. The second derivative $\omega''$ is
well-defined as a Schwartz distribution and can be identified with a signed
measure on the real line. The positive part of this measure is supported in the
concave corners $\kerx_0,\dots,\kerx_{\kerell}$ while the negative part of this
measure is supported in the convex corners $\kery_1,\dots,\kery_\kerell$, with
each atom having equal weight~$2$. An application of logarithm transforms the
product on the right-hand side of \eqref{eq:cauchy-def} into a sum. It was
observed by Kerov that it can be conveniently written in the form
\begin{multline}
	\label{eq:transition-general}
	\log\big[ z \cauchy_\omega(z) \big]= - 
\int_{-\infty}^\infty \log (z-w) \ \left(\frac{\omega(w)-|w|}{2}\right)'' \dif w = \\
- \int_{-\infty}^\infty \frac{1}{z-w} \ \left(\frac{\omega(w)-|w|}{2}\right)' \dif w,
\end{multline}
where $z$ is a complex variable.

We can drop the assumption that $\omega=\omega_{\lambda}$ is a profile of a
Young diagram and use \eqref{eq:transition-general} to define the Cauchy
transform $\cauchy_\omega$ for an arbitrary \conti diagram $\omega$. The
transition measure~$\mu_\omega$ is then defined from the Cauchy transform
$\cauchy_\omega$ via Stieltjes inversion formula. The \emph{cumulative function}
$K_\omega$ of a \conti diagram $\omega$ is defined analogously as in
\eqref{label:CDFkerov} as the \CDF of the
corresponding transition measure
\[  K_\omega(u)= \mu_\omega\big( (-\infty,u] \big) \qquad \text{for $u\in\R$.} \]

\subsection{Transition measure as a homeomorphism}

For $C>0$ we say that \emph{a \conti diagram~$\omega$ is supported on the
	interval $[-C,C]$} if $\omega(u)=|u|$ for each real number $u$ such that
$|u|\geq C$. For example, for a rescaled Young diagram  $\omega=\omega_{c
	\lambda}$ this holds true if and only if the rescaled first row, as well as
rescaled first column are shorter than $C$:
\[ c \lambda_1 \leq C, \qquad c \lambda'_1 \leq C .\]

The following result is due to Kerov (see \cite{Kerov1993-transition} and
\cite[Chapter~4, Section~1]{KerovBook}).

\begin{proposition}
	\label{prop:homeomorphism} 
	
	For each constant $C>0$, the map $\omega \mapsto
	\mu_{\omega}$, which associates a \conti diagram with its transition measure,
	is a \emph{homeomorphism} between the following two topological spaces:
	\begin{itemize}
		\item The set of \conti diagrams supported on $[-C,C]$, with the
		topology given by the supremum distance.
		
		\item The set of centered probability measures supported on the
		interval $[-C, C]$, with the weak topology of probability measures.
	\end{itemize}
\end{proposition}

This convenient homeomorphism allows us to change the perspective and to
parameterize the set of (\conti) diagrams not by their profiles, but via their
transition measure. This viewpoint has proven to be very fruitful for problems
in asymptotic representation theory; see
\cite{Kerov1993-gaussian,Biane1998,Sniady2006a,DFS} for some examples.

\subsection{The metric $d_{XY}$ on the set of \conti diagrams} 

While the supremum distance provides a useful topology for the set of \conti
diagrams (see, for example \cref{prop:homeomorphism}), some problems benefit
from a stronger topology. For this purpose, we introduce the metric $d_{XY}$,
recently proposed by \sniady \cite{Sniady2024modulus}. This metric offers a more
nuanced approach, essentially measuring the supremum distance along both the
$x$-coordinate and the $y$-coordinate separately. Heuristically, we can think of
$d_{XY}$ as capturing the maximum discrepancy between two diagrams when viewed
from both horizontal and vertical perspectives. This approach provides a
finer-grained comparison between \conti diagrams, potentially revealing subtle
differences that the standard supremum distance might overlook. The formal
definition of the metric $d_{XY}$ is detailed below.

\smallskip

For the purposes of this section we will use the following notational
shorthand. For $x,y\geq 0$ and a \conti diagram $\omega\colon\R \to\R_+$ we
will write $(x,y) \in \omega$ if the point $(x,y)$ belongs to the graph of
$\omega$ \emph{drawn in the French coordinate system} or, equivalently, if
\[   \omega(x-y) = x+y, \]
cf.~\eqref{eq:Russian}.  Additionally, for $x\geq 0$ we denote by 
\[ \Pi_Y^\omega(x) = \big\{ y : (x,y) \in \omega \bigskip\} \subseteq \R_+ \]
the projection on the $y$-axis of the intersection of the plot of $\omega$ with
the line having a specified $x$-coordinate. One can verify that
$\Pi_Y^\omega(x)$ is a non-empty closed set. Similarly, for $y\geq 0$ we denote
by
\[ \Pi_X^\omega(y) = \big\{ x : (x,y) \in \omega \bigskip\} \subseteq \R_+\]
the projection on the $x$-axis.

\smallskip

If $\omega_1,\omega_2$ are \conti diagrams, we define their $y$-distance
\[ d_Y(\omega_1, \omega_2) =  
\sup_{x\geq 0} d_H\big( \Pi_Y^{\omega_1}(x), \Pi_Y^{\omega_2}(x) \big)
\]
as the supremum (over all choices of the $x$-coordinate) of the \emph{Hausdorff
	distance} $d_H$ between their $y$-projections. Heuristically, the metric
$d_Y$
is a way to quantify the discrepancy between \conti diagrams along the
$y$-coordinate. In a similar manner we define the $x$-distance between
\conti diagrams as
\[ d_X(\omega_1, \omega_2) =  
\sup_{y\geq 0} d_H\big( \Pi_X^{\omega_1}(y), \Pi_X^{\omega_2}(y) \big).
\]

\begin{remark}
	\label{rem:dX-concrete}
Note that in the special case when $\omega_1= \omega_{\lambda}$ and
$\omega_2= \omega_{\mu}$ are profiles of the Young diagrams
$\lambda=( \lambda_1, \lambda_2, \dots )$ and $\mu=(
\mu_1, \mu_2, \dots )$ this distance has a simple
interpretation
\[ d_X( \omega_\lambda, \omega_\mu ) = \max_i \big| \lambda_i - \mu_i \big| \]
as the supremum distance between the parts of the partitions. Similarly, the
$y$-distance $d_Y( \omega_\lambda, \omega_\mu)$ is the supremum distance for the
parts of the conjugate partitions $\lambda'$ and $\mu'$.
\end{remark}

\smallskip

Finally, we define the $xy$-distance between $\omega_1$ and $\omega_2$ as the
maximum of the $x$- and the $y$-distance:
\[ d_{XY}(\omega_1,\omega_2) = \max\big( d_X(\omega_1,\omega_2),\ 
d_Y(\omega_1,\omega_2) \big). \]

\section{The general form of RSK insertion fluctuations}
\label{sec:general-form}

\subsection{Asymptotic determinism of Schensted insertion}

We start with the following generalization of the aforementioned result of Romik
and \sniady (\cref{thm:Romik-Sniady-determinism}). 

\begin{theorem}[Asymptotic determinism of Schensted insertion]
    \label{thm:determinism-new}
    Assume that:
    \begin{itemize}    
\item For each $n\geq 1$ we are given a random Young diagram $\lambda^{(n)}$; we
denote by
\[ \omega_n = \omega_{\frac{1}{\sqrt{n}} \lambda^{(n)}} \] 
its rescaled profile.

\item There is a \conti diagram $\Omega$ such that the convergence in probability
\[ \omega_n(u) \xrightarrow[n \to\infty]{P} \Omega(u) \]
holds true for each $u\in \R$.

\item There is a constant $C>0$ such that the probability of the event that the
rescaled diagram~$\omega_{n}$ is supported on the interval $[-C, C]$ converges
to $1$, as $n\to\infty$.

\item A real number $0<z <1$ is such that the quantile function $K_{\Omega}^{-1}$ of
the transition measure~$\mu_\Omega$ is continuous in $z$.
\end{itemize}

\smallskip

	Let $T^{(\nklatki)}$ be a uniformly random Poissonized tableau of shape
$\lambda^{(\nklatki)}$. Then the rescaled $u$-coordinate of the new box created
by the insertion of $z$
    \[ 
    \frac{1}{\sqrt{n}} \uIns( T^{(n)}; z)  \xrightarrow[n\to\infty]{P}
    K^{-1}_{\Omega}(z) \]
    converges in probability to the corresponding quantile of the transition measure $\mu_\Omega$.
\end{theorem}
In particular, \cref{thm:determinism-new} combined with the classic result of
Logan, Shepp, Vershik and Kerov (\cref{thm:plancherel-limitshape}, \emph{``the
	asymptotic shape of random Plancherel-distributed Young diagrams''}) provides a
conceptually new proof of the result by Romik and \sniady
(\cref{thm:Romik-Sniady-determinism}, \emph{``the asymptotic determinism of RSK
	insertion for \iid random variables''}).

The proof of \cref{thm:determinism-new} is postponed to \cref{sec:proof-of-determinism}.

\begin{remark}
	To avoid potential confusion, we clarify that throughout this paper---in 
	\cref{thm:determinism-new} above, as well as in 
	Theorems \ref{thm:CLT-shape}, \ref{thm:CLT-trans}, and \ref{thm:determinism-old} in other 
	sections---the parameter~$n$ is a ``size'' variable that grows roughly 
	linearly with the number of boxes. While in some specific situations $n$ 
	may equal the actual number of boxes, these theorems do not require the 
	Young diagram $\lambda^{(n)}$ to contain exactly $n$ boxes, nor do they 
	assume that the ratio $\frac{|\lambda^{(n)}|}{n}$ converges to $1$.
\end{remark}

\subsection{The interaction energy}
\label{sec:may-force-be-with-you}

If $\mu$ is a probability measure on the real line and $u_0\in \R$ we define
\begin{equation}
    \label{eq:interaction-enerrgy}
    \Force_\mu(u_0) = 
    \iint\limits_{\{(z_1,z_2) :\  z_1 \leq u_0 < z_2 \}}
    \frac{1}{z_2-z_1} \dif \mu(z_1) \dif \mu(z_2)
\end{equation}
whenever this double integral is finite. This quantity will play an important
role in the statement of our main result.

Let us interpret $\mu$ as the distribution of the electrostatic charge along a
one-dimensional rod. If we split this rod at the point $u_0$, the two parts will
act on one another with the electrostatic force. The integral
\eqref{eq:interaction-enerrgy} can be interpreted as the \emph{interaction
	energy} between these two parts.

\medskip

We now consider an alternative universe in which the space is two-dimensional so
that the electrostatic force decays as the inverse of the distance between the
charges. It follows that the double integral $\Force_\mu(u_0)$ is equal to the 
\emph{electrostatic force} between the two parts of the rod.

If we separate the two parts of the rod by an additional distance $l$, the
aforementioned electrostatic force is the derivative of the total energy of the
system with respect to the variable $l$. It is worth pointing out that in our
context \emph{the total energy} should be understood as the logarithmic energy
which is ubiquitous in random matrix theory \cite{BenArous-Guionnet}, Voiculescu's free entropy
\cite{Voiculescu1994}, as well as in asymptotic representation theory
\cite[page~46, quantity $Q(h)$]{Romik2015}. For this reason we suspect that the
interaction energy $\Force_\mu(u_0)$ can be found also in the context of random
matrix theory.

\medskip

Suppose that the support of the measure $\mu$ is contained in some interval
$[u_{\min}, u_{\max}]$. We choose an arbitrary real number $u_{-\infty}$ such
that $u_{-\infty}< u_{\min}$ and consider a contour $\contour$ on the complex
plane shown in \cref{fig:contour}. This contour was chosen in such a way that it
crosses the real line in two points, namely, $u_{-\infty}$ and $u_0$.

\begin{figure}
        
\begin{tikzpicture}
\begin{scope}[thick,red,decoration={
        markings,
        mark=at position 0.9 with {\arrow{>}}}
    ] 
    \draw[postaction={decorate}] (-3,-2)--(1,-2);
    \draw[postaction={decorate}] (1,-2) --(1,2);
    \draw[postaction={decorate}] (1,2)--(-3,2);
    \draw[postaction={decorate}] (-3,2)--(-3,-2);
\end{scope}

    \draw[->] (-5,0) -- (5,0);
    \draw[->] (0,-3) -- (0,3);
    \draw[ultra thick] (-2,0) -- (3,0);
    \draw[ultra thick] (-2,-0.1)-- (-2,0.1) node[anchor=south]{$u_{\min}$};
    \draw[ultra thick] (3,-0.1)-- (3,0.1) node[anchor=south]{$u_{\max}$};
    \draw[ultra thick] (1,-0.1)-- (1,0.1) node[anchor=south west]{$u_0$};
    \draw[ultra thick] (-3,-0.1)-- (-3,0.1) node[anchor=south east]{$u_{-\infty}$};

\end{tikzpicture}
    \caption{The contour $\contour$ on the complex plane.}
    \label{fig:contour}
\end{figure}
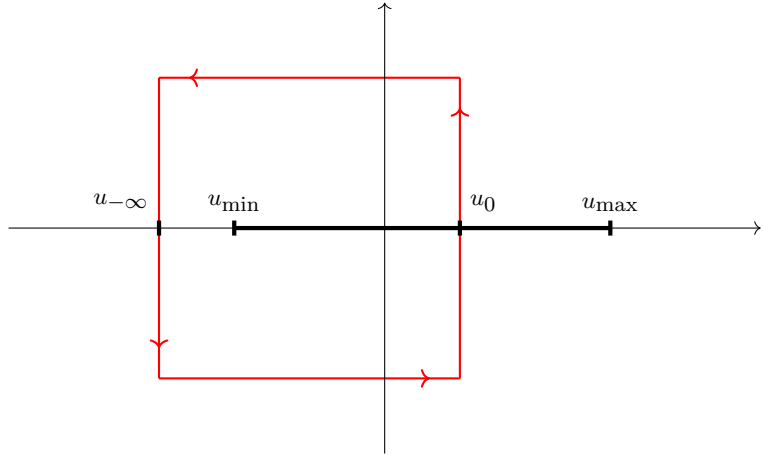

In the special case when the support of $\mu$ is a finite set
and $u_0$ does not belong to the support of $\mu$ we may apply Cauchy's
residue theorem and the interaction energy is equal
to the contour integral
\begin{equation}
    \label{eq:energy-contour-integral}
    \Force_{\mu_\Omega}(u_0) = - \frac{1}{4 \pi i} \oint_{\contour} \big[ \cauchy_\mu(z) \big]^2 \dif z,
\end{equation}
where 
$$ \cauchy_\mu(z) = \int_{\R} \frac{1}{z-x} \dif \mu(x)$$
is the Cauchy transform of $\mu$. The formula \eqref{eq:energy-contour-integral}
remains true for general measures $\mu$ as long as the Cauchy transform is
sufficiently regular in the neighborhood of $u_0$, however the justification of
its validity is more technically involved. In practical applications the formula
\eqref{eq:energy-contour-integral} is more convenient than
\eqref{eq:interaction-enerrgy}.

\subsection{The main theorem: General form of the fluctuations}

\begin{theorem}[The main theorem; fluctuations of Schensted insertion are Gaussian]
	\label{thm:CLT-shape}
    Assume that:
\begin{enumerate}[label=(A\arabic*)]   
	\item For each $n\geq 1$ we are given a random Young diagram $\lambda^{(n)}$; we denote by
	\begin{equation}
		\label{eq:rescaled-omega}
			 \omega_n = \omega_{\frac{1}{\sqrt{n}} \lambda^{(n)}} 
	\end{equation} 
	its rescaled profile.	
	
	\item \label{item:A2}
	There is a \conti diagram $\Omega$ which is the limit of $\omega_n$ (with
respect to the metric $d_{XY}$), furthermore the rate of convergence is such that
	\[ \sqrt[4]{n}\ \log n \ \cdot \ d_{XY}(\omega_n, \Omega)
	 \xrightarrow[n\to\infty]{P} 0. \]

	\item \label{item:transition-measure-omega-is-absolutely-continuous}
	There are constants $a<u_0<b$ such that the transition measure $\mu_\Omega$
restricted to the interval $[a,b]$ is absolutely continuous, with a density
given by a continuous, strictly positive function $f>0$. We denote 
\[z_0= K_\Omega(u_0), \qquad 
\Force_0 = \Force_{\mu_\Omega}(u_0), \qquad 
f_0 = f(u_0). \]

\item \label{item:A4} The restriction  $\Omega\colon [a,b] \to\R$ of the \conti diagram $\Omega$
to the interval $[a,b]$ is a contraction (i.e., it is a $K$-Lipschitz function
for some constant $K<1$).

\end{enumerate}	
	\medskip
	
	Let $T^{(\nklatki)}$ be a uniformly random Poissonized tableau of shape $\lambda^{(\nklatki)}$.
	Then the fluctuations of the newly created box around the first-order approximation
	\begin{equation}\label{eq:my-favorite-rv}
		\sqrt[4]{\nklatki} \left[ \frac{\uIns( T^{(\nklatki)}; z_0 )}{\sqrt{\nklatki}} -u_0 \right]
		\xrightarrow[\nklatki\to\infty]{\mathcal{D}} N\left( 0, \frac{\Force_0}{f_0^2} \right)
	\end{equation}    
	converge to a centered Gaussian measure with the variance
$\frac{\Force_0}{f_0^2}$.
\end{theorem}

The proof is quite involved, the first step is to reduce this theorem to
another result (\cref{thm:CLT-trans}). We postpone the details to
\cref{sec:old-version-implies-new-version}.

\subsection{Explicit examples for \cref{thm:CLT-shape}}
\label{sec:examples}

The objective of this section is to provide tangible examples that demonstrate
the applicability of \cref{thm:CLT-shape}. While many theoretical examples
exist, we specifically seek to move beyond abstract objects by focusing on cases
where the essential components—the sequence of Young diagrams, the limit \conti
diagram, its transition measure, and the
interaction energy—are all computationally tractable.

\subsubsection{Example: staircase tableaux}
\label{sec:staircase}

In \cref{thm:CLT-shape} we will pass to the subsequence of triangular numbers
\[ \nklatki={\nklatki}_N = 1+\cdots+N= \frac{N(N+1)}{2}. \]
We consider the sequence of deterministic Young diagrams $(\lambda^{(\nklatki_N)})$, where 
\begin{equation}\label{eq:staircase}
    \lambda^{(\nklatki_N)}= (N,N-1,\dots,3,2,1) 
\end{equation}
is a \emph{staircase Young diagram} with $\nklatki_N$ boxes. We define the
rescaled diagram $\omega_n$ via \eqref{eq:rescaled-omega}.

Let $\Omega\colon\R\to\R_+$ be the \emph{isosceles triangle diagram} given by
\[ \Omega(u)= \begin{cases}
	\sqrt{2} & \text{for $|u|\leq \sqrt{2}$},  \\
	|u|      & \text{otherwise}.
\end{cases}
\]

In the following we will verify that the assumptions of \cref{thm:CLT-shape} 
are satisfied for this configuration and we will calculate the missing data: 
the density of the transition measure and its interaction energy.

\medskip

The rate of convergence of the rescaled diagrams towards the limit is very fast:
\[ d_{XY}( \omega_n, \Omega ) = O\left( \frac{1}{N} \right)
= O\left( \frac{1}{\sqrt{n}} \right) \ll o\left( \frac{1}{\sqrt[4]{n} \log n} \right), \]
which ensures that \assum{item:A2} in \cref{thm:CLT-shape} holds. The
example of staircase Young diagrams and the rate of convergence of their
transition measures $\mu_{\omega_n}$ to the limit $\mu_{\Omega}$ were studied in
more detail in the recent work of \sniady \cite{Sniady2024modulus}.

\smallskip

Since the second derivative $\Omega''$ is a Schwartz distribution which can be
identified with a very simple measure supported in $\pm \sqrt{2}$, a
straightforward calculation based on \eqref{eq:transition-general} shows that
the corresponding Cauchy transform is given by
\begin{equation}
	\label{eq:cauchy-for-AS}
	 \cauchy_{\AS}(z)= \frac{1}{\sqrt{z^2-2}} \qquad \text{for $z\in\CC\setminus I$},
\end{equation}
where $I=\big[ -\sqrt{2}, \sqrt{2} \big]$. It is easy to verify (for example, by
the Stieltjes inversion formula) that the corresponding transition measure
$\mu_{\Omega}=\mu_{\AS}$ is the arcsine distribution which is supported on the
interval $I$, with the density
\[ f_{\AS}(u) = \frac{1}{\pi \sqrt{2-u^2} } \qquad \text{for $u\in I$}\]
and the \CDF
\[ K_{\Omega}(u)= F_{\AS}(u) = \frac{1}{2} + \frac{1}{\pi} \arcsin \frac{u}{\sqrt{2}}   
       \qquad \text{for $u\in I$}.\]

The formula \eqref{eq:energy-contour-integral} combined with
\eqref{eq:cauchy-for-AS} implies that if $-\sqrt{2} < u_0 < \sqrt{2}$ then the
value of $(-2) \Force_{\AS}(u_0)$ is equal to the residue of the rational
function $z\mapsto  \frac{1}{z^2-2}$ in $z=-\sqrt{2}$. In this way we get that
the interaction energy of the arcsine measure is given by the piecewise constant
function
\begin{equation}
    \label{eq:arc-sine-force}
    \Force_{\AS}(u)= 
    \begin{cases}
        0                  & \text{if $u< -\sqrt{2}$}, \\[1ex]   
        \displaystyle\frac{1}{4 \sqrt{2}} & \text{if $-\sqrt{2} < u < \sqrt{2}$}, \\[2ex]
        0                  & \text{if $u > \sqrt{2}$}.
    \end{cases}   
\end{equation}

\begin{corollary}[Schensted insertion into staircase tableaux]
	\label{corollary:staircase}
    Let $0<z_0<1$ be fixed and let 
    \[ u_0=F_{\AS}^{-1}(z_0)=  - \sqrt{2}\ \cos \pi  z_0  \]
    be the corresponding quantile of the arcsine distribution. For each integer
$n$ of the form $n=n_N$ let $T^{(n)}$ be a uniformly random Poissonized
tableau of the staircase shape $\lambda^{(n_N)}$.
    
    Then the sequence of random variables
    \[ \sqrt[4]{n} \left[ \frac{1}{\sqrt{n}} \uIns\left( T^{(n)}; z_0 \right) - u_0 \right] 
    \xrightarrow{\mathcal{D}} N\left( 0, \sigma_{\AS}^2(u_0) \right) \]
    converges in distribution to the centered Gaussian measure with the variance
    \[ \sigma_{\AS}^2(u_0) = \frac{\pi^2}{4 \sqrt{2}}  \left(2-u_0^2\right). \]
\end{corollary}
\begin{proof}
	We will verify that the assumptions of \cref{thm:CLT-shape} are satisfied.
	The relationship between $u_0$ and $z_0$ is precisely as required in
	\assum{item:transition-measure-omega-is-absolutely-continuous}.
	\assum{item:A2} was already verified above. 
	
For any numbers $a$ and $b$ with $-\sqrt{2} < a < u_0 < b < \sqrt{2}$,
the density $f_{\AS}$ restricted to $[a,b]$ is continuous and strictly 
positive. Moreover, the restriction of $\Omega$ to $[a,b]$ is $0$-Lipschitz 
(hence a contraction), satisfying Assumption~\ref{item:A4}.

The result follows from \cref{thm:CLT-shape}, with the variance computed 
from the interaction energy \eqref{eq:arc-sine-force} and the density of 
the arcsine distribution.
\end{proof}

We will revisit this example in \cref{sec:insertion-staircase} and
\cref{sec:staircase3}.

\subsubsection{Almost isosceles triangle tableaux}
\label{sec:almost-isosceles}

We revisit and generalize the above example by defining the auxiliary functions
\begin{align*}
	Y(x) &= \max\left( \sqrt{2}-x, \ 0 \right),  \\
	X(y) &= \max\left( \sqrt{2}-y, \ 0 \right) 
\end{align*}
which provide the parametrization of the \conti diagram $\Omega$ in the
French coordinate system $(x,y)$ and its transpose $(y,x)$, respectively.
 
Unlike the previous case, there is no need to restrict
to a subsequence. Consider any sequence of (deterministic or random) Young
diagrams $\big( \lambda^{(n)} \big)$ such that the rescaled rows converge to the
shape specified by the isosceles triangle diagram at the specific rate:
\begin{align}  
\label{eq:convergence-rows}	
\max_{i\geq 1}  \frac{\log n}{\sqrt[4]{n}}  
\left|  \lambda^{(n)}_i   - 
\sqrt{n}\ X\left( \frac{i}{\sqrt{n} } \right)   
\right| 
& \xrightarrow[n\to\infty]{P} 0,  \\
\intertext{and similarly for the columns:}
\label{eq:convergence-columns}
\max_{i\geq 1}  \frac{\log n}{\sqrt[4]{n}}  
\left|  \big(\lambda^{(n)}\big)'_i   - 
\sqrt{n}\ Y\left( \frac{i}{\sqrt{n} } \right)   
\right| 
& \xrightarrow[n\to\infty]{P} 0.  
\end{align}

It is straightforward to verify that these assumptions ensure 
condition \ref{item:A2} is satisfied, making \cref{thm:CLT-shape} 
applicable. Fortunately, all calculations for the \conti isosceles triangle 
diagram $\Omega$---including the transition measure, its CDF, and the interaction 
energy---were already completed in \cref{sec:staircase} and remain valid 
here.

\begin{corollary}[Schensted insertion into almost isosceles triangle tableaux]
	\label{coro:staircase-more-general}
	Under the above assumptions, \cref{corollary:staircase} remains valid 
	when $T^{(n)}$ is a uniformly random Poissonized tableau of shape
	$\lambda^{(n)}$.
\end{corollary}

This result demonstrates the power of \cref{thm:CLT-shape}: unlike the 
staircase case, we do not require Young diagrams to have a very specific 
shape. The corollary applies to any diagram sequence satisfying the 
convergence conditions above.

For comparison, staircase diagrams enjoy explicit transition measure 
formulas \cite{Sniady2024modulus}, making results like 
\cref{thm:CLT-trans} directly applicable. For the broader class of diagrams studied 
here, such explicit calculations are not available---\cref{thm:CLT-trans} 
cannot be applied directly---highlighting the significance of 
\cref{thm:CLT-shape}.

\subsubsection{More general almost triangle tableaux}

We extend the example from \cref{sec:almost-isosceles} even more.
We fix real parameters $a$ and $b$ with $|a| < 1$ and $b > 0$, and consider 
the \conti triangle diagram
\[ \Omega(u) = \max\big( au + b, \ |u| \big). \]
Note that $a = 0$ and $b = \sqrt{2}$ recovers the \conti isosceles triangle
diagram from \cref{sec:staircase}. In what follows, we compute the Cauchy
transform, transition measure, CDF, and interaction energy for this \conti
diagram. The calculations follow those in \cref{sec:staircase} closely, though
the results are less elegant. Since the methods are analogous, we provide
abbreviated explanations.

\paragraph{Critical points and Cauchy transform}

The critical points of the triangle diagram $\Omega$, which delineate the different 
linear segments, are
\[ u_1 = -\frac{b}{1+a}, \qquad\qquad u_2 = \frac{b}{1-a}. \]

The Cauchy transform of $\Omega$ takes the form 
\[ \cauchy(z) = 
(z - u_1)^{-\frac{1+a}{2}} (z - u_2)^{-\frac{1-a}{2}} 
\quad \text{for } z \in \CC \setminus I, \]
where $I = [u_1, u_2]$.

\paragraph{Transition measure and CDF}
The corresponding transition measure $\mu_\Omega$ is a scaled and centered
beta distribution with shape parameters
\[ \alpha = \frac{1-a}{2}, \qquad\qquad \beta = \frac{1+a}{2}, \]
supported on the interval $I$. Its probability density function is given by 
\[ f(u) = \frac{1}{\pi} \cos \frac{a\pi}{2} \ (u - u_1)^{-\frac{1+a}{2}}\ 
(u_2 - u)^{-\frac{1-a}{2}} \qquad \text{for } u \in I. \]
The corresponding cumulative distribution function $K_{\Omega}$ can be 
expressed explicitly in terms of the regularized incomplete beta function:
\[ K_{\Omega}(u) = 
I_{\frac{u - u_1}{u_2 - u_1}}\left( \frac{1-a}{2}, 
\frac{1+a}{2} \right) \qquad \text{for } u \in I. \]

\paragraph{Interaction energy}
Computing the interaction energy via \eqref{eq:energy-contour-integral} is 
more challenging since it involves a contour integral of the function
\[ \big[ \cauchy(z) \big]^2
= (z - u_1)^{-1-a} (z - u_2)^{-1+a} \quad \text{for } z \in \CC \setminus I, \]
which is not analytic on the interval $I$ except when $a=0$. 

For $a<0$, we choose the contour $\contour$ from \cref{fig:contour} as a 
thin rectangle very close to the real axis, ultimately taking the limit 
as the rectangle's height approaches zero. The interaction energy evaluates to
\[ \Force(u) = -\frac{1}{2 \pi} \sin \pi a \cdot \int_{u_1}^{u} 
(z-u_1)^{-1-a} (u_2-z)^{-1+a} \dif z \qquad \text{for } u \in I. \]
Using the substitution $t = \frac{z-u_1}{u_2-u_1}$, this can be expressed as
\[ \Force(u) = -\frac{1}{2 \pi} \sin \pi a \cdot \frac{1}{u_2-u_1} \cdot 
B\left(\frac{u-u_1}{u_2-u_1}; -a, a\right), \]
where $B(x; p, q)$ denotes the incomplete beta function. Note that this 
expression involves a negative second parameter $a < 0$, which requires 
careful interpretation beyond the standard definition.

For $a>0$, this derivation fails due to the excessive singularity at $u_1$. 
Instead, one must use a contour that extends rightward from $u_0$ rather 
than leftward. We omit these details here.

\paragraph{Application to random almost triangle Young tableaux}
Let $\big( \lambda^{(n)} \big)$ be a random or deterministic sequence 
satisfying the convergence conditions \eqref{eq:convergence-rows}, \eqref{eq:convergence-columns} from \cref{sec:almost-isosceles}, where 
the auxiliary functions are adapted to parameterize the triangle diagram $\Omega$:
\begin{align*}
	X(y) &= \max\left( \frac{b - (1 + a)y}{1 - a}, \ 0 \right), \\
	Y(x) &= \max\left( \frac{b - (1 - a)x}{1 + a}, \ 0 \right).
\end{align*}
For a concrete example, one may take $\lambda^{(n)}$ to be the deterministic 
diagram with rows given by
\[  \lambda^{(n)}_i = \left\lfloor  \sqrt{n}\ X\left( \frac{i}{\sqrt{n}} 
\right)  \right\rfloor,  \]
which provides the most straightforward discretization of the diagram 
$\Omega$.

\cref{thm:CLT-shape} is applicable in this setting and gives the following concrete statement.

\begin{corollary}[Schensted insertion into almost triangle tableaux]
	Assume $-1<a<0$ and let $\big( \lambda^{(n)} \big)$ be a sequence satisfying 
	the above assumptions. Let $u_0 \in (u_1, u_2)$ and define
	\[ z_0 = I_{\frac{u_0 - u_1}{u_2 - u_1}}\left( \frac{1-a}{2}, 
	\frac{1+a}{2} \right) \in (0,1). \]
	Let $T^{(n)}$ be a uniformly random Poissonized tableau of shape 
	$\lambda^{(n)}$. Then the fluctuations of the newly created box around 
	the first-order approximation satisfy
	\[
	\sqrt[4]{n} \left[ \frac{\uIns( T^{(n)}; z_0 )}{\sqrt{n}} - u_0 \right]
	\xrightarrow[n\to\infty]{\mathcal{D}} N\left( 0, \sigma^2 \right),
	\]  
	where the variance is given by
	\[ \sigma^2 = -\pi \sin \frac{\pi a}{2} \cdot (u_0 - u_1)^{1+a} 
	(u_2 - u_0)^{1-a} \cdot \frac{1}{u_2-u_1} \cdot 
	B\left(\frac{u_0-u_1}{u_2-u_1}; -a, a\right). \]
\end{corollary}

\subsubsection{Where to look for more concrete examples?}

Several additional natural candidates for the \conti diagram $\Omega$ arise for which 
explicit calculations are feasible, allowing \cref{thm:CLT-shape} to take 
concrete form.

The most immediate example is the Logan--Shepp--Vershik--Kerov limit shape
$\Omega_*$ from \cref{sec:lsvk}. For sequences of Young diagrams converging 
to this \conti diagram at the appropriate rate, our framework applies 
directly. Moreover, all necessary quantities---the density of the measure, 
CDF, and interaction energy---are explicitly known. This case will be developed 
in detail later, in \cref{sec:towards-nice-d-conjecture}. While that section 
primarily addresses \cref{conj:CLT-Plancherel} for Plancherel random 
diagrams, the same methods apply to any sequence satisfying 
Assumption~\ref{item:A2}, yielding a concrete realization of our results.

Another promising family consists of the \conti diagrams introduced by Biane
\cite{Biane2001} as limit shapes arising from the RSK correspondence applied to
random sequences of letters sampled from a finite alphabet. These diagrams have
explicit parametrizations and their transition measures are explicitly given by
free Poisson distributions, making them particularly tractable for concrete
calculations with a comparable level of difficulty as above.

More generally, any \conti diagram with an explicitly known transition 
measure that avoids exotic special functions would serve as a suitable 
test case for our methods.

\subsection{Towards \cref{conj:CLT-Plancherel}, part 1}
\label{sec:towards-big-conjecture}

\subsubsection{General strategy}
\label{sec:strategy}

One might attempt to prove \cref{conj:CLT-Plancherel} by applying
\cref{thm:CLT-shape} with $\lambda^{(n)}$ being a Plancherel-distributed random
Young diagram with $n$ boxes and $\Omega=\Omega_*$ being the
Logan--Shepp--Vershik--Kerov limit shape from \cref{sec:lsvk}.

\subsubsection{Problematic \assum{item:A2}}

Regretfully, we do not know whether \assum{item:A2} is satisfied in
this context. Some naive Monte Carlo simulations reveal that for diagram sizes $n$
in the range $10^2$ to $10^6$, spanning several orders of magnitude, the mean
distance
\begin{equation}
	\label{eq:conjectured}
	 \E\ d_{XY}(\omega_{n}, \Omega_* ) \approx C n^{\alpha}
\end{equation}
between the rescaled Plancherel-distributed random \conti diagram $\omega_{n}$ and the
limit shape~$\Omega_*$ can be well approximated by a power law with the exponent
\begin{equation}
	\label{eq:alpha}
	\alpha\approx -0.32 
\end{equation}
and the multiplicative constant $C\approx 2$. For this reason we dare to state
the following conjecture which is essentially a reformulation of the condition
\ref{item:A2}. %
\begin{conjecture}
	\label{conj:nice-d}
There exists an exponent $\alpha \leq - \frac{1}{4}$ with the property that if
$\omega_n = \omega_{\frac{1}{\sqrt{\nklatki}}\lambda^{(\nklatki)}}$ is the
rescaled  profile of a Plancherel-distributed random Young diagram
$\lambda^{(n)}$ with $n$ boxes then
\[ n^{-\alpha}\ d_{XY}(\omega_{n}, \Omega_* ) \xrightarrow[n\to\infty]{P} 0. \]
\end{conjecture}

\subsubsection{Towards \cref{conj:nice-d}}
\label{sec:towards-nice-d-conjecture}

Let $X=X(y)$ denote the parametrization of the limit curve $\Omega_*$ (given by
\eqref{eq:omegastar}) in the French coordinate system. This function associates
each $y \geq 0$ with its corresponding $x$-coordinate. By convention, $X(0)=2$ so that
$X\colon [0,\infty) \to [0,2]$. By a minor adjustment of \cref{rem:dX-concrete}
it follows that the \mbox{$x$-distance} $d_X(\omega_n, \Omega_*)$ is the
supremum distance between the function $X$ and the piecewise constant function
which to the French $y$ coordinate associates the corresponding
\mbox{$x$-coordinate} of the right envelope of the rescaled diagram
$\lambda^{(n)}$. As a consequence,
\begin{equation}
	\label{eq:dX-toOmega}
	d_X(\omega_n, \Omega_*) = 
\max \left( \max_{i\geq 1} 
 \left| \frac{1}{\sqrt{n}} \lambda^{(n)}_i - X\left(\frac{i-1}{\sqrt{n}} \right)  \right|, 
 \quad
 \max_{i\geq 1} 
  \left| \frac{1}{\sqrt{n}} \lambda^{(n)}_i - X\left(\frac{i}{\sqrt{n}} \right)  \right|
  \right).
\end{equation}

The results of Baik, Deift and Johansson \cite{BaikDeift1999} and
Okounkov \cite{Okounkov2000} imply that each finite family of (shifted and rescaled)
lengths of the bottom rows
\[ n^{-\frac{1}{6}} \left( \lambda^{(n)}_i - 2 \sqrt{n} \right) =
n^{\frac{1}{3}} \left( \frac{1}{\sqrt{n}} \lambda^{(n)}_i - 2  \right)
 \]
indexed by $i\in \{1,\dots,i_{\max}\}$ converges in distribution to a specific
non-degenerate probability distribution (which happens to be related to the
largest $i_{\max}$ eigenvalues of a large GUE random matrix). As a consequence,
each individual component of the iterated maximum \eqref{eq:dX-toOmega} is a
random variable which asymptotically is of order \emph{at least}
$n^{-\frac{1}{3}}$. By the symmetry of the problem, an analogous analysis can be
performed for the $y$-distance $d_Y(\omega_n, \Omega_*)$ which gives a lower
bound $\alpha\geq - \frac{1}{3}$ for the exponent in the conjectural asymptotic
formula \eqref{eq:conjectured}. It also shows that \cref{conj:nice-d} cannot be
true for $\alpha \leq  - \frac{1}{3}$.
 
Interestingly, this theoretical lower bound  $\alpha\geq - \frac{1}{3}$ is very
close to the value \eqref{eq:alpha} based on the Monte Carlo experiments. This
could be an indication that the \emph{`main contribution'} to the maximum in
\eqref{eq:dX-toOmega} comes from a \emph{`small'} (finite, or growing very
slowly with $n$) number of components there which correspond to the bottom rows
of the Young diagram. Monte Carlo experiments related to Plancherel-distributed
random Young diagrams with $n=10^6$ boxes give some support to this heuristic
statement. Regretfully, any rigorous results related to \cref{conj:nice-d} with
$\alpha\approx -\frac{1}{3}$ are currently beyond our reach.

\smallskip

The seminal results of Logan and Shepp \cite{LoganShepp1977} and Vershik and
Kerov \cite{VershikKerov1977} can be interpreted as a rudimentary version of
\cref{conj:nice-d} with exponent $\alpha=0$. To explore potential improvements,
we revisit their original proofs.

Using Romik's notation \cite{Romik2015}, we aim to refine
\cite[Theorem~1.26]{Romik2015} by replacing the fixed positive constant
$\epsilon>0$ with a sequence $(\epsilon_n)$ defined as:
\[ \epsilon_n = n^{\alpha} \]
where $\alpha<0$. Our goal is to maintain the theorem's conclusion (that ``the
probability of a specific event converges to 1'') while minimizing $\alpha$.
This modification must be propagated through the entire supply chain of
supporting proofs.

However, a critical component of this chain, \cite[Theorem~1.20]{Romik2015},
relies on the condition:
\[ - \epsilon_n^2 n + C \sqrt{n} + O\left( \log n \ \sqrt{n} \right) 
\xrightarrow[n\to\infty]{} -\infty \]
for some constant $C>0$. This condition holds only if $\epsilon_n$ does not
decay too rapidly, specifically when $\alpha > -\frac{1}{4}$.

Consequently, we can prove that \cref{conj:nice-d} holds for $\alpha >
-\frac{1}{4}$. While this represents an improvement over the original result, it
falls short of our target $\alpha = -\frac{1}{4}$, highlighting the limitations
of this approach.

\smallskip

We shall discuss some alternative versions of the \assum{item:A2} as
well as some other paths towards \cref{conj:jdt} later in
\cref{sec:towards-conj-jdt-part-2}.

\subsubsection{Transition measure for the limit shape and its interaction energy}

Our quest towards \cref{conj:CLT-Plancherel} via the strategy outlined in
\cref{sec:strategy} involves also identification of the transition
measure~$\mu_{\Omega_*}$ of the limit shape $\Omega_*$. By the result of Kerov
\cite{Kerov1993-transition} it is the semicircle measure $\mu_{\SC}$ with the density
$f_{\SC}$ given by \eqref{eq:SCdensity}; it is no coincidence that this measure
already appeared in \cref{sec:semicircle}. The following result provides an
explicit formula for its interaction energy.

\begin{proposition}
	\label{prop:semicircle-interaction-energy}
    For the semicircle measure $\mu_{\SC}$ the corresponding interaction energy is given by
    \[
    \Force_{\SC}(u)= 
    \begin{cases}
        \displaystyle \frac{1}{12 \pi} \left(4-u^2 \right)^{\frac{3}{2}} & \text{for $-2\leq u\leq 2$}, \\[2ex]
        0 & \text{otherwise.}
    \end{cases}
    \]    
\end{proposition}
\begin{proof}
    The method of the contour integral considered in \cref{sec:may-force-be-with-you}
    is applicable; we leave the details for the interested reader. In the following
    we provide an alternative method.

    For $\epsilon\in\R$ we define the regularized version of
    $\Force_{\SC}(u)$ given by
    \[
    \Force_\epsilon(u):= \iint_{-2 < z_1 < u < z_2 < 2} \Re \frac{1}{z_2-z_1+i \epsilon} 
    \ f_{\SC}(z_1)\ f_{\SC}(z_2)
    \dif z_1 \dif z_2;\]
    our goal is to evaluate $\Force_0(u)=\Force_{\SC}(u)$.
    Note that for real numbers $z_2 > z_1$
    \[ \R_+\ni \epsilon \mapsto \Re \frac{1}{z_2-z_1+ i \epsilon} = \frac{z_2-z_1}{(z_2-z_1)^2+\epsilon^2} \]     
    is a decreasing, positive, continuous function. 
    By Lebesgue monotone convergence theorem it follows that
    \[ F_0(u) = \lim_{\epsilon\to 0}  F_\epsilon(u). \]

    In order to evaluate $F_\epsilon(u)$ we notice that for $\epsilon>0$ the
derivative of $\Force_\epsilon$ is given by
    \begin{multline*} \Force_\epsilon'(u) = \\
        \int_{u<z_2<2} \Re \frac{1}{z_2-u+i \epsilon}\ f_{\SC}(u)\ f_{\SC}(z_2) \dif z_2 - 
        \int_{-2<z_1<u} \Re \frac{1}{u-z_1+i \epsilon}\ f_{\SC}(z_1)\ f_{\SC}(u) \dif z_1
        = \\
        f_{\SC}(u)\ \Re \int_{-2<z<2}   \frac{1}{z-(u+i \epsilon)} \ f_{\SC}(z) \dif z
        =  f_{\SC}(u) \ \Re \cauchy_{SC}(u+i \epsilon);
    \end{multline*}
    above we used the fact that the integral on the right-hand side is the Cauchy
    transform of the semicircular distribution
    \[ \cauchy_{\SC}(w) = \int_{-2<z<2}   \frac{1}{z-w} \ f_{\SC}(z) \dif z = 
    \frac{w - \sqrt{w^2-4}}{2},\] 
    evaluated at $w=i+i\epsilon$, see \cite[Section 3.1]{MingoSpeicher2017}.

    For $-2\leq u \leq 2$ the principal value of the Cauchy integral is given by
    \[ \cauchy_{\SC}(u)= \lim_{\epsilon\to 0} \Re \cauchy_{\SC}(u+i \epsilon) = - \frac{u}{2} \]
    and the convergence is uniform over the interval $[-2,2]$. 
    We proved in this way that
    \begin{equation}
        \label{eq:GB}
        \Force_{\SC}'(u) = f_{\SC}(u)\ \cauchy_{\SC}(u).
    \end{equation}    
    It follows that
    \[
    \Force_{\SC}(u)= \int_{-2}^u f_{\SC}(w) \frac{-w}{2} \dif w=
    \frac{1}{12 \pi} \left(4-u^2 \right)^{\frac{3}{2}} \qquad \text{for $-2\leq u\leq 2$}. 
    \]
    
    The above calculation can be generalized to other measures than the
semicircular law for which the Cauchy transform is sufficiently regular. Be
warned, that for the example of the arc-sine law from \cref{sec:staircase}
the analogue of \eqref{eq:GB} implies that
    \[ \Force_{\AS}'(u)= 0 \]
    which, technically speaking, is correct but not very helpful for finding the
    exact form of the formula \eqref{eq:arc-sine-force}.
\end{proof}

\subsubsection{Conclusion}

One can now easily verify that, apart from the condition \ref{item:A2}, the
remaining assumptions of \cref{thm:CLT-shape} are fulfilled for the setup
outlined in \cref{sec:strategy}.
Our choice of the variance \eqref{eq:varianceCLT} in \cref{conj:CLT-Plancherel}
is equal to $\frac{\Force_{\SC}}{(f_{\SC})^2}$ and it was based on the the variance
on the right-hand side of \eqref{eq:my-favorite-rv}.

\section{A version of the main theorem with weaker assumptions}
\label{sec:main-thm-via-transmeasure}

This section focuses on \cref{thm:CLT-trans}, which serves as both a stronger
version of the main theorem (\cref{thm:CLT-shape}) and an essential intermediate
step in its proof. We begin by discussing the motivations behind this result.

\subsection{Motivations}

\Assumptions \ref{item:A2}, \ref{item:A4} of the main result
(\cref{thm:CLT-shape}) were formulated in the language that involves the
\emph{`shape'} of the \conti diagrams $(\omega_n)$ as well as the \emph{`shape'}
of their limit $\Omega$. This language provides a very convenient intuitive
meaning to the assumptions, but it also has some disadvantages which we discuss
below.

\subsubsection{Tension between local and global} 

Firstly, \assum{item:A2} may be overly restrictive for certain
applications, as discussed in \cref{sec:towards-big-conjecture}. It would be
advantageous to replace it with a weaker condition. The key issue is that the
distance $d_{XY}(\omega_n, \Omega)$ can be viewed as a form of the supremum
distance, which is influenced by the \emph{‘worst behaving’} part of the \conti
diagram $\omega_n$. This influence persists even if the problematic part
constitutes a relatively \emph{‘small’} fraction of the Young diagram profile or
is located \emph{‘far away’} from where the new box is expected to appear. This
phenomenon is evident in Plancherel-distributed random Young diagrams, where
fluctuations at the edge are significantly larger than those in the bulk.

For this reason, it would be desirable to replace \ref{item:A2} by some
assumptions that fall into the following two types.
\begin{itemize}
\item The assumptions of the first type should be a quite restrictive about the
\emph{`local fluctuations'} of the \conti diagrams $(\omega_n)$ in the
neighborhood of the point $u_0$ where we expect the new box to appear. 

\item The assumptions of the second type should provide quite rough information
about the overall global shape of the diagrams $(\omega_n)$.
\end{itemize}

As we shall see, each assumption of \cref{thm:CLT-trans} is
expressed in terms of the transition measure of $\omega_n$. One can argue that
each of them is a mixture of these two types. Indeed, on one hand, each of these
assumptions is very sensitive to the local behavior of the transition measure of
$\omega_n$ (which itself is most sensitive to the local behavior of the \conti
diagram $\omega_n$). On the other hand, the transition measure is defined in a
very \emph{`non-local'} way so that its behavior around $u_0$ is influenced also
by very \emph{`distant'} areas of $\omega_n$.

We believe that the assumptions of \cref{thm:CLT-trans} strike the right
balance between the local and the global conditions that allows to apply the
main result in a wider class of examples.

\subsubsection{Transition measure as the language of the proof}

Secondly, the vast majority of the proof of \cref{thm:CLT-shape} is formulated
purely in the language of the transition measures of the \conti diagrams
$(\omega_n)$ and not in the language that involves their shapes. In fact, the
first step of the proof (\cref{prop:implies}) is to translate the assumptions of
\cref{thm:CLT-shape} to the language of the transition measures and to never
look back again.

\subsection{Notations}
\label{sec:notations}

If $\mu$ is a probability measure on the real line, $u_0\in\R$ and $\epsilon\geq
0$ we define \emph{the regularized interaction energy}
\[
	\ForceRegularized{\epsilon}{\mu}(u_0) = 
	\iint\limits_{\{(z_1,z_2) :\  z_1 \leq u_0 < z_2 \}}
	\frac{1}{z_2-z_1 + \epsilon} \dif \mu(z_1) \dif \mu(z_2).
\]
Note that for $\epsilon=0$ we recover the usual interaction energy which was
introduced in \cref{sec:may-force-be-with-you}.

For a probability measure $\mu$ on the real line we define 
\begin{equation}
	\label{eq:cauchy-abs}
	 \cauchy^+_\mu(z) = \int_{\R} \frac{1}{|z-x|+1} \dif \mu(x) 
\end{equation}
as a (slightly regularized) version of the Cauchy transform with the absolute
value of the original kernel. For a Young diagram $\lambda$ we will use the
simplified notation
\[\cauchy^+_\lambda= \cauchy^+_{\mu_\lambda}. \]

Recall that $K_\lambda$ is the cumulative function of the diagram $\lambda$, see
\cref{label:CDFkerov}.

\subsection{The main theorem, alternative version}

\begin{theorem}[The main theorem with alternative, weaker assumptions]
	\label{thm:CLT-trans}

	Let $u_0\in\R$, and $0<z_0<1$,  and $f_0>0$, and $\Force_0>0$ be fixed.     
	For each integer $\nklatki \geq 1$ let
	$\lambda^{(\nklatki)}$ be a random Young diagram. We denote by 
	\[ \omega_n= \omega_{\frac{1}{\sqrt{n}} \lambda^{(n)} } \]
	the corresponding rescaled random \conti diagram.

	Let $\zmienna$ be an standard Gaussian variable which is independent from
the random diagrams $(\lambda^{(n)})$. We denote
\begin{equation}
	\label{eq:uwithnoise}
	 \uwithnoise = u_0 + \frac{1}{\sqrt[4]{n}} \zmienna. 
\end{equation}

	We assume that:
	\begin{enumerate}[label=(B\arabic*)]
		\item  \label{item:local2}
		\emph{(`the cumulative functions can be approximated locally by an affine function')} 
		\[ \sqrt[4]{\nklatki} \Big[
		K_{\omega_n} (\uwithnoise)
		-  z_0 - f_0 \cdot \left(  \uwithnoise - u_0 \right)   \Big]
		\xrightarrow[n\to\infty]{P} 0, \]

\item \label{item:energy2} 
\emph{(`the regularized interaction energy can be
	locally approximated by $\Force_0$')}
\[
\ForceRegularized{\frac{1}{\sqrt{n}}}{\mu_{\omega_n}}(\uwithnoise)
\xrightarrow[n\to\infty]{P} \Force_0,
\]

		\item 
		\label{item:regular2}
		\label{item:last-assumption}
		\emph{(`the integrand within Cauchy transform is not too singular')}
		\begin{equation}
			\label{eq:regular2}
			\nklatki^{\frac{3}{8}}  \cauchy^+_{\lambda^{(n)}}(\sqrt{n} \ \uwithnoise)  
								\xrightarrow[n\to\infty]{P} 0.
		\end{equation}
	\end{enumerate}

	\medskip

		Let $T^{(\nklatki)}$ be a uniformly random Poissonized tableau of shape $\lambda^{(\nklatki)}$.
	Then the fluctuations of the newly created box around the first-order approximation
	\begin{equation}\label{eq:my-favorite-rv2}
		\sqrt[4]{\nklatki} \left[ \frac{\uIns( T^{(\nklatki)}; z_0 )}{\sqrt{\nklatki}} -u_0 \right]
		\xrightarrow[\nklatki\to\infty]{\mathcal{D}} N\left( 0, \frac{\Force_0}{f_0^2} \right)
	\end{equation}    
	converge to a centered Gaussian measure with the variance
	$\frac{\Force_0}{f_0^2}$.

\end{theorem}
The proof is postponed to \cref{sec:proof-thm-CLT-general}.

\begin{remark}
	\label{rem:spreading}

\cref{thm:CLT-trans} makes \assumptions
\ref{item:local2}--\ref{item:last-assumption} in terms of the random variable
$\zmienna$. A reader might initially assume that these assumptions hold true for
all possible values of $\zmienna\in\R$. However, it is crucial to understand
that the theorem's assumptions are much weaker than that. The key point is that
the \assumptions \ref{item:local2}--\ref{item:last-assumption} need not hold for
every single value of $\zmienna$. Instead, the theorem allows for the existence
of \emph{`bad'} values for which these assumptions may not be satisfied. By
considering $\zmienna$ as a random variable and requiring convergence in
probability, the theorem provides a convenient framework to state that such
`bad' values do not dominate the overall behavior.
\end{remark}

\begin{remark}
The reader may question the assumption that the random variable $\zmienna$
follows the Gaussian distribution. It is
important to note that there is flexibility in this assumption, and the
probability distribution of $\zmienna$ can be modified. The key requirement is
that the distribution of $\zmienna$ is absolutely continuous, with a density
that is strictly positive over the entire real line. If the result holds true
for one probability distribution satisfying this property, it will hold for any
other distribution with the same property. The details of this generalization
are left for the reader to explore. 

However, the proof of the implication from \cref{thm:CLT-shape} to
\cref{thm:CLT-trans} is more straightforward 
if the density converges to zero in $\pm\infty$.
This property motivated the choice of the Gaussian distribution.
\end{remark}

\subsection{Towards \cref{conj:jdt}, part 2}
\label{sec:towards-conj-jdt-part-2}

Building upon our discussion in \cref{sec:towards-big-conjecture}, we now
revisit our goal of proving \cref{conj:jdt}. Our previous attempt relied on
\cref{thm:CLT-shape}. In this section, we present a refined approach using
\cref{thm:CLT-trans}, which offers the advantage of weaker assumptions.

\cref{thm:CLT-trans} presents two key assumptions, among which assumption
\ref{item:local2} emerges as the most challenging. This critical condition
requires that the \CDF of the transition measure for a random
Plancherel-distributed Young diagram exhibits fluctuations within a prescribed,
sufficiently small rate.

Our investigation into this condition has yielded promising insights. Monte
Carlo experiments, as illustrated by the red thin curves in
\cref{fig:CDF-100,fig:CDF-10000}, suggest that this condition is indeed
satisfied. Furthermore, analogous results in random matrix theory, particularly
those of Gustavsson \cite{Gustavsson}, provide additional support for this
hypothesis. Based on these observations, we propose a stronger conjecture:
\begin{conjecture}
	\label{conj:transition}
	Let $\lambda^{(n)}$ be a random Young diagram with $n$ boxes distributed 
	according to the Plancherel measure. For each $-2< u <2$
	and each $\alpha<\frac{1}{2}$
	\[ n^{\alpha} \left[ 
	                      K_{\lambda^{(n)}} \left( \sqrt{n}\ u \right) - F_{\SC}(u)  
	              \right] 
	\xrightarrow[n\to\infty]{P} 0.\]
\end{conjecture}
This conjecture encompasses and extends beyond the requirements of 
\assum{item:local2}. The case $\alpha=0$ corresponds to the established weak
convergence of probability measures \cite[Theorem~3.2]{RomikSniady2015}, while
\assum{item:local2} aligns with the case $\alpha=\frac{1}{4}$.

While \cref{conj:transition} remains beyond our current reach, it offers a
tantalizing glimpse into the deeper structure of Plancherel-distributed Young
diagrams and their transition measures.

\section{Implication between two versions of the main theorem}

\label{sec:old-version-implies-new-version}

Probably the best way of giving an intuitive meaning of the assumptions of
\cref{thm:CLT-trans} is to establish a link to the main result
(\cref{thm:CLT-shape}). This link is provided by the following proposition,
which also serves as an intermediate step for the proof of the main result.

\begin{proposition}
	\label{prop:implies}
	\cref{thm:CLT-trans} implies \cref{thm:CLT-shape}.
\end{proposition}
\begin{proof}
Both theorems share the same conclusion. Therefore, to establish the
implication, it suffices to show that the assumptions of \cref{thm:CLT-shape}
entail those of \cref{thm:CLT-trans}. We begin by assuming that the hypotheses
of \cref{thm:CLT-shape} hold.  The subsequent sections are
dedicated to verifying under these conditions each assumption of
\cref{thm:CLT-trans} in turn.
\end{proof}

\subsection{Assumptions of \cref{thm:CLT-shape} imply \assum{item:local2}}
Let us choose the constants $a_0$ and $b_0$ in such a way that
\[ a< a_0 < u_0 < b_0 < b. \]
The assumptions of \cref{thm:CLT-shape}
were fine-tuned in such a way that a recent result of \sniady
\cite[Theorem~1.4]{Sniady2024modulus} is applicable with the exponent $\alpha=
\frac{1}{4}$; as a consequence
\begin{equation}
	\label{eq:modulus}
	\sqrt[4]{n}\ \sup_{z\in [a_0,b_0]} 
	   \left| K_{\lambda^{(n)}}\left(\sqrt{n}\ z\right) - K_{\Omega}(z) \right| 
		\xrightarrow[n\to\infty]{P} 0
\end{equation}
and \assum{item:local2} follows easily from the Taylor expansion of
the cumulative function:
\[ K_\Omega\left(u_0 +  \frac{c}{\sqrt[4]{n}} \right) = 
   z_0 +  \frac{f_0 c}{\sqrt[4]{n}} + o\left( \frac{c}{\sqrt[4]{n}}\right). \]

\subsection{\Assum{item:energy2}. The lower bound}

For a given probability measure $\mu$ on the real line, the interaction energy
$\Force_{\mu}(u_0)$ (cf.~\eqref{eq:interaction-enerrgy}) can be viewed as an
integral with respect to the product measure $\mu\times\mu$ of the function
\[ \phi(z_1,z_2)= \frac{1}{z_2-z_1}
          \ \indicator_{(-\infty,u_0]}(z_1)\ \indicator_{(u_0,\infty)}(z_2)
\]
which has a singularity at $z_1=z_2=u_0$.
Our strategy is to approximate this function by some more regular functions.

Let us fix $\epsilon>0$. By Assumption
\ref{item:transition-measure-omega-is-absolutely-continuous} about the local
behavior of the transition measure~$\mu_{\Omega}$ there exists $\delta>0$ and a
non-negative continuous function $g\colon \R^2 \to \R_+$ with the following
three properties:
\begin{itemize}
	\item \emph{upper bound:} 
	\[ g \leq \phi, \]
	
	\item \emph{support is away from the singularity:} 
	\[ \text{if } g(z_1,z_2) > 0 \text{ then } z_1< u_0 - \delta \text{ and } 
	z_2> u_0 + \delta,\]
	
	\item \emph{the total mass is almost maximal:}
	\[ \iint g(z_1, z_2) \dif \mu_{\Omega}(z_1) \dif \mu_{\Omega}(z_2) > 
	   \Force_{\mu_{\Omega}}(u_0) - \epsilon. \]  
\end{itemize}

Our objective in this subsection is to demonstrate that the sequence of probabilities
\begin{equation}
	\label{eq:interaction-lowerbound}
	\Pro\left[\ForceRegularized{\frac{1}{\sqrt{n}}}{\mu_{\omega_n}}(\uwithnoise) >
	\Force_0 - \epsilon\right] \xrightarrow[n\to\infty]{} 1
\end{equation}
converges to $1$. To achieve this, we employ a conditioning argument on the event
\[ 
	\uwithnoise \in [u_0 - \delta, u_0 + \delta]  
.\] 
Since the probability of the complement of this event converges to
zero as $n \to \infty$, we can assume without loss of generality that
$\uwithnoise \in [u_0 - \delta, u_0 + \delta]$ holds almost surely.

By comparing the analogue of the kernel $\phi$ for the regularized interaction
energy with the kernel $g$, we obtain the inequality
\begin{equation}
	\label{eq:friday-inequality}
	\ForceRegularized{\frac{1}{\sqrt{n}}}{\mu_{\omega_n}}(\uwithnoise) \geq 
	\frac{2\delta}{2\delta + \frac{1}{\sqrt{n}}}
	\iint g(z_1,z_2) \, \mathrm{d}\mu_{\omega_n}(z_1) \, \mathrm{d}\mu_{\omega_n}(z_2).
\end{equation}
Note that the numerical factor on the right-hand side converges to 1 as $n \to \infty$.

Kerov's results (\cite{Kerov1993-transition} and \cite[Chapter~4,
Section~1]{KerovBook}) imply that \assum{item:A2} is sufficient to
conclude that the sequence of random measures $(\mu_{\omega_n})$ converges to
$\mu_{\Omega}$ in the weak topology of probability measures, in probability.
Consequently, the double integral on the right-hand side of 
\eqref{eq:friday-inequality} converges in probability:
\begin{equation}
	\label{eq:integral-convergence}
	\iint g(z_1,z_2) \, \mathrm{d}\mu_{\omega_n}(z_1) \, \mathrm{d}\mu_{\omega_n}(z_2) 
	\xrightarrow[n\to\infty]{\Pro}
	\iint g(z_1,z_2) \, \mathrm{d}\mu_{\Omega}(z_1) \, \mathrm{d}\mu_{\Omega}(z_2).
\end{equation}

Combining \eqref{eq:friday-inequality} and \eqref{eq:integral-convergence}, we
conclude that \eqref{eq:interaction-lowerbound} holds true, completing the proof
of the lower bound in \ref{item:energy2}. The proof of the upper bound is
postponed to \cref{sec:B2-upper-bound}.

\subsection{Bounding the interaction energy from above}
\label{sec:bounding-the-interaction}

We present an intermediate result crucial for proving that the upper bound in
\assum{item:energy2} holds.

\begin{proposition}
	\label{prop:energy-via-cdf} 
	
	Let constants $a<u_0<b$ be such that the probability measure $\mu$ restricted
to the interval $[a,b]$ is absolutely continuous, with a density given by a
bounded function which is bounded below by a positive constant. We denote
	\[z_0= F_\mu(u_0), \qquad 
	\Force_0 = \Force_{\mu}(u_0).\]

	For any $\epsilon > 0$
	and any open interval $J$ containing $z_0$, there exists an open neighborhood
	$G$ of $\mu$ in the topology of weak convergence of probability measures,
	along with a natural number $n_0 \in \mathbb{N}$ such that the following holds:
	
	For every $n \geq n_0$ and for any probability measure $\nu$ on the real line
	that satisfies both of the following conditions:
	\begin{itemize}
		\item	\emph{global proximity condition (``$\nu \approx \mu$''):}
		\begin{equation}
			\label{eq:global-proximity}
			\nu \in G,
		\end{equation}	
		\item	\emph{local refinement condition (``the \CDFs of $\nu$ and
			$\mu$ are locally very close''):}
		\begin{equation}
			\label{eq:local-refinement-cdf}
			\left|F_{\nu}(u) - F_\mu(u)\right| < \frac{1}{\sqrt[4]{n}} \qquad \text{for each $u \in J$},
		\end{equation}
	\end{itemize}
	the following bound is satisfied:
	\begin{equation}
		\label{eq:hard-energy-new}
		\E \left[\Force_{\nu}(\uwithnoise)\right] < \Force_0 + \epsilon.	
	\end{equation}
\end{proposition}

The remaining part of this subsection is devoted to its proof.

\subsubsection{The expected value of $\Force_{\nu}(\uwithnoise)$ as a double integral.
Decomposition of the integration domain}

The expected value on the left-hand side of \eqref{eq:hard-energy-new} can be
written as:
\begin{equation}
	\label{eq:energy-find-upper-bound-new}
	\E \left[ \Force_{\nu}(\uwithnoise) \right] = 
	\iint_{z_1 < z_2}  
	\frac{ \Pro\left[ \uwithnoise \in [z_1,z_2]   \right] }{z_2- z_1 }
	\dif \nu(z_1) \dif \nu(z_2) 
	. 
\end{equation}

Let \( I \) be an open interval containing \( z_0 \), such that its closure \(
\overline{I} \) is contained within \( J \). The necessity of introducing
another open interval, in addition to \( J \), will become clearer in the proof
of \cref{lem:close-to-singularity}.

We will divide the integration region in the integral on the right-hand side of
\eqref{eq:energy-find-upper-bound-new} into a number of (non-disjoint) parts,
and analyze the contribution of each part separately. This division follows the
following scheme which forms a tree:
\begin{itemize}
	\item The small macroscopic region where $(z_1,z_2)\in I\times I$ and we have very good
control over the \CDF of~$\nu$. This region splits further:
	\begin{itemize}
		\item the part where 
		\[ z_2-z_1 \geq  \frac{\log n}{\sqrt[4]{n}} \] 
		and the denominator on the right-hand side of
\eqref{eq:energy-find-upper-bound-new} is well-separated from the singularity.
We will examine the contribution of this area in
\cref{lem:close-to-singularity,lem:rhs-is-small}. This scenario represents one
of the two cases where the limit contribution is non-zero for $n\to\infty$.
However, it is important to note that this contribution approaches zero in the
iterated limit, as $n\to\infty$ and the length of the interval $I$ tends to
zero.

		\item the part where 
		\[ 0< z_2-z_1< \frac{\log n}{\sqrt[4]{n}} \]
		and the coordinates are close to each other. 
			This part splits into three further (non-exclusive) possibilities:
			\begin{itemize}
				\item 
				Both coordinates are not too far from $z_0$: 
\begin{equation}
	\label{eq:both-not-too-far}
	|z_1- z_0|,\ |z_2- z_0| \leq \frac{1}{\sqrt[8]{n} \log n}. 
\end{equation}
				We will investigate the contribution of this area in
\cref{lem:very-close-to-z0} and we will show that it converges to zero.
				
    \item Both coordinates are on the same side of $z_0$ and well separated from $z_0$, i.e., one of the following two cases holds true:
\begin{equation}
	\label{eq:on-the-same-side}
	z_1, z_2 > z_0 + \frac{\log n}{\sqrt[4]{n}} \quad \text{or} \quad 
	z_1, z_2 < z_0 - \frac{\log n}{\sqrt[4]{n}}.
\end{equation}
We will investigate the contribution of the first case in
\cref{lem:close-to-each-other-away-from-z0} and show that it converges to
zero. The analysis for the second case is analogous and will be omitted.

\item The third option, which occurs when neither condition
\eqref{eq:both-not-too-far} nor \eqref{eq:on-the-same-side} is satisfied, is not
possible for sufficiently large $n$ for which
\[ \frac{2 \log n}{\sqrt[4]{n}} < \frac{1}{\sqrt[8]{n} \log n}. \]

Indeed, let us say that it is $z_2$ that is far from $z_0$ so that
\eqref{eq:both-not-too-far} is false.
The case when $z_2$ is on the left-hand
side of $z_0$, so that:
\[ z_1< z_2 < z_0 - \frac{1}{\sqrt[8]{n} \log n} \]
is not possible as it would contradict that the condition \eqref{eq:on-the-same-side} is not satisfied.

Thus $z_2$ is on the right-hand side of $z_0$ and 
\[ z_2 > z_0 + \frac{1}{\sqrt[8]{n} \log n}. \]
Since condition \eqref{eq:on-the-same-side} is not satisfied, it follows that:
\[		
z_1 \leq z_0 + \frac{\log n}{\sqrt[4]{n}}.
\]	
This leads to a contradiction:
	\[
		\frac{\log n}{\sqrt[4]{n}} > z_2 - z_1 > 
		\left(z_0 + \frac{1}{\sqrt[8]{n} \log n}\right) -
		\left(z_0 + \frac{\log n}{\sqrt[4]{n}}\right) \gg \frac{\log n}{\sqrt[4]{n}}
	\]
	for sufficiently large values of $n$.
Therefore, this case is not possible for $n$ above a certain threshold.	
				
			\end{itemize}
		
	\end{itemize}

	\item The complement of $I\times I$,  i.e., $(z_1,z_2)\notin I\times I$.
	This region splits further:
	\begin{itemize}
		\item 
		the part where 	the coordinates are on the opposite sides of $z_0$:
		\[ z_1< z_0< z_2.\]
We will investigate the contribution of this area in
\cref{lem:close-to-singularity-opposite-sides}. This is the
region which provides the main asymptotic contribution.

		\item and the one where the coordinates
are the same side of $z_0$,
say 
\[ z_0 < z_1, z_2.\]
(the other case when $z_1,z_2<z_0$ follows a similar pattern and we skip it).
There are two additional cases:
\begin{itemize}
	\item $z_1$ and $z_2$ are well separated from $z_0$ and the first inequality in
\eqref{eq:on-the-same-side} holds true. As we already mentioned,
\cref{lem:close-to-each-other-away-from-z0} is applicable in this case; 
it shows that the contribution of this area converges to zero.
	
	\item $z_1$ is very close to $z_0$ while $z_2$ is very far:
	\[ z_0< z_1< z_0 + \frac{\log n}{\sqrt[4]{n}} \qquad \text{and} \qquad z_2\notin I. \]
	We will study the contribution of this area in
\cref{lem:away-from-singularity-on-the-same-side} and we will show that it
converges to zero. The other case, when $z_1,z_2< z_0$ is analogous and we skip
it.
	
\end{itemize}
	\end{itemize}
\end{itemize}

\medskip

We will now explore the specific contributions of these cases before presenting
the culmination of our proof in \cref{sec:conclusion-prop-energy-via-cdf}.

\subsubsection{Close to the singularity}

In this section we analyze the integral on the left-hand side of
\eqref{eq:integral-over-I-times-I}. We will do it in two steps. Firstly, in
\cref{lem:close-to-singularity} we bound this integral over the product measure
$\nu \times \nu$ by a similar integral over the more regular measure $\mu\times
\mu$.

\paragraph{Conversion of the integral to $\mu \times \mu$}
\label{sec:conversion}

\begin{lemma}
	\label{lem:close-to-singularity}
	Let $\mu$ be a probability measure on the real line, and let constants $a<u_0<b$
	be such that $\mu$ restricted to the interval $[a,b]$ is absolutely continuous,
	with a density given by a function which is bounded below by a positive constant. 
	
	For any pair of open intervals \( I \) and \( J \)
	containing \( z_0 \), where the closure of~\( I \) is contained within \( J \)
	and \( F_\mu \) is Lipschitz continuous on \( J \), there exists a constant $C>0$ and a natural
	number $n_0 \in \N$ such that the following holds:
	
	For every $n \geq n_0$ and for any probability measure $\nu$ on the real line
	that satisfies the local refinement condition \eqref{eq:local-refinement-cdf}
	the following bound is satisfied:
	\begin{multline}
		\label{eq:integral-over-I-times-I}
		\iint\limits_{
			\substack{\{(z_1,z_2): \\ z_2-z_1 > \log n \ n^{-\frac{1}{4} }, 
				\\ (z_1,z_2)\in I \times I\}}
		}
		\frac{ \Pro\left[ \uwithnoise \in [z_1,z_2]   \right] }{z_2- z_1 }
		\dif \nu(z_1) \dif \nu(z_2) 
		\leq \\
		\iint\limits_{\substack{\{(z'_1,z'_2): \\ 
				z'_2 -  z'_1 > \left( \log n-  2 C \right) \ n^{-\frac{1}{4}}, \\ 
				(z'_1,z'_2)\in J \times J\}}}
		\frac{ \Pro\Big[ \uwithnoise \in 
			\left[z'_1- C n^{-\frac{1}{4}}, \;
			z'_2+ C n^{-\frac{1}{4}} \right]   
			\Big] 
		}{
			z'_2- z'_1 - 2 C n^{-\frac{1}{4}}
		}
		\dif \mu(z'_1) \dif \mu(z'_2). 		
	\end{multline}	
\end{lemma}
\newcommand{\intermediate}{I'}
\begin{proof}
	Let $\intermediate$ be an open interval such that the following sequence of inclusions is valid:
	\begin{equation}
		\label{eq:inclusions}
		z_0 \in I \subset \overline{I} \subset 
		\intermediate \subset \overline{\intermediate} 
		\subset J. 
	\end{equation}

	From the local refinement condition \eqref{eq:local-refinement-cdf}, it follows
	that the image \( F_\nu(I) \) is contained within an \( o(1)
	\)-neighborhood of \( F_\mu(I) \) as \( n \to \infty \). Given that \( F_\mu \)
	is a homeomorphism on the interval \( J \), we can conclude that
	\[
	\overline{F_\mu(I)} \subset F_\mu(\intermediate).
	\]
	Thus, there exists an integer \( n_0 \) such that for all sufficiently large \(
	n \geq n_0 \), and for any measure \( \nu \) satisfying the local refinement
	condition \eqref{eq:local-refinement-cdf}, the following inclusion holds for the
	corresponding quantiles:
	\[
	F_\nu(I) \subseteq F_\mu(\intermediate).
	\] 
	
	\medskip
	
	For each $i\in\{1,2\}$, the integral on the left-hand side of
	\eqref{eq:integral-over-I-times-I} over the variable \( z_i \in I \) with
	respect to the measure \( \nu \) can be transformed using a change of variables
	by setting \( z_i = F_\nu^{-1}(t_i) \). The variable $t_i$ has a natural
	interpretation as the quantile of the measure $\nu$. This allows us to express
	the integral in terms of the uniform measure over the variable $t_i$, where \(
	t_i \in F_\nu(I) \).
	
	An analogous idea can be applied to the integral on the right-hand side of
	\eqref{eq:integral-over-I-times-I} over the variable \( z'_i \in J \) with
	respect to the measure $\mu$. By combining these two changes of the variables we
	can express the variable on the left-hand side
	\[ z_i = \left( F_\nu^{-1} \circ F_\mu\right) (
	z_i') \]
	in terms of the variable on the right-hand side, over the set
	\[ \big( F_\mu^{-1} \circ F_\nu^{-1}\big)(I) \subseteq J.\]
	
	We apply \cref{lem:quantile-of-cdf-is-almost-identity}. With the notations of
	that result, our claim holds true for 
	\[ C:= \frac{1}{f_{\min}}. \qedhere \]
\end{proof}

\begin{lemma}
	\label{lem:quantile-of-cdf-is-almost-identity} 
	
	Under the assumptions of
	\cref{lem:close-to-singularity} and \eqref{eq:inclusions}, we denote by
	$f_{\min}>0$ any lower bound on the density of the measure $\mu$ on the
	interval $J$.
	If \( n \) is sufficiently large, then for any \( z' \in
	\intermediate \) and any probability measure $\nu$ that satisfies the local
	refinement condition \eqref{eq:local-refinement-cdf}, the following bound
	holds:
	\[
	\left| \left(F_\nu^{-1} \circ F_\mu\right)(z') - z' \right| \leq
	\frac{1}{f_{\min} \sqrt[4]{n}}.
	\]
\end{lemma}

\begin{proof}
	Let \( z = \big(F_\nu^{-1} \circ F_\mu\big)(z') \). Due to the definition of the quantile
	function, particularly at discontinuities, it follows that
	\begin{equation}
		\label{eq:quantile}
		F_\mu(z') \in \left[ F_\nu(z^-),\ F_\nu(z)\right],
	\end{equation}
	with the usual convention that $F_\nu(z^-)= \lim_{\xi\to z^-} F_\nu(\xi)$
	denotes the left limit of \CDF.
	
	We define:
	\begin{align*}
		z_+ & := z + \frac{|F_\mu(z) - F_\nu(z)|}{f_{\min}} , \\ 
		z_- & := z - \frac{ |F_\mu(z) - F_\nu(z^-)|}{f_{\min}}.
	\end{align*}
	
	There exists \( n_0 \) such that for all sufficiently large \( n \geq n_0 \),
	both \( z_- \) and \( z_+ \) are guaranteed to remain within the interval \( J
	\) by the local refinement condition \eqref{eq:local-refinement-cdf}. Since \(
	F_\mu \) locally grows at least linearly with slope \( f_{\min} \), we have:
	\[
	F_\mu(z_+) \geq F_\mu(z) + \left|F_\mu(z) - F_\nu(z)\right| \geq F_\nu(z) \geq F_\mu(z'),
	\]
	where the last inequality follows by \eqref{eq:quantile}.
	As $F_\mu$ is strictly increasing on $J$, we conclude that 
	\[
	z' \leq z_+.
	\]
	
	The argument for the inequality \( z_- \leq z' \) follows similarly.
	
	Consequently, as $z_- \leq z \leq z_+$, we establish that:
	\[
	|z - z'| \leq \frac{1}{f_{\min}} \max\left( |F_\mu(z) - F_\nu(z)|,\ |F_\mu(z) - F_\nu(z^-)| \right).
	\]
	This completes the proof by the local refinement condition
	\eqref{eq:local-refinement-cdf}.
\end{proof}

\paragraph{Estimate for the integral with respect to $\mu\times\mu$}

\begin{lemma}
	\label{lem:rhs-is-small}
	Under the assumptions of \cref{lem:close-to-singularity}, let us further suppose
	that the density of the measure $\mu$, when restricted to the interval $J$, is
	bounded above by $f_{\max}$. Given these conditions, we can establish an upper
	bound for the right-hand side of \eqref{eq:integral-over-I-times-I} 
	by:
	\[ f_{\max}^2 \ |J| + o(1).\]
\end{lemma}
\begin{proof}
	We begin by applying a change of variables to the integral on the right-hand
	side of \eqref{eq:integral-over-I-times-I}. Let us integrate over the variables:
	\[ z_1', \qquad t := z_2' - z_1'. \]
	To obtain an upper bound, we extend the integration domain as follows: 
	we integrate $z_1'$ over the whole real line $\R$, and we integrate $t$ over the interval 
	\[\left[ \frac{\log n - 2C}{\sqrt[4]{n}},\  |J|\right].\]
	
	For a fixed value of $t = z_2' - z_1'$, the integral of the numerator simplifies to:
	\[ \int \mathbb{P}\left[\uwithnoise \in [z_1'+a, z_1'+b]\right] \,\mathrm{d}z_1' = b-a. \]
	Notably, this result is independent of the specific probability distribution of $\uwithnoise$.
	
	Consequently, we can bound the right-hand side of \eqref{eq:integral-over-I-times-I} by:
	\begin{multline*} 
		f_{\max}^2 \int_{ (\log n - 2 C) n^{- \frac{1}{4} } }^{|J|}  
		\frac{t+2 C n^{-\frac{1}{4}}}{t- 2C n^{-\frac{1}{4}}} \dif t \leq \\ 
		f_{\max}^2 \left( |J| + \frac{4C}{\sqrt[4]{n}} 
		\int_{(\log n - 4 C) n^{-\frac{1}{4}}}^{|J|+ 2 C n^{-\frac{1}{4}}}  
		\frac{1}{u} \dif u \right) =
		f_{\max}^2 \ |J| + o(1),
	\end{multline*}	
	which completes the proof.
\end{proof}

\subsubsection{Extremely close to the singularity}

\begin{lemma}
	\label{lem:very-close-to-z0} 
Let $\mu$ be as in \cref{prop:energy-via-cdf}.
There exists a sequence $(C_n)$ which converges to
zero such that for any measure $\nu$ which fulfills the local refinement
condition \eqref{eq:local-refinement-cdf} the following bound holds true:
	\[ 	
 \iint\limits_{\substack{z_1 < z_2, \\
			|z_1-z_0|,\ |z_2-z_0| \leq \frac{1}{\sqrt[8]{n} \log n}
		}}  
	\frac{ \Pro\left[ \uwithnoise \in [z_1,z_2]   \right] }{z_2- z_1 }
	\dif \nu(z_1) \dif \nu(z_2)  < C_n.
	\]
\end{lemma}
\begin{proof}
The integral is bounded by 
the product of some upper bound on the integrand with the measure of the domain of integration.

Observe that the integrand is bounded above by the density of the random
variable $\uwithnoise$; this density is of order $O\left( \sqrt[4]{n} \right)$.

The measure of the domain of integration bounded by the square of:
\[   
\nu\left( \left[ 
z_0 -  \frac{1}{\sqrt[8]{n} \log n}, \
z_0 +  \frac{1}{\sqrt[8]{n} \log n} 
\right] 
\right)
\leq 
\mu\left( \left[ 
z_0 -  \frac{1}{\sqrt[8]{n} \log n}, \
z_0 +  \frac{1}{\sqrt[8]{n} \log n} 
\right] 
\right)
+ \frac{2}{\sqrt[4]{n}}.
\]
Clearly, the right-hand side is of order at most $o\left( \frac{1}{\sqrt[8]{n}}
\right)$, so its square us of order $o\left( \frac{1}{\sqrt[4]{n}}
\right)$. 

The whole product is therefore of order $o(1)$, as required.
\end{proof}

\subsubsection{Both variables on the same side of $z_0$, but not too close}

\begin{lemma}
	\label{lem:close-to-each-other-away-from-z0} There exists a sequence $(C_n)$
	which converges to zero, such that for any integer $n\geq 1$ and any
	probability measure $\nu$ on the real line the following bound is true:
	\[
	\iint\limits_{\substack{
			\{ (z_1,z_2): \\ 
			z_2>z_1 > z_0 + n^{-\frac{1}{4}} \log n 
			\}
	}}
	\frac{ \Pro\left[ \uwithnoise \in [z_1,z_2]   \right] }{z_2- z_1 }
	\dif \nu(z_1) \dif \nu(z_2) 
	\leq C_n.
	\]
\end{lemma}
\begin{proof}
	The integrand is bounded above by the supremum of the probability density of
	$\uwithnoise$ on the interval 
	\[ \left[ \frac{\log n}{\sqrt[4]{n}}, \infty \right) \] 
	which is the same as the product of $\sqrt[4]{n}$ by the supremum of
	the probability density of $\zmienna$  on the interval 
	\[ \left[ \log n, \infty\right). \] 
	A direct calculation involving the explicit formula for the
	density of the normal distribution concludes the proof.
\end{proof}

\subsubsection{The main contribution to the integral (away from the singularity)}

The numerator in \eqref{eq:energy-find-upper-bound-new} is trivially bounded by
$1$. Consequently, the following result offers an upper bound for the
corresponding part of the integral in \eqref{eq:energy-find-upper-bound-new}.

\begin{lemma}
	\label{lem:close-to-singularity-opposite-sides}
	Let $\mu$ be a probability measure on the real line and $u_0\in \R$. We denote
\[z_0= F_\mu(u_0), \qquad 
\Force_0 = \Force_{\mu}(u_0).\]

For any $\epsilon > 0$ and any open interval $I$ containing $z_0$, there exists
an open neighborhood $G$ of $\mu$ in the topology of weak convergence of
probability measures such that for any probability measure $\nu \in G$ the
following bound holds:
	\begin{equation}
		\label{eq:weak}
		\iint\limits_{\substack{\{(z_1,z_2): \\ z_1 \leq z_0 \leq z_2, \\ (z_1,z_2)\notin I \times I\}}}
		\frac{1}{z_2 - z_1} \,\mathrm{d}\nu(z_1) \,\mathrm{d}\nu(z_2) < \Force_0 + \epsilon.
	\end{equation}
\end{lemma}

\begin{proof}
The integral on the left-hand side of \eqref{eq:weak} involves a bounded,
positive integrand integrated over a closed set.  This integral can be rewritten
as an integral over $\mathbb{R}^2$ of the original integrand
\[ \R^2\ni (z_1,z_2) \mapsto \frac{1}{z_2-z_1} \]
multiplied by an
indicator function of a closed subset of $\mathbb{R}^2$. This product is the
pointwise limit of a weakly decreasing sequence $(f_n)$ of continuous,
uniformly bounded functions $f_n: \mathbb{R}^2 \to \mathbb{R}_+$. Consequently,
the left-hand side of \eqref{eq:weak} is bounded above by
	\begin{equation}
		\label{eq:smooth-weak}
		\iint_{\mathbb{R}^2} f_n(z_1, z_2) \,\mathrm{d}\nu(z_1) \,\mathrm{d}\nu(z_2).
	\end{equation}
	
	Let $n_0 \geq 1$ be a fixed integer to be specified later. Consider the
integral \eqref{eq:smooth-weak} for $n := n_0$ as $\nu$ converges to $\mu$ in
the weak topology of probability measures. Since the integral is a continuous
functional over the integrands, there exists a neighborhood $G$ of $\mu$ such
that for any $\nu \in G$,
	\begin{equation}
		\label{eq:smooth-weak2}
		\iint_{\mathbb{R}^2} f_{n_0}(z_1, z_2) \,\mathrm{d}\nu(z_1) \,\mathrm{d}\nu(z_2) \leq 
		\frac{\epsilon}{2} + \iint_{\mathbb{R}^2} f_{n_0}(z_1, z_2) \,\mathrm{d}\mu(z_1) \,\mathrm{d}\mu(z_2).
	\end{equation}
	
	Moreover,
	\begin{equation*}
		\lim_{n\to\infty} \iint_{\mathbb{R}^2} f_n(z_1, z_2) \,\mathrm{d}\mu(z_1) \,\mathrm{d}\mu(z_2) =
		\iint\limits_{\substack{\{(z_1,z_2): \\ z_1 \leq z_0 \leq z_2, \\ (z_1,z_2)\notin I \times I\}}}
		\frac{1}{z_2-z_1} \,\mathrm{d}\mu(z_1) \,\mathrm{d}\mu(z_2) \leq \Force_0,
	\end{equation*}
	allowing us to choose $n_0$ such that
	\begin{equation}
		\label{eq:smooth-weak3}
		\iint_{\mathbb{R}^2} f_{n_0}(z_1, z_2) \,\mathrm{d}\mu(z_1) \,\mathrm{d}\mu(z_2) < 
										\Force_0 + \frac{\epsilon}{2}.
	\end{equation}
	
	Combining \eqref{eq:smooth-weak}, \eqref{eq:smooth-weak2}, and
\eqref{eq:smooth-weak3} completes the proof.
\end{proof}

\subsubsection{One variable very close to $z_0$, the other far away}

\begin{lemma}
	\label{lem:away-from-singularity-on-the-same-side}
	Let $\mu$ be as in \cref{prop:energy-via-cdf}.	
	For each $\epsilon>0$ there exists $n_0$ and an open neighborhood $G$ of the
measure $\mu$ such that for any $n\geq n_0$ and for any
probability measure $\nu\in G$ on the real line, the following bound holds true:
\begin{equation}
	\iint\limits_{\substack{\{(z_1,z_2):  \\
			z_1 < z_2, \\ 
			|z_1-z_0| < n^{-\frac{1}{4}} \log n , 
			\\ z_2\notin I
			\}
		}}
	\frac{1}{z_2- z_1 }
	\dif \nu(z_1) \dif \nu(z_2) 
	\leq \epsilon.
\end{equation}
\end{lemma}
\begin{proof}
Note that when $n$ is sufficiently large, the integrand is uniformly bounded
above by a constant which is independent of $n$. Our strategy is to find
the neighborhood $G$ in such a way that the measure of the integration area is
smaller than an arbitrary $\epsilon'>0$.

Let $c>0$ be sufficiently small that
\[ \mu\left( \left[ z_0,\  z_0+c\right] \right) < \epsilon'. \]
There exists a neighborhood $G$ of $\mu$ such that for any $\nu\in G$ the
following bound is valid:
\[ \nu\left( \left[ z_0,\  z_0+c\right] \right) < \epsilon'. \]

When $n$ is sufficiently large then
\[ \frac{\log n}{\sqrt[4]{n}} < c \]
which completes the proof. 
\end{proof}

\subsubsection{Conclusion of the proof of \cref{prop:energy-via-cdf}}
\label{sec:conclusion-prop-energy-via-cdf}

Let $\epsilon'>0$ be a positive constant which will be fixed later. For this
value of $\epsilon:=\epsilon'$ we apply \cref{lem:close-to-singularity-opposite-sides,lem:away-from-singularity-on-the-same-side}
We define $G$ to be the open neighborhood of the measure $\mu$ defined as
the intersection of the neighborhoods provided by these lemmas.

By summing the contributions of all of the above cases it follows that there
exists a sequence $C_n$ which converges to zero and $n_0$ such that for any
$n\geq n_0$ and any measure $\nu\in G$ which fulfills the local refinement
condition \eqref{eq:local-refinement-cdf} the following bound is true:
\[ \E \left[\Force_{\nu}(\uwithnoise)\right] \leq  \Force_0 + 2 \epsilon' + f_{\max}^2 \ |J| + C_n. \]
Since $\epsilon'>0$ as well as the length of the interval $J$ can be chosen to
be arbitrarily small, \cref{prop:energy-via-cdf} follows.

\subsection{\cref{prop:energy-via-cdf} implies the upper bound in \assum{item:energy2}}
\label{sec:B2-upper-bound}

For an integer $n \geq 1$, define the following random variables:
\begin{align*}
	\Delta_n &= \ForceRegularized{\frac{1}{\sqrt{n}}}{\mu_{\omega_n}}(\uwithnoise) - \Force_0, \\
	\Delta_n^+ &= \max(\Delta_n, 0), \\
	\Delta_n^- &= \max(-\Delta_n, 0),
\end{align*}
such that $\Delta_n = \Delta_n^+ - \Delta_n^-$ and $\Delta_n^-\geq 0$, as well
as $\Delta_n^+ \geq 0$.

Fix $\epsilon > 0$ and $\eta > 0$. Our goal is to show that
\begin{equation}
	\label{eq:interaction-upperbound}
	\limsup_{n\to\infty} \Pro \left[ \Delta_n > \epsilon \right] =
	\limsup_{n\to\infty} \Pro \left[ \Delta_n^+ > \epsilon \right]
	< \eta.
\end{equation}
In the pursuit of this goal we may condition on events whose probabilities
converge to $1$, preserving the validity of our proof while offering greater
analytical flexibility. We will perform such a conditioning twice. We start by
setting $\delta > 0$ and $G$ to be the parameters provided by
\cref{prop:energy-via-cdf} for $\epsilon' := \frac{1}{2} \epsilon \eta$.

Firstly, recall that the sequence of random measures $(\mu_{\omega_n})$
converges to $\mu_{\Omega}$ in the weak topology of probability measures, in
probability. (For a detailed discussion of this convergence, refer to the
paragraph preceding Equation~\eqref{eq:integral-convergence}.) It follows that
we can condition on the event that the global proximity condition
\eqref{eq:global-proximity} is satisfied for $\nu:=\omega_n$:
\[ \mu_{\omega_n} \in G. \]

Secondly, by \eqref{eq:modulus}, we can additionally condition on the event that
the local refinement condition \eqref{eq:local-refinement-cdf} holds true for
$\nu:=\mu_{\omega_n}$.

Under these two conditionings, \cref{prop:energy-via-cdf} is applicable for
$\nu:=\omega_n$ and all sufficiently large values of $n$, thus by applying the
expected value over $\omega_n$ it follows that
\begin{equation}
	\label{eq:limsup-delta}
	\limsup_{n\to\infty} \E \Delta_n \leq 
	\limsup_{n\to\infty} \E \left[\Force_{\mu_{\omega_n}}(\uwithnoise)\right] - \Force_0
	< \epsilon' = \frac{1}{2}\epsilon \eta.
\end{equation}

The random variables $\Delta_n^- \geq 0$ are uniformly bounded from above by $\Force_0$. Thus,
\begin{align*}
	\Delta_n^- 
	&\leq 
	\frac{1}{2} \epsilon \eta + 
	      \Force_0\  \mathbbm{1}_{\left[ \Delta_n^- > \frac{1}{2} \epsilon \eta \right]}, \\
	\intertext{and}
	\E \Delta_n^- 
	&\leq  
	\frac{1}{2} \epsilon \eta+ \Pro\left[ \Delta_n^- > \frac{1}{2} \epsilon \eta \right] \Force_0 .
\end{align*}
By applying \eqref{eq:interaction-lowerbound} with $\epsilon' = \frac{1}{2}
\epsilon \eta$ it follows that the second  term on the right-hand side converges
to zero. Consequently,
\begin{equation}
	\label{eq:limsup-delta-minus}
	\limsup_{n\to\infty} \E \Delta_n^- \leq \frac{1}{2} \epsilon\eta.
\end{equation}

By Markov's inequality,
\[
\Pro\left[ \Delta_n > \epsilon \right] = 
\Pro\left[ \Delta_n^+ > \epsilon \right] \leq 
\frac{\E \Delta_n^+}{\epsilon} 
= \frac{\E \Delta_n + \E \Delta_n^-}{\epsilon}.
\]
Combining \eqref{eq:limsup-delta} and \eqref{eq:limsup-delta-minus}, we find
that the right-hand side is asymptotically bounded by $\eta$, thus establishing
\eqref{eq:interaction-upperbound} and completing the proof of
\assum{item:energy2}.

\subsection{Towards \assum{item:regular2}. Cauchy transform in a random point is
	not too singular}

Recall that the modified Cauchy transform is defined in
\eqref{eq:cauchy-abs}. This quantity \( \cauchy^+_{\rho_n}(z) \) is
particularly effective in scenarios where \( \rho_n \) is part of a sequence of
probability measures supported on an interval of the form \[
\left[-\Theta\left(\sqrt{n}\right),\ \Theta\left(\sqrt{n}\right)\right], \] which
expands proportionally to \( \sqrt{n} \), and where the argument \( z =
\Theta(\sqrt{n}) \) scales in a similar manner.

In some cases, it is advantageous to dilate the \( n \)-th element of the
sequence \( (\rho_n) \) by a factor of \( \frac{1}{\sqrt{n}} \), resulting in a
new sequence of probability measures supported on a common compact interval. The
modified Cauchy transform \( \cauchy^+_{\rho_n}(z) \) for the original measure
is related to the following expression for the dilated measure: 
\begin{equation}
	\label{eq:cauchy-rescaled}
	\cauchy^+_{\rho_n}\left( \sqrt{n}\ u \right) = \frac{1}{\sqrt{n}} \int \frac{1}{|x - u|
	+ n^{-\frac{1}{2}}} \, \dif D_{\frac{1}{\sqrt{n}}, \rho_n}(x), 
\end{equation}
where \( D_{\frac{1}{\sqrt{n}}, \rho_n} \) denotes the dilation of the measure
\( \rho_n \) by the factor \( \frac{1}{\sqrt{n}} \).

\smallskip

The following result concerns the integral on the right-hand side of
\eqref{eq:cauchy-rescaled} and is crucial for establishing
\assum{item:regular2}. Heuristically, even if the modified Cauchy transform
exhibits locally very large values, its behavior becomes significantly more
regular when evaluated at a point perturbed by random noise.

\begin{lemma}
	\label{lem:Cauchy-good-control}
Let $(\nu_n)$ be a sequence of random probability measures on the real line, and
let $\mu$ be a probability measure on the same line. Assume there exist real
numbers $a < b$ such that the measure $\mu$ restricted to the interval $[a, b]$
is absolutely continuous, with a bounded density.

Furthermore, assume that the \CDFs of the measures $(\nu_n)$ converge
to the \CDF of $\mu$ at the prescribed speed:
\begin{equation}
	\label{eq:assumption-cdf-converges-fast}
	 \sup_{u \in [a,b] }  \sqrt[4]{n} \
\left| F_{\nu_n}(u) - F_\mu(u) \right| 
\xrightarrow[n\to\infty]{P} 0. 
\end{equation}

Let $\zmienna$ and $\uwithnoise$ be as in the statement of \cref{thm:CLT-trans}.
Then, for any $u_0\in (a,b)$, we have that
\begin{equation}
	\label{eq:towards-b3}
	 \frac{1}{(\log n)^2} 
\int \frac{1}{|x- \uwithnoise |+ n^{-\frac{1}{2} }} \dif \nu_n(x)
 \xrightarrow[n\to\infty]{P} 0. 
\end{equation} 
\end{lemma}

The remaining part of this section is devoted to the proof.

\subsubsection{Conditioning}

Our general strategy is to prove that the probability of 
the left-hand side of \eqref{eq:towards-b3} exceeding some constant $\epsilon > 0$
converges to zero.
For any constant $A > 0$, this probability is bounded above by the sum:
	\[ \Pro\Big[ |\zmienna| > A \Big] + 
	\Pro\giventhat*{
		\int \frac{1}{|x- \uwithnoise |+ n^{-\frac{1}{2} }} \dif \nu_n(x) > \epsilon\ (\log n)^2 
		}{
		\ |\zmienna| \leq A 
		}. \]
We can choose $A>0$ sufficiently large such that the first summand becomes
arbitrarily small. In the remaining part of the proof we will show that the
conditional probability in the second summand converges to zero as $n\to\infty$.
Given our conditional setup, we can assume in the following that $\zmienna$ is a
random variable supported on the interval $[-A, A]$, obtained by conditioning
the standard normal random variable.

By employing a similar conditioning argument, we can invoke
\eqref{eq:assumption-cdf-converges-fast} to assume that
\begin{equation}
	\label{eq:cdfs-are-close}
	 	 \sup_{u \in [a,b] }  \sqrt[4]{n} \
\left| F_{\nu_n}(u) - F_\mu(u) \right| < 1 
\end{equation}
holds with probability $1$ for all sufficiently large values of $n$.

\subsubsection{Split the domain of integration}
\label{sec:split-the-domain}

We will split the domain of integration in  
\begin{equation}
	\label{eq:reg-cauchy}
\int_{\R} 
\frac{1}{|x- \uwithnoise| + n^{-\frac{1}{2}}} \dif \nu_n(x) 
\end{equation}
into three parts, and investigate the asymptotic behavior of each part separately.

\paragraph{Easy deterministic estimate, very far from $u_0$} 

We start with the first integral which
contributes to \eqref{eq:reg-cauchy}, namely
\[
	\int_{ \R\setminus [a,b]  } 
\frac{1}{|x- \uwithnoise| +n^{-\frac{1}{2}}} \dif \nu_n(x). 
\]
Let $c>0$ be a positive constant satisfying
\[ [ u_0 - c , \ u_0 + c ] \subset (a, b). \]
Given that the random perturbation $\zmienna$ is bounded, there exists an $n_0
\in \N$ such that for all $n \geq n_0$, the following inequality holds
uniformly for any $x \in \R\setminus \left[a, b\right]$ contributing to the integral:
\[ |x- \uwithnoise| \geq c. \]
Consequently, we can establish an almost sure upper bound for the integral:
\begin{equation}
	\label{eq:cauchy-integral-bound-1}
		\int_{ \R\setminus [a,b]} 
\frac{1}{|x- \uwithnoise| +n^{-\frac{1}{2}}} \dif \nu_n(x) \leq   \frac{1}{c}. 
\end{equation}
This bound holds with probability $1$ for all $n\geq n_0$.

\paragraph{Deterministic estimate, intermediate distance to $u_0$} 

We consider now the second integral which contributes to \eqref{eq:reg-cauchy}.
In this section we will show that:
\begin{equation}
	\label{eq:cauchy-integral-bound-2}
\int_{
	 x\in [a,b] \text{ and }
	|x- u_0| > 2 A n^{-\frac{1}{4}}
}
\ \ \
\frac{1}{|x-\uwithnoise| +n^{-\frac{1}{2}}} \dif \nu_n(x) = O\left( \log n \right).
\end{equation}

\smallskip

For any $x$ contributing to the integral \eqref{eq:cauchy-integral-bound-2} we
have $|x- u_0| > 2 A n^{-\frac{1}{4}}$; on the other hand $|\uwithnoise- u_0|
\leq  A n^{-\frac{1}{4}}$; it follows that
\[ |x -  \uwithnoise | \geq \left( |x-u_0| \right) - \left( |\uwithnoise - u_0| \right) 
\geq  
 \frac{| x-  u_0|}{2}. \]
It follows that \eqref{eq:cauchy-integral-bound-2} is bounded from above by
\begin{equation}
	\label{eq:upper-bound}
	2 \int_{
	x\in [a,b] \text{ and }
	|x- u_0| > 2 A n^{-\frac{1}{4}}
}
\ \ \
\frac{1}{|x- u_0| } \dif \nu_n(x). 
\end{equation}

The integration area  encompasses two distinct segments: one to the left 
and one
to the right of $u_0$. In the following discussion, we will focus on the segment
to the right. The integral over the left segment can be derived using a similar
approach. 

We will utilize integration by parts, as outlined below:
\begin{multline*}
\int_{ u_0 + 2 A n^{-\frac{1}{4}}}^{ b }
\frac{1}{x-u_0 } \dif \nu_n(x)= 
\\
\left[ \frac{1}{x- u_0 } F_{\nu_n}(x) \right]_{ u_0 + 2  A n^{-\frac{1}{4}}}^{b } +
\int_{ u_0 + 2  A n^{-\frac{1}{4}}}^{ b }\ \ 
\frac{1}{\left( x-u_0 \right)^2}  F_{\nu_n}(x) \dif x.
\end{multline*}
An equivalent formula can be established for the corresponding integral over the
same interval with respect to the measure $\mu$. By subtracting these two
formulas, and using the bound \eqref{eq:cdfs-are-close}, we can derive:
\begin{multline}
\label{eq:by-parts}
	\left| 
	\int_{ u_0 + 2  A n^{-\frac{1}{4}}}^{ b }
	\frac{1}{x- u_0 } \dif \nu_n(x)
	-
	\int_{ u_0 + 2  A n^{-\frac{1}{4}}}^{ b }
\frac{1}{x- u_0 } \dif \mu(x)
	\right|
	\leq 
\\
\shoveleft{\frac{|F_{\nu_n}(b)- F_\mu(b)| }{b-u_0} +
\frac{
	\left|
	       F_{\nu_n}\left(  u_0 + 2  A n^{-\frac{1}{4}} \right)- 
	                  F_\mu\left(  u_0 + 2  A n^{-\frac{1}{4}} \right)
	\right| 
	}{ 
	  2  A n^{-\frac{1}{4}}
	  } +} \\
\int_{ u_0 + 2  A n^{-\frac{1}{4}}}^{ b }\ \ 
\frac{|F_{\nu_n}(x)- F_{\mu}(x)|}{\left( x-u_0 \right)^2}   \dif x =  O(1)
\end{multline}
in the sense that there is a universal, deterministic constant which is an upper
bound for the right-hand side.

We denote by $C$ any upper bound on the density of the measure $\mu$ on the
interval $[a,b]$. The integral which is the subtrahend on the left-hand side of
\eqref{eq:by-parts} can be bounded as follows:
\[
	\int_{ u_0 + 2  A n^{-\frac{1}{4}}}^{ b }
\frac{1}{x- u_0 } \dif \mu(x) \leq 
C \int_{ u_0 + 2  A n^{-\frac{1}{4}}}^{ b }
\frac{1}{x- u_0 } \dif x = O\left( \log n \right).
\]

\smallskip

The preceding considerations establish that the random integral
\eqref{eq:upper-bound} possesses a deterministic upper bound of \( O\left( \log
n \right) \). 
As a direct consequence, the claimed upper bound in
\eqref{eq:cauchy-integral-bound-2} is also valid.

\paragraph{Probabilistic estimate, very close to $u_0$}

Consider the interval
\[ I_n= \left[ u_0 - \frac{2 A}{\sqrt[4]{n}}, \  u_0 + \frac{2 A}{\sqrt[4]{n}} \right]. \]
Our goal in this section is to show that
\begin{equation}
	\label{eq:cauchy-integral-bound-3}
	\frac{1}{(\log n)^2} 
\int_{I_n}
\frac{1}{|x-\uwithnoise| + n^{-\frac{1}{2}}} \dif \nu_n(x)
\xrightarrow[n\to\infty]{P} 0.
\end{equation}

\smallskip

Let $X_n$ be a random variable generated by the following two-step sampling procedure.
Firstly, we sample the random variables $\nu_n$ and $\zmienna$. In the second step we sample 
$X_n$ 
whose distribution is derived from the conditioning of the measure $\nu_n$
to the interval $I_n$.
In this way the conditional expectation of the integral on the left hand side of
\eqref{eq:cauchy-integral-bound-3}:
\begin{equation}
	\label{eq:integral-near-singularity}
\E \giventhat*{
	\int_{I_n}
	\frac{1}{|x-\uwithnoise| + n^{-\frac{1}{2}}} \dif \nu_n(x) 
}{
	\nu_n
}
= 
\nu_n(I_n)
\ \
 \E\giventhat*{ 
 	\frac{1}{
 	\left| X_n - \uwithnoise \right| +
 	 n^{-\frac{1}{2} }
 	}
 }{
\nu_n
}
\end{equation}
can be expressed in terms of the conditional expected value of a random variable
involving~$X_n$. In the following we shall investigate each of the two factors
on the right hand side. Without loss of generality we may assume that $n$ is
large enough so that $I_n\subset (a,b)$.

\smallskip

By \cref{eq:cdfs-are-close}:
\[ \big| \nu_n(I_n) - \mu(I_n) \big| \leq \frac{2}{\sqrt[4]{n}}. \]
As before, let $C$ denote any upper bound on the density of $\mu$ on the interval $(a,b)$.
It follows that 
\[ \mu(I_n) \leq \frac{4 A C}{\sqrt[4]{n}}. \]
In this way we found a deterministic upper bound
\[ \nu_n(I_n) = O\left( \frac{1}{\sqrt[4]{n}} \right) \]
for the first factor on the right-hand side of \eqref{eq:integral-near-singularity}.

\smallskip

In order to bound the second factor on the right-hand side of
\eqref{eq:integral-near-singularity} we use conditioning over the measure
$\nu_n$. With this conditioning, the second factor on the right-hand side of
\eqref{eq:integral-near-singularity} can be seen as the integral
\begin{equation}
	\label{eq:close-to-singularity}
	\E\left[
		\frac{1}{
			\left| X_n - \uwithnoise \right| +
			n^{-\frac{1}{2} }
		}
	\right]
	=
	 \int \frac{1}{|x|+ n^{-\frac{1}{2}} } 
	      \dif \mu_{Z_n}(x)
\end{equation}
over the (conditional) distribution of the random variable 
\[ Z_n= X_n - \uwithnoise= (X_n - u_0) - \frac{1}{\sqrt[4]{n}} \zmienna .\]

Firstly, notice that the absolute value of this random variable is almost surely
bounded by
\[ 
\frac{3 A }{\sqrt[4]{n}}. 
\]

Secondly, since the probability distribution of $Z_n$ is a convolution of:
\begin{itemize}
	\item the distribution of $(X_n - u_0)$, and 
	\item the Gaussian distribution with the standard deviation equal to
$\frac{1}{\sqrt[4]{n}}$,
\end{itemize} 
it follows that the distribution of $Z_n$ is absolutely continuous, with the
density bounded from above by that of a Gaussian measure, which is
$\frac{\sqrt[4]{n}}{\sqrt{2\pi}} $.

By combining these observations it follows that the integral
\eqref{eq:close-to-singularity} is bounded from above by
\[ \frac{2 \sqrt[4]{n}}{\sqrt{2\pi} } 
    	\int_{n^{-\frac{1}{2}}}^{3 A n^{-\frac{1}{4}}} \frac{1}{x} \dif x  =
O\left( \log n \ \sqrt[4]{n} \right). \]

\smallskip

Our analysis thus far shows that the conditional expectation given in
\eqref{eq:integral-near-singularity} is almost surely bounded above by a
deterministic sequence that grows at most like $O(\log n)$. Consequently, the
left-hand side of \eqref{eq:cauchy-integral-bound-3} converges to $0$ in the
$L^1$ norm. This also implies that the convergence in probability stated in
\eqref{eq:cauchy-integral-bound-3} holds true.

\subsubsection{Conclusion of the proof of \cref{lem:Cauchy-good-control}}

Our investigation of the integral \eqref{eq:reg-cauchy} is complete.
Equations
\eqref{eq:cauchy-integral-bound-1},
\eqref{eq:cauchy-integral-bound-2}, and
\eqref{eq:cauchy-integral-bound-3}
establish the proof of \cref{lem:Cauchy-good-control}.

\subsection{Assumptions of \cref{thm:CLT-shape} imply \assum{item:regular2}}

Our strategy is to apply \cref{lem:Cauchy-good-control} for the sequence of
random measures $(\nu_n)$, where \[ \nu_n := \mu_{\omega_n} \] is the transition
measure of the random \conti diagram $\omega_n$, and for $\mu:= \mu_\Omega$
equal to the transition measure of the limit \conti diagram $\Omega$.

We begin by verifying that the assumptions of \cref{lem:Cauchy-good-control} are
satisfied. Recall that we have already established \eqref{eq:modulus}, which
demonstrates that the \CDFs converge at the required rate. Moreover,
\assum{item:transition-measure-omega-is-absolutely-continuous} ensures that the
limit measure possesses a locally bounded density, as required. Thus, we
conclude that \cref{lem:Cauchy-good-control} is applicable.

Upon applying this lemma and utilizing \eqref{eq:cauchy-rescaled}, we obtain the
following result:
\[
	\frac{\nklatki^{\frac{4}{8}} }{(\log n)^2}   
	       \cauchy^+_{\lambda^{(n)}}(\sqrt{n} \ \uwithnoise)  
	\xrightarrow[n\to\infty]{P} 0
	\]
which provides a significantly stronger bound than what is stipulated by
\assum{item:regular2}.

\section{The first tool:  The cumulative function of a tableau}
\label{sec:towards-proof}

\subsection{Cumulative function of a tableau}
\label{sec:cumulative}

In our recent paper \cite{MarciniakSniadyAlternating}, we introduced the 
\emph{cumulative function} for a Poissonized tableau $T$. This function, denoted
as $F_T\colon \R \to [0,1]$, is defined as follows:
\[ 
F_T(u) = \inf\big\{ z \in [0,1] : \uIns(T; z) > u \big\} \qquad \text{for } u \in \R;
\]
where $\uIns(T; z)$ represents the $u$-coordinate of the box $\Ins(T; z)$, as
defined in \cref{sec:definition-of-uIns}. If the infimum is taken over an empty
set, we define $F_T(u) = 1$. For an example, see \cref{fig:french2}.

\begin{figure}
		\begin{tikzpicture}[scale=1.5,rotate=45]

		\coordinate (start) at (-2,-2);
		\coordinate (p1) at ($(start)+(-2,2)$);
		\coordinate (p2) at ($(start)+(2,-2)$);
		
		\coordinate (a) at ($2*(0.70710678118,0.70710678118)$);
		\coordinate (q0) at ($(p1)!(-4,0)!(p2)$);
		\coordinate (q1) at ($(p1)!(-2,0)!(p2)$);
		\coordinate (q2) at ($(p1)!(0,0)!(p2)$);
		\coordinate (q3) at ($(p1)!(3,0)!(p2)$);
		\coordinate (q4) at ($(p1)!(4,0)!(p2)$);

		\draw[decoration={markings,mark=at position 1 with {\arrow[scale=2]{>}}},
		postaction={decorate}] (p1) -- (p2) node[anchor=north west]{$u$};
		
		\draw[decoration={markings,mark=at position 1 with {\arrow[scale=2]{>}}},
		postaction={decorate}] (start) -- ($(start)+1.2*(a)$) node[anchor=west]{$F_T(u)$};

		\foreach \x in {-3, -2, -1, 0, 1, 2, 3} 
		{ \draw ($(p1)!(\x,0)!(p2)$) +(1pt,1pt) -- +(-1pt,-1pt) node[anchor=north] {\x}; }
		
		\foreach \y/\yt in {0.2/0.2,0.4/0.4,0.6/0.6,.8/0.8,1/1} 
		{ \draw[thin,dotted] ($(start)+\y*(a)+(-2,2)$)--($(start)+\y*(a)+(2,-2)$);
			\draw ($(start)+\y*(a)+(-1pt,1pt)$) -- ($(start)+\y*(a)+(1pt,-1pt)$) node [anchor=west]{\tiny$\yt$};
		}

		\draw[black!40, ultra thick] (4.5,0) node [anchor=south]{\textcolor{black}{$x$}} -- (0,0) -- (0,3.5) node[anchor=east]{\textcolor{black}{$y$}};

		\begin{scope}
			\clip (0,0) -- (3,0) -- (3,1) -- (1,1) -- (1,2) -- (0,2);
			\draw (0,0) grid (3,3); 
		\end{scope}

		\draw[ultra thick] (0,0) -- (3,0) -- (3,1) -- (1,1) -- (1,2) -- (0,2) -- cycle;

		\foreach \x/ \y/ \t in { 3/0/{(0, 0.2) }, 1/1/{(0.2,0.7)}, 0/2/{(0.7,1)} }
		{ 
			\fill[blue] (\x,\y) circle (2pt);
			\draw[blue, dashed] (\x,\y) -- ($(p1)!(\x,\y)!(p2)$ );
		}

		\draw[red,ultra thick] ($(q0)+0*(a)$)  -- ($(q1)+0*(a)$) -- ($(q1)+.01*(a)$); 
		\draw[red,ultra thick] ($(q1)+0.2*(a)-.01*(a)$) -- ($(q1)+0.2*(a)$) -- ($(q2)+0.2*(a)$) -- ($(q2)+0.2*(a)+.01*(a)$); 
		\draw[red,ultra thick] ($(q2)+0.7*(a)-.01*(a)$)  --($(q2)+0.7*(a)$)  -- ($(q3)+0.7*(a)$)-- ($(q3)+0.7*(a)+.01*(a)$); 
		\draw[red,ultra thick] ($(q3)+1*(a)-.05*(a)$)   --($(q3)+1*(a)$)   -- ($(q4)+1*(a)$);

		\draw[red,very densely dashed, ultra thick]  ($(q1)+0*(a)$) -- ($(q1)+0.2*(a)$); 
		\draw[red,very densely dashed, ultra thick] ($(q2)+0.2*(a)$) -- ($(q2)+0.7*(a)$); 
		\draw[red,very densely dashed, ultra thick] ($(q3)+0.7*(a)$) -- ($(q3)+1*(a)$);

		\path node () at (0.5,0.5){0.1};
		\path node () at (1.5,0.5){0.2};
		\path node () at (2.5,0.5){0.7};
		\path node () at (0.5,1.5){0.5};

	\end{tikzpicture}
	
	\caption{A Poissonized tableau $T$ shown in Russian coordinates. The red line
	depicts its cumulative function $F_T$. This zigzag line splits the infinite
	rectangle $\R \times [0,1]$ into two regions: the \emph{northwest
		region} and the \emph{southeast region}.}
	
	\label{fig:french2}
\end{figure}
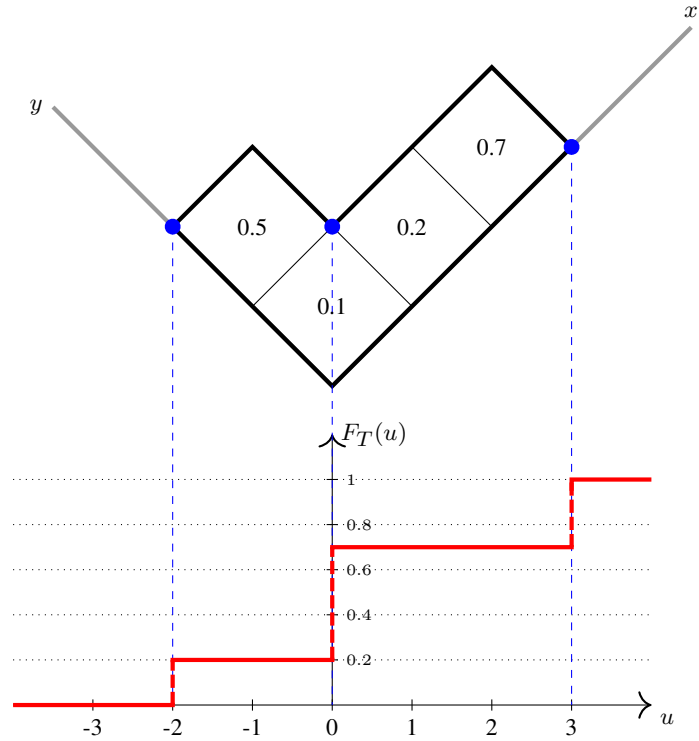

Due to the monotonicity of Schensted row insertion, the following equivalence
holds for any $u \in \R$ and $z \in [0,1)$:
\begin{equation}
	\label{eq:why-ft}
	F_T(u) \leq z \iff \uIns(T; z) \geq u.
\end{equation}
This relationship implies that the cumulative function can be regarded as the
inverse of the insertion function $[0,1] \ni z \mapsto \uIns(T; z) \in
\R$. Specifically, the relationship between $F_T$ and $\uIns(T; \cdot)$
is analogous to that between the \CDF of a measure
and its quantile function. Consequently, the cumulative function is an ideal
tool for proving the main results of this paper, as they all pertain to the
function $\uIns$.

\subsection{Wider perspective: cumulative function of a random tableau}

Consider a Young diagram $\lambda$ (chosen either randomly or deterministically)
and let $T$ be a random Poissonized tableau of shape $\lambda$. An intriguing
problem in this context is the study of the asymptotics of the random function
$u \mapsto F_T(u)$ as the number of boxes in $\lambda$ approaches infinity.

\subsubsection{Plancherel-distributed Poissonized tableaux}

We consider a specific instance of this general problem by examining the random
Plancherel-distributed Poissonized tableau $T$ with $n$ boxes (see
\cref{sec:plancherel-tableau}). Figures \ref{fig:CDF-100} and
\ref{fig:CDF-10000} illustrate the results of a single Monte Carlo simulation,
comparing three key functions:
\begin{enumerate}[label=(\roman*)]
	\item \label{item:i-tab} 
	the cumulative function $F_T$ of the tableau (thick red line);
	
	\item \label{item:i-tran} 
	the \CDF $K_\lambda$ of the transition measure of
	$\lambda$, where $\lambda$ is the shape of $T$ (thin blue line);
	
	\item \label{item:i-as} 
	the \CDF of the rescaled arcsine law (smooth black
	line).
\end{enumerate}

\begin{figure}
	    \begin{tikzpicture}[xscale=0.2,yscale=5]
	\draw[->] (-25,0) -- (25,0);
	\foreach \x in {-20,-15,...,20} { \draw (\x,0.4pt)--(\x,-0.4pt) node[anchor=north]{\tiny $\x$}; };
	\draw[->] (0,0) -- (0,1.2);
	\foreach \y/\yt in {0.2/0.2,0.4/0.4,0.6/0.6,0.8/0.8,1/1} { \draw (-20pt,\y)--(20pt,\y) node[anchor=west]{\tiny $\yt$}; };
	\draw[thick] plot[smooth,thick] file {figures/data/CDF-Wigner100.txt};

	\draw[red,ultra thick] (-26.000000,0) -- (-17.000000,0)-- (-17.000000,0.01); 
	\draw[red,ultra thick] (-17.000000,0.127535) -- (-17.000000,0.147535) -- (-5.000000,0.147535)-- (-5.000000,0.157535); 
	\draw[red,ultra thick] (-5.000000,0.532202) --(-5.000000,0.542202) -- (6.000000,0.542202) -- (6.000000,0.552202);
	\draw[red,ultra thick] (6.000000,0.795199) --(6.000000,0.805199) -- (15.000000,0.805199) -- (15.000000,0.815199);  
	\draw[red,ultra thick] (15.000000,0.917959) --(15.000000,0.927959) -- (18.000000,0.927959)-- (18.000000,0.937959);  
	\draw[red,ultra thick] (18.000000,0.99) --(18.000000,1) -- (26.000000,1);
	;	
	\draw[red,very very densely dashed,ultra thick] (-17.000000,0.01) -- (-17.000000,0.147535); 
	\draw[red,very very densely dashed,ultra thick] (-5.000000,0.147535) -- (-5.000000,0.542202); 
	\draw[red,very very densely dashed,ultra thick] (6.000000,0.542202) -- (6.000000,0.805199); 
	\draw[red,very very densely dashed,ultra thick] (15.000000,0.805199) -- (15.000000,0.927959); 
	\draw[red,very densely dashed,ultra thick] (18.000000,0.917959) -- (18.000000,1);
	
	\draw[blue,thick] (-26.000000,0) -- (-17.000000,0)-- (-17.000000,0.01); 
	\draw[blue,thick] (-17.000000,0.077039) --(-17.000000,0.087039) -- (-14.000000,0.087039)-- (-14.000000,0.097039); 
	\draw[blue,thick] (-14.000000,0.136270) --(-14.000000,0.146270) -- (-10.000000,0.146270)-- (-10.000000,0.156270); 
	\draw[blue,thick] (-10.000000,0.206903) --(-10.000000,0.216903) -- (-5.000000,0.216903)-- (-5.000000,0.226903); 
	\draw[blue,thick] (-5.000000,0.406912) --(-5.000000,0.416912) -- (-2.000000,0.416912)-- (-2.000000,0.426912); 
	\draw[blue,thick] (-2.000000,0.457620) --(-2.000000,0.467620) -- (1.000000,0.467620)-- (1.000000,0.477620);
	\draw[blue,thick] (1.000000,0.621032) --(1.000000,0.631032) -- (6.000000,0.631032)-- (6.000000,0.641032);
	\draw[blue,thick] (6.000000,0.753550) --(6.000000,0.763550) -- (9.000000,0.763550) -- (9.000000,0.773550);
	\draw[blue,thick] (9.000000,0.858327) --(9.000000,0.868327) -- (15.000000,0.868327)-- (15.000000,0.878327);
	\draw[blue,thick] (15.000000,0.909973) --(15.000000,0.919973) -- (18.000000,0.919973)-- (18.000000,0.929973);
	\draw[blue,thick] (18.000000,0.99) --(18.000000,1) -- (26.000000,1);
	
	\draw[blue,middle densely dashed,thick] (-17.000000,0.01) -- (-17.000000,0.087039); 
	\draw[blue,middle densely dashed,thick] (-14.000000,0.087039) -- (-14.000000,0.146270); 
	\draw[blue,middle densely dashed,thick] (-10.000000,0.146270) -- (-10.000000,0.216903); 
	\draw[blue,middle densely dashed,thick] (-5.000000,0.216903) -- (-5.000000,0.416912); 
	\draw[blue,middle densely dashed,thick] (-2.000000,0.416912) -- (-2.000000,0.467620); 
	\draw[blue,middle densely dashed,thick] (1.000000,0.477620) -- (1.000000,0.631032); 
	\draw[blue,middle densely dashed,thick] (6.000000,0.631032) -- (6.000000,0.763550); 
	\draw[blue,middle densely dashed,thick] (9.000000,0.763550) -- (9.000000,0.868327); 
	\draw[blue,middle densely dashed,thick] (15.000000,0.868327) -- (15.000000,0.919973); 
	\draw[blue,middle densely dashed,thick] (18.000000,0.929973) -- (18.000000,1);
\end{tikzpicture}  
	
	\caption{Results of a single Monte Carlo simulation. The insertion tableau $T =
	P(w_1, \dots, w_{\nklatki})$ with $\nklatki = 100$ boxes was generated using
	the RSK algorithm applied to a sequence of \iid $U(0,1)$ random variables.
	This figure displays two functions related to this tableau. The thin blue line
	represents the \CDF $K_\lambda$ of the transition
	measure for the shape $\lambda$ of $T$. The thick red line shows the cumulative
	function $F_T$ of the tableau $T$. For reference, the smooth black line depicts the \CDF
	of the semicircle distribution, supported on the interval
	$\big[-2\sqrt{\nklatki},\ 2\sqrt{\nklatki}\big]$.}
		
	\label{fig:CDF-100}
\end{figure}
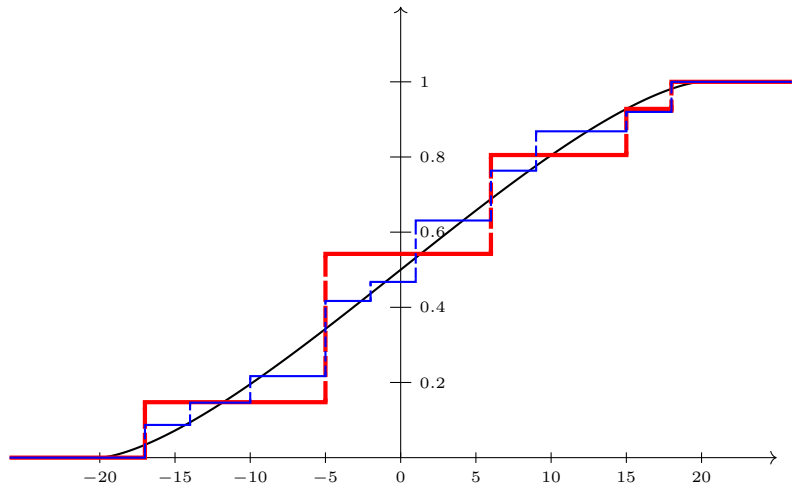

\begin{figure}
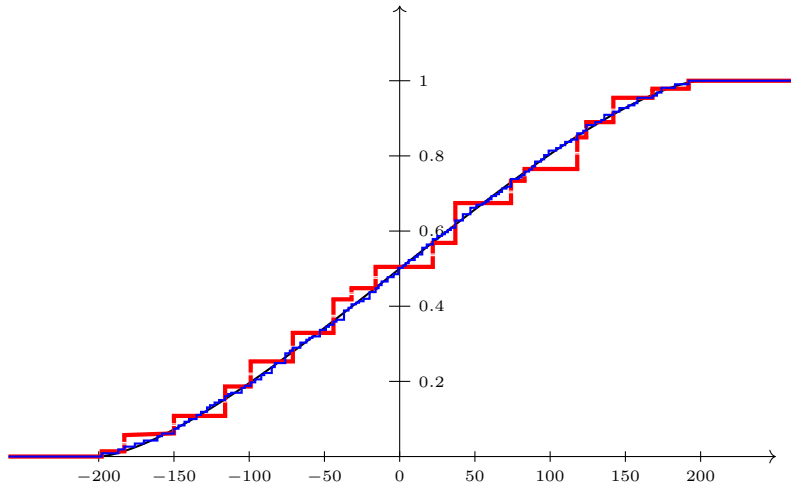

	\subfile{figures/FIGURE-cumulative-function-second.tex}

	\caption{An analogue of \cref{fig:CDF-100} for $n = 10^4$ boxes.}
	\label{fig:CDF-10000}
\end{figure}

As $n \to \infty$, both curves \ref{item:i-tab} and \ref{item:i-tran} converge
to curve \ref{item:i-as}. This convergence is well-established:
\begin{itemize}
	\item For curve \ref{item:i-tab}, it is a reformulation of the result by Romik
and \sniady (\cref{thm:Romik-Sniady-determinism}) about the asymptotic
determinism of the Schensted insertion, as well as a special case of
\cref{thm:determinism-new}.
	
	\item For curve \ref{item:i-tran}, it is a classical result by Kerov on the
transition measure of Plancherel-distributed diagrams
\cite{Kerov1993-transition,KerovBook}.
\end{itemize}

Intriguingly, the apparent rates of convergence differ between the two curves:
\begin{itemize} 
	\item[\ref{item:i-tab}] The observed rate appears to be of order
	$\Theta(n^{-\frac{1}{4}})$. While we lack a rigorous proof for this assertion,
	it aligns well with the findings presented in \cref{sec:fine-asymptotics}.
	
	\item[\ref{item:i-tran}] The convergence seems to be much faster, with the rate
	appearing to be of a significantly smaller order,
	$O(n^{-\frac{1}{2}+\epsilon})$, as conjectured in \cref{conj:transition}.
\end{itemize}

\subsubsection{Future research directions}

An intriguing avenue for future research is the investigation of the expected
number of discontinuities in the cumulative function $F_T$ and their spatial
distribution as the number of boxes tends to infinity. This study is closely
related to the density of infinite geodesics in the percolation tree associated
with Plancherel-distributed infinite standard tableaux on one side, and the
geometry of bumping routes and bumping forests on the other side
\cite{RomikSniady2015,MarciniakMaslankaSniady-bumping}. Understanding these
discontinuities could provide deeper insights into the asymptotic behaviors of
random tableaux. We aim to address this problem in forthcoming work.

\subsection{Single-point evaluation of the cumulative function}

While a comprehensive analysis of the statistical properties of the cumulative
function $F_T$ is beyond the scope of this paper, we focus on a more targeted
approach. Our primary interest lies in understanding the probability
distribution of $F_T(u)$, where the cumulative function is evaluated at a single
point $u$. This approach provides valuable insights while remaining
computationally manageable.

In our recent work \cite{MarciniakSniadyAlternating}, we derived explicit
combinatorial formulas for the cumulants of this probability distribution in
terms of Kerov's transition measure $\mu_\lambda$ of the tableau's
shape~$\lambda$. For the purposes of this paper, we will utilize a few key
results from that work, which are particularly useful for the asymptotic
analysis, as detailed in \cref{lem:alternating-trees} below.

\subsubsection{Cumulants of a probability distribution}

Recall that cumulants \cite{Lauritzen2002,Kendall} are quantities that
provide an alternative to moments for characterizing probability distributions.
For a random variable $X$, its $n$-th cumulant $\kappa_n=\kappa_n(X)$ is defined
as the coefficient in the Taylor expansion of the logarithm of the
moment-generating function:
\[ \log \E\left[e^{tX}\right] = \sum_{n=1}^{\infty} \kappa_n \frac{t^n}{n!}. \]
The first cumulant $\kappa_1$ is the mean, the second cumulant $\kappa_2$ is the
variance, and higher-order cumulants provide information about the shape of the
distribution. The standard normal distribution is uniquely characterized by its
cumulants: its first cumulant (mean) is $0$, its second cumulant (variance) is
$1$, and all higher-order cumulants ($\kappa_n$ for $n \geq 3$) are zero.

\subsubsection{Cumulants of the cumulative function}

\begin{lemma}[{\cite[Equations (3.2) and (3.3) and Corollary 3.4]{MarciniakSniadyAlternating}}]
	\label{lem:alternating-trees} 
	Let $\lambda$ be a fixed Young diagram and $T$ be
	a uniformly random Poissonized tableau of shape $\lambda$. Let $u \in
	\R$ be fixed. The expected value and the variance of $F_T(u)$ are given
	by:
	\begin{align}
		\label{eq:expval}
		\E F_T(u) &= \sum_{x_1 \leq u} \mu_\lambda(x_1), \\
		\label{eq:variance}
		\Var F_T(u) &= \sum_{\substack{x_1 \leq u \\ x_2 > u}} \frac{1}{x_2 - x_1 + 1} \mu_\lambda(x_1) \mu_\lambda(x_2).
	\end{align}
	
	For any $k \geq 1$, the $k$-th cumulant of $F_T(u)$ satisfies the following bound:
	\[ \left| \kappa_k \left( F_T(u) \right) \right| \leq (k-1)! \left[ \cauchy^+_\lambda(u) \right]^{k-1}, \]
	where $\cauchy^+_\lambda(u)$ was defined in \cref{sec:notations}.
\end{lemma}

\section{The second tool: The double cumulative function of a Young diagram}

\label{sec:double-cumulative}

For a Young diagram $\lambda$, we introduce the \emph{double cumulative
	function} $\doublecumulative_\lambda(u, z)$, which serves as a two-dimensional
analogue of the usual \CDF from probability theory.
This function is a tool for analyzing the insertion of a number
into a uniformly random Poissonized tableau of shape~$\lambda$. When a number is
inserted into such a tableau, the double cumulative function
$\doublecumulative_\lambda(u, z)$ combines information about the new box's
position, the inserted number, and the associated probability. This provides a
comprehensive analytical framework for studying Schensted insertion into
Poissonized tableaux.

\subsection{Definition}

Let $\lambda$ be a fixed Young diagram, and let $T$ be a uniformly random
Poissonized tableau of shape $\lambda$. We define the \emph{double cumulative
	function} of $\lambda$ as:
\[ \doublecumulative_\lambda \colon \R \times [0,1] \to [0,1]. \]
This function is given by:
\[ \doublecumulative_\lambda(u, z) = \Pro\big( \uIns(T; z) \geq u \big) =
\Pro\big( F_T(u) \leq z \big) \qquad \text{for } u \in \R, \, z \in [0,1]. \]
It represents the probability that one of the two equivalent conditions
\eqref{eq:why-ft} holds true, or equivalently, that the point $(u, z)$ belongs
to the northwest region with respect to the plot of the cumulative function $u
\mapsto F_T(u)$ (see \cref{fig:french2}). For an example corresponding to
$\lambda = (3,1)$, see \cref{fig:diagram41}.

\subsection{Partitioning of the domain}

The double cumulative function $\doublecumulative_\lambda(u, z)$ exhibits discontinuities along vertical lines where the first coordinate equals the $u$-coordinates of the concave corners of the Young diagram $\lambda$. These vertical lines partition the infinite rectangle $\R \times [0,1]$ into a series of rectangles, which may be finite or semi-infinite. Within each rectangle, $\doublecumulative_\lambda(u, z)$ depends solely on the second variable $z$. This behavior will be further described in \cref{sec:prop-u}. See \cref{fig:4plots} for an illustration.

\subsection{Properties for fixed $u$}
\label{sec:prop-u}
For each fixed $u$, the function 
\begin{equation}
	\label{eq:fixed-u}
	[0,1] \ni z \mapsto
\doublecumulative_\lambda(u, z)
\end{equation}
serves as the \CDF for the random variable
$F_T(u)$. Consequently, $\doublecumulative_\lambda$ is weakly increasing with
respect to the second variable $z$. Results from our recent paper
\cite{MarciniakSniadyAlternating} provide information about this random variable
$F_T(u)$, which we will use later in \cref{lem:convergence-easy}  and \cref{thm:gaussian-profile}
to analyze the behavior of $\doublecumulative_\lambda$ for fixed
values of the first variable. 

\begin{remark}
Since the distribution of the random variable $F_T(u)$ is a mixture of some beta
distributions $\operatorname{Beta}(\alpha,\beta)$ over integers
$\alpha,\beta\geq 1$ with $\alpha+\beta=|\lambda|+2$; it follows that the
function \eqref{eq:fixed-u} is a polynomial, with degree bounded from above by
$|\lambda|$.
\end{remark}

\subsection{Properties for fixed $z$}

For a fixed $z \in [0,1]$, the tail function:
\begin{equation}
	\label{eq:cdf-over-u}
	\R \ni u \mapsto 1 - \doublecumulative_\lambda(u, z) = \Pro\big( \uIns(T; z) < u \big)
\end{equation}
is \emph{almost} equal to the \CDF of the random variable $\uIns(T; z)$. This
property will be crucial for proving \cref{thm:determinism-new,thm:CLT-trans} in
\cref{sec:proofs-of-main-results}. (Strictly speaking, the standard definition
of \CDF involves a \emph{strict} inequality on the right-hand side of
\eqref{eq:cdf-over-u} while we employ a \emph{weak} inequality. This minor
discrepancy will not affect correctness of our arguments presented in
\cref{sec:proofs-of-main-results}.) A simple consequence of this observation is
that the double cumulative function is a weakly decreasing function of the first
variable.

The probability distribution of $\uIns(T; z)$ can be directly inferred from
plots such as those in \cref{fig:4plots}. For a fixed $z$, the vertical unit
interval is partitioned by the plotted curves into sub-intervals. The lengths of
these sub-intervals correspond to the probabilities associated with specific
concave corners of the Young diagram. This visual representation allows for an
intuitive understanding of how the insertion process behaves for different
values of $z$.

\subsection{Navigating the two viewpoints}

Our analytical strategy employs a dual approach to comprehensively examine the
double cumulative function and its implications for the insertion process. We
begin by fixing $u$ and use the results from \cite{MarciniakSniadyAlternating}.
The goal is to derive `two-dimensional' results concerning
$\doublecumulative_\lambda(u, z)$ for an arbitrary choice of $u$ and $z$.
Subsequently, we shift our focus to fixed values of $z$, allowing us to deduce
vital information about the \CDF of the insertion function $\uIns(T; z)$.

\subsection{Examples}
\label{sec:insertion-staircase}

\subsubsection{The example for $\lambda = (3,1)$}
\begin{figure}
    \centering

\begin{tikzpicture}[scale=1.5,rotate=45]

\begin{scope}[shift={(-2,-2)},rotate=-45,xscale=sqrt(0.5),yscale=2]

\fill[fill=yellow!30] (-5,0) rectangle (4.5,1);

\foreach \ya\kolor in {0/0.00, 1/0, 2/0.000011, 3/0.000037, 4/0.000088,
    5/0.00017, 6/0.00029, 7/0.00047, 8/0.00069, 9/0.00098, 10/0.0013, 11/0.0018,
    12/0.0023, 13/0.0029, 14/0.0036, 15/0.0044, 16/0.0054, 17/0.0065, 18/0.0077,
    19/0.0090, 20/0.010, 21/0.012, 22/0.014, 23/0.016, 24/0.018, 25/0.020, 26/0.022,
    27/0.025, 28/0.028, 29/0.031, 30/0.034, 31/0.037, 32/0.041, 33/0.045, 34/0.049,
    35/0.054, 36/0.058, 37/0.062, 38/0.068, 39/0.072, 40/0.078, 41/0.084, 42/0.090,
    43/0.096, 44/0.10, 45/0.11, 46/0.12, 47/0.12, 48/0.13, 49/0.14, 50/0.15, 51/0.16,
    52/0.16, 53/0.18, 54/0.18, 55/0.20, 56/0.20, 57/0.21, 58/0.23, 59/0.24, 60/0.25,
    61/0.26, 62/0.27, 63/0.28, 64/0.30, 65/0.31, 66/0.33, 67/0.34, 68/0.35, 69/0.37,
    70/0.38, 71/0.40, 72/0.41, 73/0.43, 74/0.45, 75/0.46, 76/0.48, 77/0.49, 78/0.52,
    79/0.53, 80/0.55, 81/0.56, 82/0.59, 83/0.61, 84/0.62, 85/0.66, 86/0.67, 87/0.69,
    88/0.72, 89/0.73, 90/0.75, 91/0.78, 92/0.80, 93/0.83, 94/0.84, 95/0.88, 96/0.89,
    97/0.92, 98/0.95, 99/0.97}
{ \fill[fill=blue!60, opacity=\kolor] (0,\ya/100) rectangle +(3,0.01); }

\foreach \ya\kolor in {0/0.00, 1/0.012, 2/0.023, 3/0.035, 4/0.048, 5/0.060,
    6/0.072, 7/0.084, 8/0.096, 9/0.11, 10/0.12, 11/0.13, 12/0.14, 13/0.16, 14/0.17,
    15/0.18, 16/0.20, 17/0.21, 18/0.22, 19/0.23, 20/0.25, 21/0.26, 22/0.27, 23/0.29,
    24/0.30, 25/0.31, 26/0.33, 27/0.34, 28/0.35, 29/0.37, 30/0.38, 31/0.39, 32/0.41,
    33/0.41, 34/0.43, 35/0.45, 36/0.45, 37/0.47, 38/0.48, 39/0.49, 40/0.50, 41/0.52,
    42/0.53, 43/0.55, 44/0.56, 45/0.58, 46/0.58, 47/0.59, 48/0.61, 49/0.62, 50/0.62,
    51/0.64, 52/0.66, 53/0.67, 54/0.67, 55/0.69, 56/0.70, 57/0.72, 58/0.72, 59/0.73,
    60/0.75, 61/0.75, 62/0.77, 63/0.78, 64/0.80, 65/0.80, 66/0.81, 67/0.81, 68/0.83,
    69/0.84, 70/0.84, 71/0.86, 72/0.86, 73/0.88, 74/0.88, 75/0.89, 76/0.89, 77/0.91,
    78/0.91, 79/0.92, 80/0.92, 81/0.94, 82/0.94, 83/0.95, 84/0.95, 85/0.95, 86/0.97,
    87/0.97, 88/0.97, 89/0.98, 90/0.98, 91/0.98, 92/0.98, 93/0.98, 94/1.0, 95/1.0,
    96/1.0, 97/1.0, 98/1.0, 99/1.0} 
{ \fill[fill=blue!60, opacity=\kolor] (-2,\ya/100)
    rectangle +(2,0.01); }

\fill[fill=blue!60] (-5,0) rectangle (-2,1);

\draw[Oranges-5-2,ultra thick, dashed] (-3.5,0) -- (-3.5,1); 
\draw[Oranges-5-3,ultra thick, dashed] (-0.5,0) -- (-0.5,1); 
\draw[Oranges-5-4,ultra thick, dashed] (1.5,0) -- (1.5,1);   
\draw[Oranges-5-5,ultra thick, dashed] (3.5,0) -- (3.5,1);   

\end{scope}

            \coordinate (start) at (-2,-2);
            \coordinate (p1) at ($(start)+(-2,2)$);
            \coordinate (p2) at ($(start)+(2,-2)$);
            
            \coordinate (a) at ($2*(0.70710678118,0.70710678118)$);
            \coordinate (q0) at ($(p1)!(-4,0)!(p2)$);
            \coordinate (q1) at ($(p1)!(-2,0)!(p2)$);
            \coordinate (q2) at ($(p1)!(0,0)!(p2)$);
            \coordinate (q3) at ($(p1)!(3,0)!(p2)$);
            \coordinate (q4) at ($(p1)!(4,0)!(p2)$);

            \draw[decoration={markings,mark=at position 1 with {\arrow[scale=2]{>}}},
            postaction={decorate}] (p1) -- (p2) node[anchor=north west]{$u$};
            
            \draw[decoration={markings,mark=at position 1 with {\arrow[scale=2]{>}}},
            postaction={decorate}] (start) -- ($(start)+1.2*(a)$) node[anchor=west]{$z$};

            \foreach \x in {-3, -2, -1, 0, 1, 2, 3} 
            { \draw ($(p1)!(\x,0)!(p2)$) +(1pt,1pt) -- +(-1pt,-1pt) node[anchor=north] {\x}; }
            
            \foreach \y/\yt in {0.2/0.2,0.4/0.4,0.6/0.6,.8/0.8,1/1} 
            { 
                \draw ($(start)+\y*(a)+(-1pt,1pt)$) -- ($(start)+\y*(a)+(1pt,-1pt)$) node [anchor=west]{\tiny$\yt$};
            }

            \draw[black!40, ultra thick] (4.5,0) node [anchor=south]{\textcolor{black}{$x$}} -- (0,0) -- (0,3.5) node[anchor=east]{\textcolor{black}{$y$}};

            \begin{scope}
                \clip (0,0) -- (3,0) -- (3,1) -- (1,1) -- (1,2) -- (0,2);
                \draw (0,0) grid (3,3); 
            \end{scope}

            \draw[ultra thick] (0,0) -- (3,0) -- (3,1) -- (1,1) -- (1,2) -- (0,2) -- cycle;

            \foreach \x/ \y/ \t in { 3/0/{(0, 0.2) }, 1/1/{(0.2,0.7)}, 0/2/{(0.7,1)} }
            { 
                \fill[blue] (\x,\y) circle (2pt);
                \draw[blue, dashed] (\x,\y) -- ($(p1)!(\x,\y)!(p2)$ );
            }

\end{tikzpicture}
    
\caption{Density plot of the double cumulative function
	$\doublecumulative_\lambda(u, z)$ for the Young diagram $\lambda = (3,1)$.
	Bright yellow indicates areas where the insertion function takes values close to
	$0$, while dark blue indicates values close to $1$. The plots of the double
	cumulative function along the four vertical dashed lines are shown in
	\cref{fig:4plots}.}

    \label{fig:diagram41}
\end{figure}
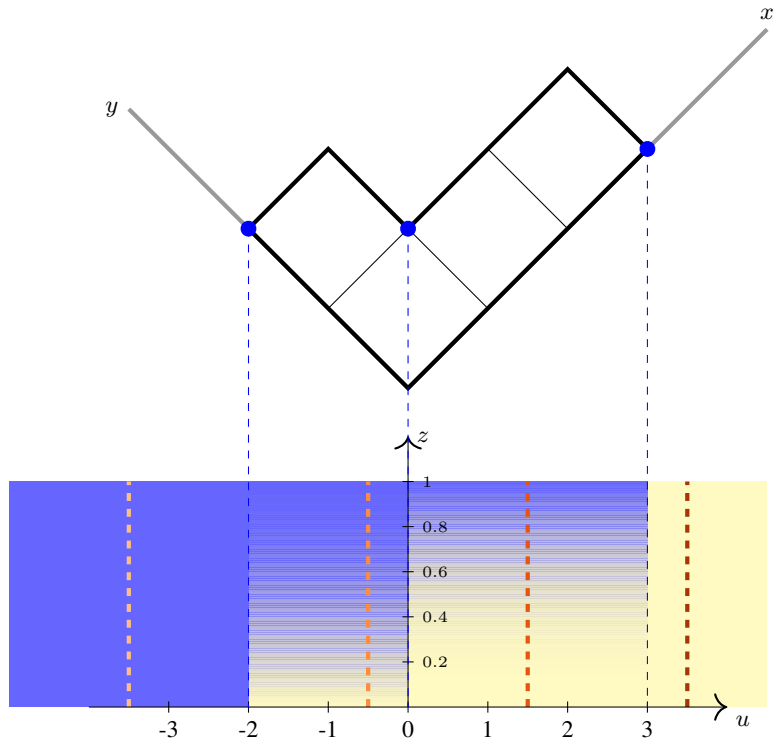

\begin{figure}
    \begin{tikzpicture}[scale=5]
        \draw[dotted] (0,0) grid[step=0.2] (1,1);
        \draw[->] (-0.1, 0) -- (1.1, 0) node[right] {$z$};
        \draw[->] (0,-0.1) -- (0, 1.1) node[above] {$\doublecumulative(u,z)$};
        
        \draw[Oranges-5-2, ultra thick] (0,1) -- (1,1);
        \draw[Oranges-5-2] (0.3,1.07) node {$u\leq -2$};

        \draw[domain=0:1, smooth, variable=\x, Oranges-5-3, ultra thick] plot ({\x}, {-341/250*(\x - 1)*(\x - 1)*(\x - 1)*\x + 2013/500*(\x - 1)*(\x - 1)*\x*\x - 4*(\x - 1)*\x*\x*\x + \x*\x*\x*\x});
        \draw[Oranges-5-3] (0.3,0.75) node {$-2\leq u < 0$};

   \draw[domain=0:1, smooth, variable=\x, Oranges-5-4, ultra thick] plot ({\x}, {-329/250*(\x - 1)*\x*\x*\x + \x*\x*\x*\x});
\draw[Oranges-5-4] (1.05,0.5) node {$0\leq u < 3$};

       \draw[Oranges-5-5, ultra thick] (0,0) -- (1,0);
\draw[Oranges-5-5] (0.9,0.1) node {$u\geq 3$};

    \foreach \y/\yt in {0.2/0.2,0.4/0.4,0.6/0.6,.8/0.8,1/1} 
{    \draw (0,\y) +(0.2pt,0) -- +(-0.2pt,0) node [anchor=east]{\tiny$\yt$};
    \draw (\y,0) +(0,0.2pt) -- +(0,-0.2pt) node [anchor=north]{\tiny$\yt$};
}

    \end{tikzpicture}
    
	\caption{Double cumulative function plots for Young diagram $\lambda = (3,1)$.
	This figure extends the example from \cref{fig:diagram41}, showing $z \mapsto
	\doublecumulative_\lambda(u,z)$ for various values of $u$. These plots
	illustrate how the function behaves across different $u$ values, providing
	insight into the insertion probabilities for Poissonized tableaux of this
	shape.}
	
	\label{fig:4plots}
\end{figure}
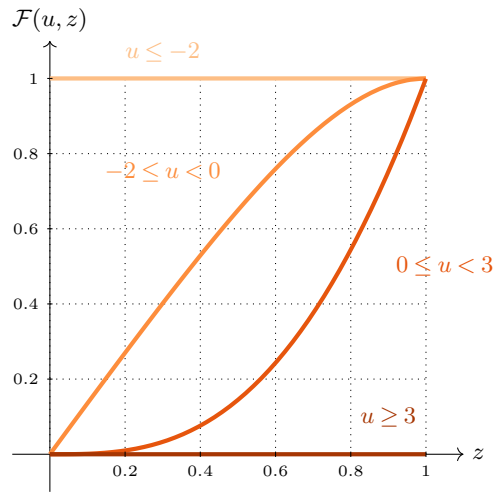

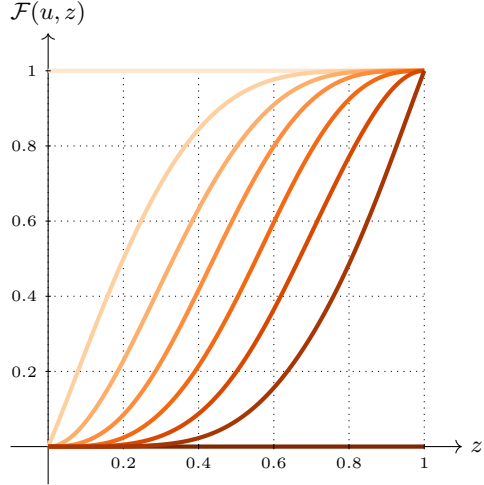
\begin{figure}
    \begin{tikzpicture}[scale=5]
        \draw[dotted] (0,0) grid[step=0.2] (1,1);
        \draw[->] (-0.1, 0) -- (1.1, 0) node[right] {$z$};
        \draw[->] (0,-0.1) -- (0, 1.1) node[above] {$\doublecumulative(u,z)$};
        
        \draw[Oranges-9-2, ultra thick] (0,1) -- (1,1);
        \draw[Oranges-9-9, ultra thick] (0,0) -- (1,0);

        \foreach \y/\yt in {0.2/0.2,0.4/0.4,0.6/0.6,.8/0.8,1/1} 
        {    \draw (0,\y) +(0.2pt,0) -- +(-0.2pt,0) node [anchor=east]{\tiny$\yt$};
            \draw (\y,0) +(0,0.2pt) -- +(0,-0.2pt) node [anchor=north]{\tiny$\yt$};
        }

        \foreach \y in {3,4,5,6,7,8}
        { \draw[Oranges-9-\y,ultra thick] plot[smooth] file {figures/insertionfunction/kolor\y.txt}; }
    \end{tikzpicture}
    
   \caption{Double cumulative function plots for the staircase Young diagram
	$\lambda = (6,5,4,3,2,1)$. This diagram features $7$ concave corners, whose
	$u$-coordinates partition the real line into $8$ intervals (finite or
	infinite). The topmost bright curve (constant at $1$) corresponds to the
	leftmost (infinite) interval, while the bottommost dark curve (constant at $0$)
	represents the rightmost (infinite) interval. This figure is analogous to
	\cref{fig:4plots}.}
    
   \label{fig:8plots}
\end{figure}

\begin{figure}
    \centering
    \begin{tikzpicture}[xscale=0.3,yscale=10]
        \fill[yellow!30] (-22,0) rectangle (22,1);
        \fill[blue!60] (-22,0) rectangle (-19,1);
        
        \foreach \plik in {1,2,...,19} {
            \loaddata{figures/density/density\plik.txt}

            \foreach \ya\kolor in \loadeddata
            { \fill[fill=blue!60, opacity=\kolor] (19-2*\plik,\ya) rectangle +(2,0.005); } }

        \draw[red,ultra thick,xscale=19] plot[smooth] file {figures/data/sinus.txt};

        \loaddata{figures/data/cloud_19.txt}
        \foreach \z\u in \loadeddata
        {\draw[ultra thick,white] (\u,\z) +(-0.2,0.01) -- +(0.2,-0.01)
            +(0.2,0.01) -- +(-0.2,-0.01); }
        
        \foreach \z\u in \loadeddata
        {\draw[black] (\u,\z) +(-0.2,0.01) -- +(0.2,-0.01)
            +(0.2,0.01) -- +(-0.2,-0.01); }
        
        \draw[->] (0,-0.1) -- (0,1.1);
        \foreach \y/\yt in {0.2/0.2,0.4/0.4,0.6/0.6,.8/0.8,1/1} 
        { 
            \draw (0,\y) +(-10pt,0) -- +(10pt,0) node [anchor=west]{\small$\yt$};
        }

        \draw[->] (-22,0) -- (22,0);
        \foreach \y in {-20,-19,...,20} 
        { 
            \draw (\y,0) +(0,0.1pt) -- +(0,-0.1pt);
        }
        
        \foreach \y/\yt in {-20/-20,-15/-15,-10/-10,-5/-5,5/5,10/10,15/15,20/20} 
        { 
            \draw (\y,0) +(0,0.3pt) -- +(0,-0.3pt) node [anchor=north]{\small$\yt$};
        }

    \end{tikzpicture}
    
    \caption{Double cumulative function visualization for the staircase Young diagram $\lambda = (19,18,\dots,1)$. This figure extends the analysis presented in \cref{fig:diagram41}, illustrating the behavior for a larger staircase shape.
    The crosses represent points $(\uIns(T;z), z)$, where
	$z$ is sampled from the uniform distribution $U(0,1)$, and $T$ is an
	independently sampled random Poissonized tableau of shape $\lambda$. The
	thick red curve depicts the \CDF of the (dilated)
	arcsine law, which is the limit measure of transition measures for
	large staircase tableaux.}
    
    \label{fig:19plot}
\end{figure}
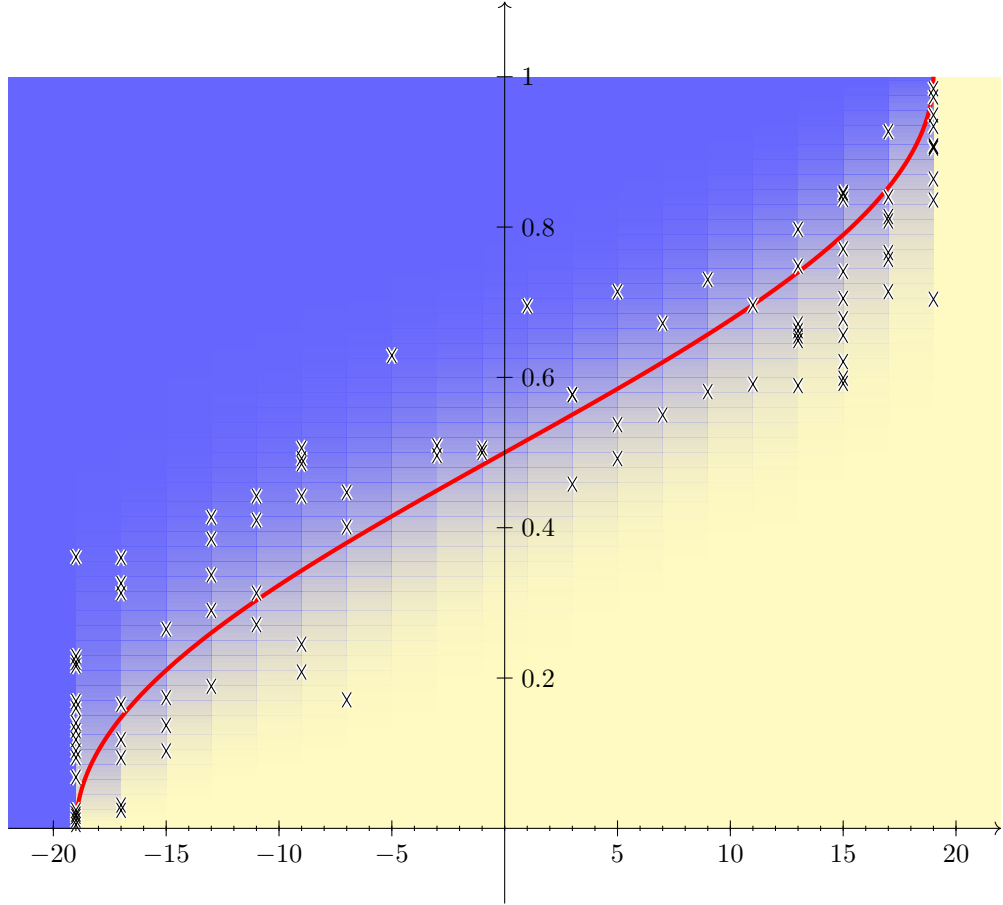

The double cumulative function for $\lambda = (3,1)$ is visualized in
\cref{fig:diagram41} as a density plot. The cumulative function plots along the
four vertical lines are shown in \cref{fig:4plots}.

\subsubsection{Staircase tableaux}

Expanding on our earlier discussion in \cref{sec:staircase}, we present
visualizations of the double cumulative function for two distinct staircase
diagrams in \cref{fig:8plots,fig:19plot}.

\section{First-order asymptotics of the double cumulative function}
\label{sec:first-order-double-function}

The phenomenon we aim to elucidate is vividly exemplified by large staircase
diagrams, as depicted in \cref{fig:19plot}. Let us examine the density
plot of the double cumulative function $\doublecumulative_{\lambda}\colon \R
\times [0,1] \to [0,1]$ for a Young diagram $\lambda$. The domain, an
infinite rectangle $\R \times [0,1]$, is bisected by the graph of the
cumulative function $K_\lambda: \R \to [0,1]$, delineating two distinct
regions:
\begin{itemize}
	\item the southeast region, characterized by small values of $\doublecumulative_{\lambda}$ (approaching 0),
	\item the northwest region, distinguished by large values of $\doublecumulative_{\lambda}$ (nearing 1).
\end{itemize}
The following proposition formalizes this observation:

\begin{proposition}
\label{lem:convergence-easy} 
Consider a sequence $(\lambda^{(n)})$ of
random Young diagrams. We posit the existence of a probability measure $\nu$ on
$\R$ such that for each $u \in \R$ where the cumulative
distribution function $F_\nu$ is continuous, the following convergence in
probability holds:
\[   K_{\lambda^{(n)}}\left( \sqrt{n} \ u \right) \xrightarrow[n\to\infty]{P} F_\nu(u). \]

For $u \in \R$ and $z \in [0,1]$:

If $z < F_\mu(u^-) = \lim_{v \to u^-} F_\mu(v)$, then:
\[ \doublecumulative_{\lambda^{(n)}}\left(\sqrt{n} \ u, z\right) \xrightarrow[n\to\infty]{P} 0. \]

Conversely, if $z > F_\mu(u)$, then:
\[ \doublecumulative_{\lambda^{(n)}}\left(\sqrt{n}\ u, z\right) \xrightarrow[n\to\infty]{P} 1.\]
\end{proposition}

\newcommand{\uOne}{u'}
\newcommand{\uTwo}{u}
\newcommand{\uPrime}{u''}

We begin with a lemma concerning the double cumulative function
$\doublecumulative_{\lambda}$ in a deterministic setting. The key idea is that
when a point $(u, z)$ lies below the graph of $K_\lambda$ and is sufficiently
separated from it both vertically and horizontally, there exists an upper bound
on the value of $\doublecumulative_{\lambda}(u, z)$.

\begin{lemma}
	\label{eq:double-czebyszew}
	Let $\lambda$ be a Young diagram, and consider real numbers $\uOne < \uTwo$ and $z$ satisfying:
	\[ 0 \leq z < K_{\lambda}(\uOne). \]
	
	Then the following inequality holds:
	\[ 0 \leq \doublecumulative_{\lambda}(\uTwo, z) \leq
	\frac{1}{(\uTwo-\uOne) \big(K_{\lambda}(\uOne) -z\big)^2}.
	\]
\end{lemma}

\begin{proof}
	Let $T$ be a uniformly random Poissonized tableau of shape $\lambda$. The
explicit formula \eqref{eq:variance} for the variance of $F_{T}(v)$ implies:
	\[ \int_{\uOne}^{\uTwo} \Var F_{T}(v) \dif v \leq 1.\]
By the mean value theorem for integrals, there exists $\uPrime \in [\uOne,
\uTwo]$ such that:
	\[ \Var F_{T}(\uPrime) \leq \frac{1}{\uTwo-\uOne}. \]
	
	From \eqref{eq:expval}, we know the mean value satisfies:
	\[ \E F_{T}(\uPrime) = K_{\lambda}(\uPrime) > z. \]	
	This allows us to apply the Bienaymé--Chebyshev inequality:
	\begin{multline*} 
		0\leq \doublecumulative_{\lambda}(\uTwo, z) \leq
		\doublecumulative_{\lambda}(\uPrime, z) =
		\Pro\left( F_{T}(\uPrime) \leq z \right) \leq   \\
		\Pro\left( \left| F_{T}(\uPrime)- \E F_T(\uPrime) \right| \geq  \E F_{T}(\uPrime) -z  \right) \leq  \\
		\frac{\Var F_{T}(\uPrime)}{ \left( \E F_{T}(\uPrime) - z\right)^2 } \leq 
		\frac{1}{(\uTwo-\uOne)  \left( K_{\lambda}(\uOne) - z\right)^2 },  
	\end{multline*}
	as required.
\end{proof}

\begin{proof}[Proof of \cref{lem:convergence-easy}]
	We begin by addressing the first part of the claim, assuming $z < F_\mu(u^-)$.
	There exists $u' < u$ such that $F_\nu(u') > z$. Without loss of generality, we
may assume that $u'$ is a point of continuity for $F_\nu$, ensuring:
	\[ K_{\lambda^{(n)}}\left( \sqrt{n}\ u'\right) \xrightarrow[n\to\infty]{P} F_\nu(u'). \] 
	
	Applying \cref{eq:double-czebyszew}, we obtain:
	\[   0 \leq 
	\doublecumulative_{\lambda^{(n)}}\left(\sqrt{n}\ u, z\right) 
	\leq
	\begin{cases}   
		\frac{1}{(u-u') \sqrt{n}\ \left( K_{\lambda^{(n)}}\left(\sqrt{n}\ u'\right) - z\right)^2 }  
		& \text{if } K_{\lambda^{(n)}}\left( \sqrt{n}\ u'\right) > z,  \\
		1 & \text{otherwise}.
	\end{cases}
	\]   
	The random variable on the right-hand side converges in probability to zero,
completing the proof of the first part.
	
The second part of the claim is established through a parallel line of
reasoning, albeit with an adjustment of certain inequality directions.
This modification is applied to the function:
\[ (u,z) \mapsto 1- \doublecumulative_{\lambda^{(n)}}\left(\sqrt{n}\ u,
z\right) \]
which pertains to the complementary events.
\end{proof}

\section{Fine asymptotics of the double cumulative function}
\label{sec:fine-asymptotics}

The phenomenon we aim to describe is exemplified by large staircase
diagrams, as vividly depicted in Figure \ref{fig:19plot}. As previously
discussed in Section \ref{sec:first-order-double-function}, the domain of the
density plot for the double cumulative function $\doublecumulative_{\lambda}\colon 
\R \times [0,1] \to [0,1]$ is partitioned by the graph of the cumulative
function $K_\lambda\colon \R \to [0,1]$. This partition delineates two
distinct regions: the southeast region, characterized by values approaching $0$,
and the northwest region, distinguished by values nearing~$1$.

Of particular intrigue is the narrow band immediately adjacent to the plot of
$K_\lambda$, where the transition between these regions unfolds. As the
magnitude of $\lambda$ increases, the relative width of this transition zone
converges to zero, creating a fascinating boundary layer phenomenon. The crux of
this section, encapsulated in Theorem \ref{thm:gaussian-profile}, reveals a
remarkable property: within this transition area, the double cumulative function
can be approximated by the \CDF of the standard normal distribution.

\subsection{New coordinate system} 

We employ the notations from \cref{thm:CLT-trans}. To provide a formal,
quantitative description of the behavior of
$\doublecumulative_{\lambda^{(n)}}$, we introduce a new coordinate system 
on its domain $\R\times [0,1]$. The new coordinates $(\stala, \xi)$
correspond to a point $(U_{n,\stala}, Z_{n,\stala}) \in \R \times [0,1]$,
where:
\begin{align*}
	U_{n,\stala} &= \sqrt{n}\ u_0 + \sqrt[4]{n}\ \stala, \\
	Z_{n,\stala} &= z_0 + \frac{\xi}{\sqrt[4]{n}}.
\end{align*}
This coordinate system is centered around the point $(\sqrt{n}\ u_0, z_0)$ that
represents a first-order approximation for the insertion $T^{(n)} \leftarrow
z_0$. Notably, the units of this system are scaled in proportion to
$\sqrt[4]{n}$ or its inverse, which aligns with the scale of problem.

\subsection{The Gaussian profile of the double cumulative function} 

The following theorem encapsulates the behavior of the
double cumulative function in the transition area:
\begin{theorem}
	\label{thm:gaussian-profile}
Under the assumptions of \cref{thm:CLT-trans}
	\begin{equation}
		\label{eq:area-of-transition}
		  \doublecumulative_{\lambda^{(n)}}
	  \left(U_{n,q}  ,\ 
	        Z_{n,\xi} \right) \xrightarrow[n\to\infty]{P}
	\Phi\left(  \frac{\xi- f_0 q}{\sqrt{\Force}} \right)
	\end{equation}        
   for each $q,\xi\in\R$, 	where $\Phi(x)$ denotes the \CDF of the standard
   normal distribution, given by
   \begin{equation}
   	\Phi(x) = \frac{1}{\sqrt{2\pi}} \int_{-\infty}^x e^{-\frac{t^2}{2}} \, \dif t.
   \end{equation}
\end{theorem}
The remaining part of this section is devoted to the proof.

\begin{remark}
We shall now examine the geometric properties of the surface defined by the mapping
\[
(q, \xi) \mapsto \Phi\left( \frac{\xi - f_0 q}{\sqrt{\Force}} \right),
\]
which emerges as the limiting function on the right-hand side of equation
\eqref{eq:area-of-transition}. A notable characteristic of this surface is that
its level curves are straight lines, expressible in the form
\[
\xi = f_0 q + c,
\]
where \( c \) denotes a constant. These lines are  parallel to the tangent line
to the local approximation of the graph of \( K_{\lambda^{(n)}} \) given by the
\assum{item:local2} in \cref{thm:CLT-trans}. The aforementioned
geometric property is visually discernible in the illustrative example
presented in \cref{fig:19plot}.
\end{remark}

\subsection{Case study: large staircase tableaux}
\label{sec:staircase3}

We revisit the staircase tableaux example introduced in \cref{sec:staircase}, now examining it through the lens of 
\cref{thm:gaussian-profile}.

A consequence of \cref{thm:gaussian-profile} is that for a large Young diagram
$\lambda$ and fixed $u$, the function $z \mapsto \doublecumulative_{\lambda}(u,
z)$ can be approximated by the \CDF of a normal distribution. The parameters of
this approximating distribution—mean and variance—are determined by equations
\eqref{eq:expval} and \eqref{eq:variance}, respectively.

A key insight emerges for staircase tableaux: the variance, up to a geometric
scaling factor, closely approximates the interaction energy
\eqref{eq:arc-sine-force} of the arcsine distribution. This distribution arises
as the limit transition measure for large staircase diagrams (see \cref{sec:staircase}). Notably, for this measure, the interaction energy remains
constant across the entire spectrum of valid $u$ values.

\Cref{fig:8plots} provides visual corroboration of this concept. The
depicted family of curves resembles a series of shifted plots of a single scaled
\CDF of the standard normal distribution. This aligns with our theoretical
expectations of the normal distribution approximation with a fixed variance. A
similar phenomenon is observed in \cref{fig:19plot}: for a wide range of
$u$ values, the color profile along a fixed $u$ coordinate appears similar, up
to a vertical shift.

Moreover, for large staircase diagrams, the mean value of the approximating
normal distribution, given by \eqref{eq:expval}, is approximately the \CDF of the
aforementioned arcsine measure. This relationship is visually evident in 
\cref{fig:19plot}, where the region of significant color change closely follows
the plot of the arcsine measure's \CDF.

\subsection{Normalized cumulative function}

For an integer $n \geq 1$ and $\stala \in \R$ we consider the random variable
\[ X_{n,\stala} = \frac{\sqrt[4]{n}}{\sqrt{\Force}} 
\left[ F_{T^{(n)}}\big(U_{n,\stala} \big) - \left(z_0+ \frac{f_0 \stala}{\sqrt[4]{n}  } \right) \right] 
= 
\frac{\sqrt[4]{n}}{\sqrt{\Force}} 
\left[ F_{T^{(n)}}\big(U_{n,\stala}\big) - z_0 \right] - \frac{f_0 \stala}{\sqrt{\Force}}.
\]
This random variable is the result of an affine transformation applied to
$F_{T^{(n)}}(U_{n,\stala})$, which is the cumulative  function of
$T^{(n)}$ evaluated at the point with the new coordinate $\stala$.
As we shall see, the constants in this
affine transformation were chosen in such a way that $X_{n,\stala}$ behaves
asymptotically like a standard normal random variable.

The double cumulative function $\doublecumulative_{\lambda^{(n)}}$ is
intimately linked to $X_{n,\stala}$ through the following lemma:
\begin{lemma}
	\label{lem:link-to-x}
	For any $q,\xi\in\R$
\[ \doublecumulative_{\lambda^{(n)}}( U_{n,q}, Z_{n,\xi} ) = 
           \Pro\left[ 	X_{n,\stala} \leq  \frac{\xi - f_0 \stala}{\sqrt{\Force}} \right]. \]
\end{lemma}
\begin{proof}

	The value of the double cumulative function concerns the insertion
$T^{(n)}\leftarrow Z_{n,\xi}$ and represents the probability that one of the
following equivalent conditions occurs:
\begin{itemize}
	\item The insertion position in the transformed coordinates is greater than $\stala$.
	\item The original coordinate of the insertion exceeds $U_{n,\stala}$.
	\item The cumulative function $F_{T^{(n)}}$ evaluated at $U_{n,\stala}$ is less than or equal to $Z_{n,\xi}$.
	\item The random variable $X_{n,\stala}$ is at most $\frac{\xi - f_0 \stala}{\sqrt{\Force}}$.
\end{itemize}

Formally, this can be expressed as:
\begin{multline}
	\label{eq:4conditions}
	\frac{\uIns(T^{(n)}; Z_{n,\xi}) - \sqrt{n}\ u_0}{\sqrt[4]{n}} \geq \stala \quad \iff \quad  
	\uIns(T^{(n)}; Z_{n,\xi}) \geq U_{n, \stala} \quad \iff \\
	F_{T^{(n)}}\big(U_{n,\stala}\big) \leq Z_{n,\xi} \quad 
	\iff \quad  
	X_{n,\stala} \leq  \frac{\xi - f_0 \stala}{\sqrt{\Force}},
\end{multline}
see \eqref{eq:why-ft}.
\end{proof}

\subsection{Randomized probability distribution of $X_{n,q}$}

Our objective is to analyze the probability distribution of $X_{n,\stala}$.
However, the assumptions in \cref{thm:CLT-trans} prove insufficient to draw
conclusions about its behavior for an arbitrary constant $\stala \in
\R$ (see \cref{rem:spreading}). To overcome this limitation, we
introduce randomness to the spatial parameter $\stala$.

\subsubsection{Randomized approach}

Instead of examining $X_{n,\stala}$ for a fixed transformed coordinate
$\stala$, we study $X_{n,\zmienna}$, where $\zmienna$ is a standard Gaussian 
random variable independent from $T^{(n)}$. Without loss of generality, we
assume that this random variable coincides with the Gaussian random variable
from \cref{thm:CLT-trans}.

This probabilistic approach allows us to leverage the power of probabilistic
methods by considering average behavior over a distribution of spatial
parameters, and to shift our goal from proving a result that holds for each
value of $\stala$ to demonstrating that the probability of our claimed result
converges to $1$ for a random choice of $\stala$.

\subsubsection{Formal framework}

Let $\sigmafield$ denote the $\sigma$-field generated by the random variables
$\zmienna$ and $\lambda^{(n)}$. Our target result is encapsulated in the
following result:
\begin{proposition}
	\label{lem:conditional-normal} 
	
The conditional probability distribution of the random variable
$X_{n,\zmienna}$, given $\sigmafield$, converges in probability to the standard
Gaussian measure. More precisely, the supremum distance between the
corresponding \CDFs converges in probability to $0$:
	\begin{equation}
		\label{eq:sup-distance}
		\sup_t \left| \Pro\big( X_{n,\zmienna} \leq t\  \big| \ \sigmafield \big) - \Phi(t) \right|
		        \xrightarrow[n\to\infty]{P} 0,
	\end{equation}
	where $\Phi(t)$ is the \CDF of the standard normal distribution.
\end{proposition}

This result formalizes the convergence of the conditional probability
distribution, which is a random probability measure on the real line, to the
standard Gaussian measure. The use of the supremum distance between cumulative
distribution functions provides a rigorous metric for this convergence.

The proof of this result is postponed to \cref{sec:proof-of-conditional-normal}.

\subsubsection{Two spacial coordinate systems}

In the proof of \cref{lem:conditional-normal}, it is necessary to employ the
assumptions from \cref{thm:CLT-trans}, which utilize slightly different
notation. Recall that the random variable \(\uwithnoise\), defined in
\eqref{eq:uwithnoise}, was used in \cref{thm:CLT-trans} to refer to the
rescaled diagram \(\omega_n = \omega_{\frac{1}{\sqrt{n}} \lambda^{(n)}}\). This
variable is clearly related to the random variable \(U_{n,\zmienna}\) by a
simple scaling:
\begin{equation}
	\label{eq:uwithnoise-related}
	\uwithnoise = \frac{1}{\sqrt{n}} U_{n,\zmienna}.
\end{equation}
Since \(U_{n,\zmienna}\) will be used to refer to the original, non-rescaled
diagram \(\lambda^{(n)}\), it follows that \(\uwithnoise\) and
\(U_{n,\zmienna}\) represent the same point in the dia gram but in two different
coordinate systems.

\subsubsection{Conditional moments and conditional cumulants}

The proof strategy for \linebreak\cref{lem:conditional-normal} employs an adaptation of
the method of moments to a conditional framework. Rather than directly
examining the conditional moments $\E[X_{n,\zmienna}^k \mid \sigmafield]$, we
shall pursue a more nuanced approach by focusing on the \emph{conditional
	cumulants} of $X_{n,\zmienna}$. This conditional setup can be intuitively
understood as fixing the values of the spacial variable $\zmienna$ and the
Young diagram $\lambda^{(n)}$. Consequently, \cref{lem:alternating-trees}
becomes applicable, the implications of which we shall elucidate below.

Formula \eqref{eq:expval} from \cref{lem:alternating-trees}  and the
relationship  \eqref{eq:uwithnoise-related} imply that the conditional expected
value is given by
\begin{multline} 
	\label{eq:cond-cumu-1}
	\E \giventhat*{X_{n,\zmienna}}{\sigmafield}  =
\frac{\sqrt[4]{n}}{\sqrt{\Force}} 
\left[ \E \giventhat*{ F_{T^{(n)}}\big(U_{n,\zmienna}\big) }{\sigmafield}  - z_0 \right] 
- \frac{f_0 \zmienna}{\sqrt{\Force}}  =
\\
\frac{\sqrt[4]{n}}{\sqrt{\Force}} 
\left[ K_{\omega_n}(\uwithnoise)  - z_0 \right] 
  - \frac{f_0 \sqrt[4]{n} (\uwithnoise-u_0)}{\sqrt{\Force}}  
\xrightarrow[n\to\infty]{P} 0
\end{multline}
which converges in probability to zero by \assum{item:local2}.

Formula \eqref{eq:variance} and the relationship  \eqref{eq:uwithnoise-related}
imply that the conditional variance is given by
\begin{equation}
	\label{eq:cond-cumu-2}
	 \Var\giventhat*{X_{n,\zmienna}}{\sigmafield} = 
\frac{\sqrt[2]{n}}{\Force} \Var 
 F_{T^{(n)}}\big(U_{n,\zmienna} \big) 
\xrightarrow[n\to\infty]{P} 1 
\end{equation}
where the convergence is a consequence of the \assum{item:energy2}.

For $k\geq 3$ the absolute value of the corresponding cumulant can be
bounded thanks to \cref{lem:alternating-trees} as follows:
\begin{multline}
	\label{eq:cond-cumu-3}
\bigg| 	\kappa_k \giventhat*{X_n}{\sigmafield} \bigg| =
\left( \frac{\sqrt[4]{n}}{\sqrt{\Force}}  \right)^k	
  \left| \kappa_k \giventhat*{F_{T^{(n)}}\big(U_{n,\zmienna} \big) }{\sigmafield} \right| \leq \\
\left( \frac{\sqrt[4]{n}}{\sqrt{\Force}}  \right)^k  
  (k-1)! \left[ \cauchy^+_\lambda(U_{n,\zmienna} ) \right]^{k-1} 
	\xrightarrow[n\to\infty]{P} 0
\end{multline}
which converges in probability to $0$  by 
\assum{item:regular2}.

The fundamental implication of equations
\eqref{eq:cond-cumu-1}--\eqref{eq:cond-cumu-3} is that the sequence of
conditional cumulants of $X_{n,\zmienna}$ converges in probability to that of
the standard normal distribution as $n\to\infty$.

\subsubsection{Finite number of moments is enough}

\begin{lemma}
	\label{lem:convergence-to-normal} 
	
For any $\epsilon > 0$, there exist an
integer $n_0 \geq 1$ and $\delta > 0$ such that the following holds: 

Let $\mu$
be an arbitrary probability measure on the real line whose first $n_0$ moments
(or cumulants) exist and are $\delta$-close to the corresponding moments (or
cumulants) of the standard normal distribution $N(0,1)$. Then, the supremum
distance between the \CDF of $\mu$ and that of the
standard normal distribution is less than $\epsilon$. Formally,
	\begin{equation}
		\sup_{t \in \R} |F_\mu(t) - \Phi(t)| < \epsilon.
	\end{equation}
\end{lemma}
\begin{proof}
	We proceed by contradiction. Suppose the statement were false. Then there
would exist a sequence of probability measures that converges in moments to
the standard normal distribution but does not converge weakly to it. However,
this is impossible, as the normal distribution is uniquely determined by its
moments (a property known as moment determinacy). The contradiction
establishes the lemma.
\end{proof}

\newcommand{\event}{A}

\subsubsection{Proof of \cref{lem:conditional-normal}}
\label{sec:proof-of-conditional-normal}

\begin{proof}[Proof of \cref{lem:conditional-normal}]

Let $\epsilon>0$ be arbitrary, let $n_0$ and $\delta$ be the values provided by
\cref{lem:convergence-to-normal}. Equations
\eqref{eq:cond-cumu-1}--\eqref{eq:cond-cumu-3} imply that the probability that
the first $n_0$ conditional cumulants of $X_n$ are $\delta$-close to their
counterparts for the standard normal distribution converges to $1$.
Consequently, by \cref{lem:convergence-to-normal}, we have that the probability
of the event
\[
	\sup_t \left| \Pro\big( X_{n,\zmienna} \leq t\  \big| \ \sigmafield \big) - \Phi(t) \right| < \epsilon
\]
converges to $1$ as $n\to\infty$, as required.
\end{proof}

\subsection{Conditioning on a bounded interval}

The conclusion of \cref{lem:conditional-normal} remains valid under additional
conditioning. Specifically, for any constants $\stala_1 < \stala_2$, we may
condition on the event $\{\stala_1 < \zmienna < \stala_2\}$ without affecting
the lemma's result. Consequently, we can proceed with our analysis under the
assumption that the random variable $\zmienna$ is confined to the interval
$(\stala_1, \stala_2)$ without loss of generality.

\subsection{Conditional \CDF of $X_{n,\zmienna}$}

The following result bears a strong resemblance to \cref{lem:conditional-normal}. However, a key distinction lies in the nature of the conditional probabilities 
considered:
\begin{itemize}
	\item In \cref{lem:conditional-normal}, the probability is conditioned on both the choice of $\lambda^{(n)}$ and $\zmienna$.
	\item In contrast, \cref{lem:ccdf} considers a probability conditioned solely on the choice of $\lambda^{(n)}$.
\end{itemize}

\begin{proposition}
	\label{lem:ccdf}
	For each $t\in\R$
\begin{equation}
	\label{eq:converges-to-Psi}
		 \Pro\giventhat*{ X_{n,\zmienna} \leq  t }{\lambda^{(n)}} \xrightarrow[n\to\infty]{P} \Phi(t).
\end{equation}
\end{proposition}
\begin{proof}
	The difference of the random variable on the left-hand side of
\eqref{eq:converges-to-Psi}, and the hypothetical limit:
	\begin{equation}
		\label{eq:cond-cdf-is-Psi}
		\Pro\giventhat*{ X_{n,\zmienna} \leq  t }{\lambda^{(n)}} - \Phi(t)  = 
	\E\giventhat*{  
		\Pro\giventhat*{ 
			X_{n,\zmienna} \leq  t }{\sigmafield}
		- \Phi(t) 
	}{\lambda^{(n)}}
	\end{equation}	
	can be obtained by taking the conditional expected value in two steps. Its
$L^1$ norm is bounded from above by the $L^1$ norm of the random variable
within the most external conditional expected value on the right-hand side, which is
	\[ 
	\E \left[ \ 
	\left|	\Pro\giventhat*{ 
		X_{n,\zmienna} \leq  t }{\sigmafield}
	- \Phi(t)
	\right| \ \right]
	\]
	which converges to zero by \cref{lem:conditional-normal} 
\end{proof}

\subsection{The consequences of monotonicity of the double cumulative function}

The double cumulative function $\doublecumulative_{\lambda^{(n)}}$ exhibits weak
monotonicity with respect to its first coordinate. Consequently, we can
establish a lower bound for \eqref{eq:area-of-transition} by averaging over the
random variable $\zmienna \in (q_1, q_2)$. This bound can be expressed as
follows:
\begin{multline}
	\label{eq:magic}
\doublecumulative_{\lambda^{(n)}}(U_{n,q_1}, Z_{n,\xi}) \geq 
\E \giventhat*{ \doublecumulative_{\lambda^{(n)}}(U_{n,\zmienna}, Z_{n,\xi}) }{\lambda^{(n)}} =
\Pro\giventhat*{ X_{n,\zmienna} \leq  \frac{\xi - f_0 \zmienna}{\sqrt{\Force}} }{\lambda^{(n)}} \geq \\
\Pro\giventhat*{ X_{n,\zmienna} \leq  \frac{\xi - f_0 q_2}{\sqrt{\Force}} }{\lambda^{(n)}} 
           \xrightarrow[n\to\infty]{P} \Phi\left( \frac{\xi - f_0 q_2}{\sqrt{\Force}} \right)
\end{multline}
where the equality follows from \cref{lem:link-to-x}, and the convergence in
probability is a consequence of \cref{lem:ccdf}.

Let $\epsilon > 0$. We set $q_1 := q$ and choose $q_2 > q$ such that the
right-hand side of \eqref{eq:magic} is $\epsilon$-close to $\Phi\left( \frac{\xi
	- f_0 q}{\sqrt{\Force}} \right)$. This choice ensures that the probability of
the event
\[ \doublecumulative_{\lambda^{(n)}}(U_{n,q},Z_{n,\xi})\geq
\Phi\left(\frac{\xi -f_0q}{\sqrt{\Force}}\right)-2\epsilon \]
converges to $1$ as
$n \to \infty$.
This demonstrates that the lower bound for the double cumulative function
$\doublecumulative_{\lambda^{(n)}}$ is asymptotically close to the standard
normal \CDF, within an error margin of
$2\epsilon$.

Proof of the opposite inequality follows similarly and is omitted for brevity.
This completes the proof of \cref{thm:gaussian-profile}.

\section{Proofs of the main results}

\label{sec:proofs-of-main-results}

In this short section, we present the long-awaited proofs of
\cref{thm:determinism-new,thm:CLT-trans}. For \cref{thm:determinism-new},
which addresses the asymptotic determinism of Schensted insertion, we will take
a detour through \cref{thm:determinism-old}. Although both theorems reach
the same conclusion, the assumptions of \cref{thm:determinism-old} are
framed in terms of Kerov's transition measure.

\subsection{The main idea}

Let $\lambda$ be a random Young diagram, and let $T$ be a uniformly random
Poissonized tableau of shape $\lambda$. Consider a constant $z \in [0, 1]$. The
central concept that underlies all proofs in this section is the relationship
between the tail distribution function of the $u$-coordinate of the newly
inserted box during the insertion $T \leftarrow z$ and the expected value of the
double cumulative function. Specifically, we have:
\begin{equation}
	\label{eq:good-tail}
	 \Pro\left[ \uIns\left( T; z\right) \geq u \right] =  
	    \E \doublecumulative_{\lambda}(u, z)
\end{equation}
for any $u\in\R$. 

It is important to note that when $\lambda$ is deterministic, the above equality
represents the definition of the double cumulative function. In the general
case, this relationship can be established by taking the appropriate average of
both sides of the equation.

\subsection{First-order asymptotics for Schensted insertion}

\subsubsection{Kerov transition measure formulation}

We begin our analysis with a refined version of \cref{thm:determinism-new},
formulated in the language of Kerov transition measure. This formulation not
only operates under weaker assumptions but also serves as a crucial first step
in proving \cref{thm:determinism-new}. 
\begin{theorem}
	\label{thm:determinism-old} 
	Let the following conditions hold:
	\begin{itemize}
		\item 
	For each $n\geq 1$ we are given a random Young diagram $\lambda^{(n)}$. 
	
	\item There is a probability measure
	$\nu$ on the real line such that for each $u\in\R$ which is a continuity point of 
	its \CDF, the limit
	\[   K_{\lambda^{(n)}}\left( \sqrt{n} \ u \right) 
	\xrightarrow[n\to\infty]{P} F_\nu(u) \]
	holds true in probability.
	
	\item For a given real number $0 < z < 1$, the quantile function $F_\nu^{-1}$
is continuous at $z$. 
	\end{itemize}

		Let $T^{(\nklatki)}$ be a uniformly random Poissonized tableau of shape
	$\lambda^{(\nklatki)}$. Then
	\begin{equation}
	\label{eq:insertion-lln}	
	\frac{1}{\sqrt{n}} \uIns\left( T^{(n)}; z\right)  \xrightarrow[n\to\infty]{P}
	F_\nu^{-1}(z).
	\end{equation}
\end{theorem}
\begin{proof}
Our strategy hinges on demonstrating that the limit of the CDF for the random
variable on the left-hand side of \eqref{eq:insertion-lln} converges to a step
function. Specifically, we aim to prove that for any $u \neq u_0$:
	\[ \lim_{n\to\infty} \Pro\left[ \uIns(T^{(n)}; z) < u  \right]  = 
	\begin{cases}
		0 & \text{if } u< u_0, \\
		1 & \text{if } u> u_0.
	\end{cases}
	\]

We begin by considering the case where $u < u_0$. By monotonicity, we can assume
without loss of generality that $u$ is a continuity point of $F_\nu$. Applying
\eqref{eq:good-tail}, we obtain:
	\[ \Pro\left[ \uIns(T^{(n)}; z) < u  \right]  = 1 -\E \doublecumulative_{\lambda^{(n)}}(u, z) 
	\xrightarrow[n\to\infty]{} 0.\]
This convergence to zero is a direct consequence of the second part of \cref{lem:convergence-easy}.

The case where $u > u_0$ follows an analogous reasoning process, completing our
proof of convergence to the step function.
\end{proof}

\subsubsection{Proof of \cref{thm:determinism-new}}
\label{sec:proof-of-determinism}

\begin{proof}[Proof of \cref{thm:determinism-new}]
The first three assumptions of \cref{thm:determinism-new} imply that the random
sequence of continuous diagrams $\omega_n$ converges to $\Omega$ in probability,
as defined by the topology introduced by Kerov (see \cref{prop:homeomorphism}
and the associated references). Consequently, \cref{prop:homeomorphism} becomes
applicable, ensuring that the conditions of \cref{thm:determinism-old} are
satisfied with $\nu := \mu_\Omega$.
\end{proof}

\subsection{Gaussianity of insertion fluctuations: the proof}
\label{sec:proof-thm-CLT-general}

\begin{proof}[Proof of \cref{thm:CLT-trans}]
By applying \cref{thm:gaussian-profile} for the special choice of the
vertical coordinate $\xi=0$, we can demonstrate that the \CDF of the random
variable in question converges to that of a centered normal distribution with
variance $\frac{\Force}{f_0^2}$ as $n \to \infty$. Specifically, for any
$q\in\R$ by Equation \eqref{eq:good-tail}:
\begin{multline*} 
	\Pro\left[	\sqrt[4]{\nklatki} \left[ \frac{\uIns( T^{(\nklatki)}; z_0 )}{\sqrt{\nklatki}} -u_0 \right] <q  \right] = \\ 
1- \E \left[ \doublecumulative_{\lambda^{(n)}}\big( U_{n,q}, Z_{n,0} \big) \right]
\xrightarrow[n\to\infty]{} 1-	\Phi\left(  \frac{- f_0 q}{\sqrt{\Force}} \right)= 
\Phi\left(  \frac{f_0 q}{\sqrt{\Force}} \right).\qedhere
\end{multline*}
\end{proof}

\begin{acks}[Acknowledgments]
We express our gratitude to Márton Balázs, Gaetan Borot, Marek Bożejko, Maciej
Dołęga, Valentin F\'eray, 
Pablo Ferrari, Patrik Ferrari, Mustazee Rahman, and Dan Romik for their
valuable discussions and suggestions regarding the bibliography.

Additionally, we thank Maciej Hendzel for conducting the Monte Carlo experiments
related to \cref{conj:CLT-Plancherel}.
\end{acks}
\begin{funding}
Research was supported by Narodowe Centrum Nauki, grant number \linebreak[4]
2017/26/A/ST1/00189. 

Additionally, the first named author was supported by Narodowe Centrum
Bada\'n i Rozwoju, grant number POWR.03.05.00-00-Z302/17-00.
\end{funding}

\bibliographystyle{imsart-number} %
\bibliography{biblio}       %

\def\cprime{$'$}
\begin{thebibliography}{36}

\bibitem{BaikDeift1999}
\begin{barticle}[author]
\bauthor{\bsnm{Baik},~\bfnm{J.}\binits{J.}},
  \bauthor{\bsnm{Deift},~\bfnm{P.}\binits{P.}} \AND
  \bauthor{\bsnm{Johansson},~\bfnm{K.}\binits{K.}}
(\byear{1999}).
\btitle{On the distribution of the length of the longest increasing subsequence
  of random permutations}.
\bjournal{J. Amer. Math. Soc.}
\bvolume{12}
\bpages{1119--1178}.
\end{barticle}
\endbibitem

\bibitem{BenArous-Guionnet}
\begin{barticle}[author]
\bauthor{\bsnm{Ben~Arous},~\bfnm{G.}\binits{G.}} \AND
  \bauthor{\bsnm{Guionnet},~\bfnm{A.}\binits{A.}}
(\byear{1997}).
\btitle{Large deviations for {Wigner}'s law and {Voiculescu}'s non-commutative
  entropy}.
\bjournal{Probab. Theory Relat. Fields}
\bvolume{108}
\bpages{517--542}.
\bdoi{10.1007/s004400050119}
\end{barticle}
\endbibitem

\bibitem{Biane1998}
\begin{barticle}[author]
\bauthor{\bsnm{Biane},~\bfnm{P.}\binits{P.}}
(\byear{1998}).
\btitle{Representations of symmetric groups and free probability}.
\bjournal{Adv. Math.}
\bvolume{138}
\bpages{126--181}.
\bdoi{10.1006/aima.1998.1745}
\bmrnumber{1644993 (2001b:05225)}
\end{barticle}
\endbibitem

\bibitem{Biane2001}
\begin{barticle}[author]
\bauthor{\bsnm{Biane},~\bfnm{P.}\binits{P.}}
(\byear{2001}).
\btitle{Approximate factorization and concentration for characters of symmetric
  groups}.
\bjournal{Internat. Math. Res. Notices}
\bvolume{2001}
\bpages{179--192}.
\bdoi{10.1155/S1073792801000113}
\bmrnumber{1813797 (2002a:20017)}
\end{barticle}
\endbibitem

\bibitem{BogachevSu2007}
\begin{barticle}[author]
\bauthor{\bsnm{Bogachev},~\bfnm{Leonid~V.}\binits{L.~V.}} \AND
  \bauthor{\bsnm{Su},~\bfnm{Zhonggen}\binits{Z.}}
(\byear{2007}).
\btitle{Gaussian fluctuations of {Y}oung diagrams under the {P}lancherel
  measure}.
\bjournal{Proc. R. Soc. Lond. Ser. A Math. Phys. Eng. Sci.}
\bvolume{463}
\bpages{1069--1080}.
\bdoi{10.1098/rspa.2006.1808}
\bmrnumber{2310137}
\end{barticle}
\endbibitem

\bibitem{DFS}
\begin{barticle}[author]
\bauthor{\bsnm{Do{\l}{\k{e}}ga},~\bfnm{Maciej}\binits{M.}},
  \bauthor{\bsnm{F{\'e}ray},~\bfnm{Valentin}\binits{V.}} \AND
  \bauthor{\bsnm{{\'S}niady},~\bfnm{Piotr}\binits{P.}}
(\byear{2010}).
\btitle{Explicit combinatorial interpretation of {Kerov} character polynomials
  as numbers of permutation factorizations}.
\bjournal{Adv. Math.}
\bvolume{225}
\bpages{81--120}.
\bdoi{10.1016/j.aim.2010.02.011}
\end{barticle}
\endbibitem

\bibitem{FerrariMartin2009}
\begin{barticle}[author]
\bauthor{\bsnm{Ferrari},~\bfnm{P.~A.}\binits{P.~A.}},
  \bauthor{\bsnm{Martin},~\bfnm{J.~B.}\binits{J.~B.}} \AND
  \bauthor{\bsnm{Pimentel},~\bfnm{L.~P.~R.}\binits{L.~P.~R.}}
(\byear{2009}).
\btitle{A phase transition for competition interfaces}.
\bjournal{Ann. Appl. Probab.}
\bvolume{19}
\bpages{281--317}.
\bdoi{10.1214/08-AAP542}
\bmrnumber{2498679 (2010h:60264)}
\end{barticle}
\endbibitem

\bibitem{Fulton1997}
\begin{bbook}[author]
\bauthor{\bsnm{Fulton},~\bfnm{W.}\binits{W.}}
(\byear{1997}).
\btitle{Young Tableaux: With Applications to Representation theory and
  Geometry}.
\bseries{London Mathematical Society Student Texts}
\bvolume{35}.
\bpublisher{Cambridge University Press}, \baddress{Cambridge}.
\bmrnumber{1464693 (99f:05119)}
\end{bbook}
\endbibitem

\bibitem{GorinRahman2019}
\begin{barticle}[author]
\bauthor{\bsnm{Gorin},~\bfnm{Vadim}\binits{V.}} \AND
  \bauthor{\bsnm{Rahman},~\bfnm{Mustazee}\binits{M.}}
(\byear{2019}).
\btitle{Random sorting networks: local statistics via random matrix laws}.
\bjournal{Probab. Theory Related Fields}
\bvolume{175}
\bpages{45--96}.
\bdoi{10.1007/s00440-018-0886-1}
\bmrnumber{4009705}
\end{barticle}
\endbibitem

\bibitem{Gustavsson}
\begin{barticle}[author]
\bauthor{\bsnm{Gustavsson},~\bfnm{Jonas}\binits{J.}}
(\byear{2005}).
\btitle{Gaussian fluctuations of eigenvalues in the {GUE}}.
\bjournal{Ann. Inst. H. Poincar\'{e} Probab. Statist.}
\bvolume{41}
\bpages{151--178}.
\bdoi{10.1016/j.anihpb.2004.04.002}
\bmrnumber{2124079}
\end{barticle}
\endbibitem

\bibitem{Kendall}
\begin{bmisc}[author]
\bauthor{\bsnm{Kendall},~\bfnm{Maurice}\binits{M.}} \AND
  \bauthor{\bsnm{Stuart},~\bfnm{Alan}\binits{A.}}
(\byear{1979}).
\btitle{The advanced theory of statistics. ({In} three volumes) {Vol}. 2:
  {Inference} and relationship. 4th ed}.
\bhowpublished{London, {High} {Wycombe}: {Charles} {Griffin} \& {Company}
  {Ltd}. {X}, 748 p. {{\textsterling}} 23.50 (1979).}
\end{bmisc}
\endbibitem

\bibitem{Kerov1993-transition}
\begin{barticle}[author]
\bauthor{\bsnm{Kerov},~\bfnm{S.}\binits{S.}}
(\byear{1993}).
\btitle{Transition probabilities for continual {Y}oung diagrams and the
  {M}arkov moment problem.}
\bjournal{Funct. Anal. Appl.}
\bvolume{27}
\bpages{104--117}.
\end{barticle}
\endbibitem

\bibitem{Kerov1993-gaussian}
\begin{barticle}[author]
\bauthor{\bsnm{Kerov},~\bfnm{Serguei}\binits{S.}}
(\byear{1993}).
\btitle{Gaussian limit for the {P}lancherel measure of the symmetric group}.
\bjournal{C. R. Acad. Sci. Paris S\'{e}r. I Math.}
\bvolume{316}
\bpages{303--308}.
\bmrnumber{1204294}
\end{barticle}
\endbibitem

\bibitem{KerovBook}
\begin{bbook}[author]
\bauthor{\bsnm{Kerov},~\bfnm{S.~V.}\binits{S.~V.}}
(\byear{2003}).
\btitle{Asymptotic representation theory of the symmetric group and its
  applications in analysis}.
\bseries{Translations of Mathematical Monographs}
\bvolume{219}.
\bpublisher{American Mathematical Society, Providence, RI}
\bnote{Translated from the Russian manuscript by N. V. Tsilevich, With a
  foreword by A. Vershik and comments by G. Olshanski}.
\bdoi{10.1090/mmono/219}
\bmrnumber{1984868}
\end{bbook}
\endbibitem

\bibitem{KerovVershik1986}
\begin{barticle}[author]
\bauthor{\bsnm{Kerov},~\bfnm{S.~V.}\binits{S.~V.}} \AND
  \bauthor{\bsnm{Vershik},~\bfnm{A.~M.}\binits{A.~M.}}
(\byear{1986}).
\btitle{The characters of the infinite symmetric group and probability
  properties of the {R}obinson-{S}chensted-{K}nuth algorithm}.
\bjournal{SIAM J. on Algebraic and Discrete Methods}
\bvolume{7}
\bpages{116--124}.
\bdoi{10.1137/0607014}
\bmrnumber{819713 (87e:22014)}
\end{barticle}
\endbibitem

\bibitem{Lauritzen2002}
\begin{bbook}[author]
\bauthor{\bsnm{Lauritzen},~\bfnm{Steffen~L.}\binits{S.~L.}} \AND
  \bauthor{\bsnm{Hald},~\bfnm{A.}\binits{A.}}
(\byear{2002}).
\btitle{Thiele: pioneer in statistics}.
\bpublisher{Oxford University Press, New York}
\bnote{Thiele's papers translated from the Danish by Steffen L. Lauritzen, With
  appreciations of Thiele's work by Lauritzen and A. Hald}.
\bdoi{10.1093/acprof:oso/9780198509721.001.0001}
\bmrnumber{2055773}
\end{bbook}
\endbibitem

\bibitem{LoganShepp1977}
\begin{barticle}[author]
\bauthor{\bsnm{Logan},~\bfnm{B.~F.}\binits{B.~F.}} \AND
  \bauthor{\bsnm{Shepp},~\bfnm{L.~A.}\binits{L.~A.}}
(\byear{1977}).
\btitle{A variational problem for random {Y}oung tableaux}.
\bjournal{Advances in Math.}
\bvolume{26}
\bpages{206--222}.
\bmrnumber{1417317 (98e:05108)}
\end{barticle}
\endbibitem

\bibitem{MarciniakMaslankaSniady-bumping}
\begin{barticle}[author]
\bauthor{\bsnm{Marciniak},~\bfnm{Miko{\l}aj}\binits{M.}},
  \bauthor{\bsnm{Ma{\'s}lanka},~\bfnm{{\L}ukasz}\binits{{\L}.}} \AND
  \bauthor{\bsnm{{\'S}niady},~\bfnm{Piotr}\binits{P.}}
(\byear{2021}).
\btitle{Poisson limit of bumping routes in the {Robinson}--{Schensted}
  correspondence}.
\bjournal{Probab. Theory Relat. Fields}
\bvolume{181}
\bpages{1053--1103}.
\bdoi{10.1007/s00440-021-01084-y}
\end{barticle}
\endbibitem

\bibitem{MarciniakSniadyAlternating}
\begin{bmisc}[author]
\bauthor{\bsnm{Marciniak},~\bfnm{Miko{\l}aj}\binits{M.}} \AND
  \bauthor{\bsnm{{\'S}niady},~\bfnm{Piotr}\binits{P.}}
(\byear{2024}).
\btitle{Cumulants of threshold for Schensted row insertion into random
  tableaux}.
\bnote{Preprint arXiv:2407.06213}.
\bdoi{10.48550/arXiv.2407.06213}
\end{bmisc}
\endbibitem

\bibitem{MingoSpeicher2017}
\begin{bbook}[author]
\bauthor{\bsnm{Mingo},~\bfnm{James~A.}\binits{J.~A.}} \AND
  \bauthor{\bsnm{Speicher},~\bfnm{Roland}\binits{R.}}
(\byear{2017}).
\btitle{Free probability and random matrices}.
\bseries{Fields Institute Monographs}
\bvolume{35}.
\bpublisher{Springer, New York; Fields Institute for Research in Mathematical
  Sciences, Toronto}.
\bdoi{10.1007/978-1-4939-6942-5}
\bmrnumber{3585560}
\end{bbook}
\endbibitem

\bibitem{Okounkov2000}
\begin{barticle}[author]
\bauthor{\bsnm{Okounkov},~\bfnm{Andrei}\binits{A.}}
(\byear{2000}).
\btitle{Random matrices and random permutations}.
\bjournal{Internat. Math. Res. Notices}
\bvolume{20}
\bpages{1043--1095}.
\bdoi{10.1155/S1073792800000532}
\bmrnumber{1802530}
\end{barticle}
\endbibitem

\bibitem{PittelRomik2007}
\begin{barticle}[author]
\bauthor{\bsnm{Pittel},~\bfnm{Boris}\binits{B.}} \AND
  \bauthor{\bsnm{Romik},~\bfnm{Dan}\binits{D.}}
(\byear{2007}).
\btitle{Limit shapes for random square {Y}oung tableaux}.
\bjournal{Adv. in Appl. Math.}
\bvolume{38}
\bpages{164--209}.
\bdoi{10.1016/j.aam.2005.12.005}
\bmrnumber{2290809}
\end{barticle}
\endbibitem

\bibitem{RahmanVirag2021}
\begin{bmisc}[author]
\bauthor{\bsnm{Rahman},~\bfnm{Mustazee}\binits{M.}} \AND
  \bauthor{\bsnm{Virag},~\bfnm{Balint}\binits{B.}}
(\byear{2021}).
\btitle{Infinite geodesics, competition interfaces and the second class
  particle in the scaling limit}.
\bnote{preprint arXiv:2112.06849}.
\bdoi{10.48550/ARXIV.2112.06849}
\end{bmisc}
\endbibitem

\bibitem{Romik2015}
\begin{bbook}[author]
\bauthor{\bsnm{Romik},~\bfnm{Dan}\binits{D.}}
(\byear{2015}).
\btitle{The surprising mathematics of longest increasing subsequences}.
\bseries{Institute of Mathematical Statistics Textbooks}
\bvolume{4}.
\bpublisher{Cambridge University Press, New York}.
\bmrnumber{3468738}
\end{bbook}
\endbibitem

\bibitem{RomikSniady2015}
\begin{barticle}[author]
\bauthor{\bsnm{Romik},~\bfnm{Dan}\binits{D.}} \AND
  \bauthor{\bsnm{{\'S}niady},~\bfnm{Piotr}\binits{P.}}
(\byear{2015}).
\btitle{Jeu de taquin dynamics on infinite {Y}oung tableaux and second class
  particles}.
\bjournal{Ann. Probab.}
\bvolume{43}
\bpages{682--737}.
\bdoi{10.1214/13-AOP873}
\bmrnumber{3306003}
\end{barticle}
\endbibitem

\bibitem{Rost1981}
\begin{barticle}[author]
\bauthor{\bsnm{Rost},~\bfnm{H.}\binits{H.}}
(\byear{1981}).
\btitle{Nonequilibrium behaviour of a many particle process: density profile
  and local equilibria}.
\bjournal{Z. Wahrsch. Verw. Gebiete}
\bvolume{58}
\bpages{41--53}.
\bdoi{10.1007/BF00536194}
\bmrnumber{635270 (83a:60176)}
\end{barticle}
\endbibitem

\bibitem{Sniady2006a}
\begin{barticle}[author]
\bauthor{\bsnm{{\'S}niady},~\bfnm{P.}\binits{P.}}
(\byear{2006}).
\btitle{Gaussian fluctuations of characters of symmetric groups and of {Y}oung
  diagrams}.
\bjournal{Probab. Theory and Rel. Fields}
\bvolume{136}
\bpages{263--297}.
\bdoi{10.1007/s00440-005-0483-y}
\bmrnumber{2240789 (2007d:20020)}
\end{barticle}
\endbibitem

\bibitem{Sniady2024modulus}
\begin{barticle}[author]
\bauthor{\bsnm{{\'S}niady},~\bfnm{Piotr}\binits{P.}}
(\byear{2025}).
\btitle{Modulus of continuity of {Kerov} transition measure for continual
  {Young} diagrams}.
\bjournal{Electron. J. Probab.}
\bvolume{30}
\bpages{26}.
\bnote{Id/No 108}.
\bdoi{10.1214/25-EJP1369}
\end{barticle}
\endbibitem

\bibitem{Stanley1999}
\begin{bbook}[author]
\bauthor{\bsnm{Stanley},~\bfnm{R.~P.}\binits{R.~P.}}
(\byear{1999}).
\btitle{Enumerative combinatorics. {V}ol. 2}.
\bseries{Cambridge Studies in Advanced Mathematics}
\bvolume{62}.
\bpublisher{Cambridge University Press}, \baddress{Cambridge}.
\bdoi{10.1017/CBO9780511609589}
\bmrnumber{1676282 (2000k:05026)}
\end{bbook}
\endbibitem

\bibitem{Thoma1964}
\begin{barticle}[author]
\bauthor{\bsnm{Thoma},~\bfnm{Elmar}\binits{E.}}
(\byear{1964}).
\btitle{Die unzerlegbaren, positiv-definiten {K}lassenfunktionen der
  abz\"{a}hlbar unendlichen, symmetrischen {G}ruppe}.
\bjournal{Math. Z.}
\bvolume{85}
\bpages{40--61}.
\bdoi{10.1007/BF01114877}
\bmrnumber{173169}
\end{barticle}
\endbibitem

\bibitem{TracyWidom}
\begin{barticle}[author]
\bauthor{\bsnm{Tracy},~\bfnm{Craig~A.}\binits{C.~A.}} \AND
  \bauthor{\bsnm{Widom},~\bfnm{Harold}\binits{H.}}
(\byear{1994}).
\btitle{Level-spacing distributions and the {Airy} kernel}.
\bjournal{Commun. Math. Phys.}
\bvolume{159}
\bpages{151--174}.
\bdoi{10.1007/BF02100489}
\end{barticle}
\endbibitem

\bibitem{VassilievDuzhin2019}
\begin{barticle}[author]
\bauthor{\bsnm{Vassiliev},~\bfnm{N.~N.}\binits{N.~N.}},
  \bauthor{\bsnm{Duzhin},~\bfnm{V.~S.}\binits{V.~S.}} \AND
  \bauthor{\bsnm{Kuzmin},~\bfnm{A.~D.}\binits{A.~D.}}
(\byear{2019}).
\btitle{Investigation of properties of equivalence classes of permutations by
  inverse {R}obinson--{S}chensted --{K}nuth transformation}.
\bjournal{Information and Control Systems}
\bvolume{98}
\bpages{11--22}.
\bdoi{10.31799/1684-8853-2019-1-11-22}
\end{barticle}
\endbibitem

\bibitem{VershikKerov1977}
\begin{barticle}[author]
\bauthor{\bsnm{Vershik},~\bfnm{A.~M.}\binits{A.~M.}} \AND
  \bauthor{\bsnm{Kerov},~\bfnm{S.~V.}\binits{S.~V.}}
(\byear{1977}).
\btitle{Asymptotics of the {P}lancherel measure of the symmetric group and the
  limit form of {Y}oung tableaux}.
\bjournal{Soviet Math. Dokl.}
\bvolume{18}
\bpages{527--531}.
\end{barticle}
\endbibitem

\bibitem{VershikKerov1981a}
\begin{barticle}[author]
\bauthor{\bsnm{Vershik},~\bfnm{A.~M.}\binits{A.~M.}} \AND
  \bauthor{\bsnm{Kerov},~\bfnm{S.~V.}\binits{S.~V.}}
(\byear{1981}).
\btitle{Asymptotic theory of the characters of a symmetric group}.
\bjournal{Funktsional. Anal. i Prilozhen.}
\bvolume{15}
\bpages{15--27, 96}.
\bmrnumber{639197 (84a:22016)}
\end{barticle}
\endbibitem

\bibitem{Voiculescu1994}
\begin{barticle}[author]
\bauthor{\bsnm{Voiculescu},~\bfnm{Dan}\binits{D.}}
(\byear{1994}).
\btitle{The analogues of entropy and of {F}isher's information measure in free
  probability theory. {II}}.
\bjournal{Invent. Math.}
\bvolume{118}
\bpages{411--440}.
\bdoi{10.1007/BF01231539}
\bmrnumber{1296352}
\end{barticle}
\endbibitem

\bibitem{Wojtyniak2019}
\begin{bmastersthesis}[author]
\bauthor{\bsnm{Wojtyniak},~\bfnm{Karolina}\binits{K.}}
(\byear{2019}).
\btitle{Asymptotyka losowych trajektorii jeu de taquin},
\btype{Bachelor's Thesis},
\bpublisher{Nicolaus Copernicus University in Toru\'n}.
\end{bmastersthesis}
\endbibitem

\end{thebibliography}

\end{document}